\documentclass[10pt]{amsart}
\usepackage{amssymb, enumitem}
\usepackage[all]{xy}
\usepackage{hyperref, aliascnt}
\usepackage{mathtools}
\usepackage{bbm}
\def\today{\number\day\space\ifcase\month\or   January\or February\or
   March\or April\or May\or June\or   July\or August\or September\or
   October\or November\or December\fi\   \number\year}

\theoremstyle{definition}
\newtheorem{lma}{Lemma}[section]

\newaliascnt{thmCt}{lma}
\newtheorem{thm}[thmCt]{Theorem}
\aliascntresetthe{thmCt}

\newaliascnt{corCt}{lma}
\newtheorem{cor}[corCt]{Corollary}
\aliascntresetthe{corCt}

\newaliascnt{propCt}{lma}
\newtheorem{prop}[propCt]{Proposition}
\aliascntresetthe{propCt}

\newtheorem*{thm*}{Theorem}
\newtheorem*{cor*}{Corollary}
\newtheorem*{prop*}{Proposition}

\newaliascnt{pgrCt}{lma}

\aliascntresetthe{pgrCt}

\newaliascnt{dfCt}{lma}
\newtheorem{df}[dfCt]{Definition}
\aliascntresetthe{dfCt}

\newaliascnt{remCt}{lma}
\newtheorem{rem}[remCt]{Remark}
\aliascntresetthe{remCt}

\newaliascnt{remsCt}{lma}

\aliascntresetthe{remsCt}

\newaliascnt{egCt}{lma}
\newtheorem{eg}[egCt]{Example}
\aliascntresetthe{egCt}

\newaliascnt{egsCt}{lma}

\aliascntresetthe{egsCt}

\newaliascnt{qstCt}{lma}

\aliascntresetthe{qstCt}

\newaliascnt{pbmCt}{lma}

\aliascntresetthe{pbmCt}

\newaliascnt{notaCt}{lma}

\aliascntresetthe{notaCt}

\newaliascnt{cnjCt}{lma}

\aliascntresetthe{cnjCt}

\newcommand{\beq}{\begin{equation}}
\newcommand{\eeq}{\end{equation}}
\newcommand{\beqa}{\begin{eqnarray*}}
\newcommand{\eeqa}{\end{eqnarray*}}
\newcommand{\bal}{\begin{align*}}
\newcommand{\eal}{\end{align*}}
\newcommand{\bi}{\begin{itemize}}
\newcommand{\ei}{\end{itemize}}
\newcommand{\be}{\begin{enumerate}}
\newcommand{\ee}{\end{enumerate}}

\newcommand{\ep}{\varepsilon}

\newcommand{\Z}{{\mathbb{Z}}}
\newcommand{\R}{{\mathbb{R}}}
\newcommand{\C}{{\mathbb{C}}}
\newcommand{\N}{{\mathbb{N}}}

\newcommand{\Hi}{{\mathcal{H}}}
\newcommand{\K}{{\mathcal{K}}}

\newcommand{\B}{{\mathcal{B}}}
\newcommand{\U}{{\mathcal{U}}}

\newcommand{\T}{{\mathbb{T}}}

\newcommand{\id}{{\mathrm{id}}}

\newcommand{\supp}{{\mathrm{supp}}}

\newcommand{\Aut}{{\mathrm{Aut}}}

\newcommand{\Ad}{{\mathrm{Ad}}}

\newcommand{\Cu}{{\mathrm{Cu}}}

\newcommand{\CCu}{{\mathbf{Cu}}}

\newcommand{\ca}{$C^*$-algebra}

\newcommand{\uca}{unital $C^*$-algebra}

\newcommand{\I}{\infty}

\title[]{The equivariant Cuntz semigroup}

\author{Eusebio Gardella and Luis Santiago}

\date{\today}

\address{Westf\"alische Wilhelms-Universit\"at M\"unster, Fachbereich Mathematik, ¨
Einsteinstrasse 62, 48149 M\"unster, Germany }

\email[]{gardella@uni-muenster.de}

\address{Institute of Mathematics, University of Aberdeen, Fraser Noble Building, Aberdeen
AB24 3UE, UK}

\email[]{lmoreno@abdn.ac.uk}
\urladdr{http://homepages.abdn.ac.uk/lmoreno/pages/index.html}


\begin{document}

\begin{abstract}
We introduce an equivariant version of the Cuntz semigroup, that takes an action of a compact group
into account. The equivariant Cuntz semigroup is naturally a semimodule over the representation semiring
of the given group. Moreover, this semimodule satisfies a number of additional structural properties.
We show that the equivariant Cuntz semigroup, as a functor, is continuous and stable. Moreover, cocycle
conjugate actions have isomorphic associated equivariant Cuntz semigroups.
One of our main results is an analog of Julg's theorem: the equivariant Cuntz semigroup is canonically
isomorphic to the Cuntz semigroup of the crossed product. We compute the induced semimodule structure on the
crossed product, which in the abelian case is given by the dual action. As an application of our results,
we show that freeness of a compact Lie group action on a compact Hausdorff space can be characterized in
terms of a canonically defined map into the equivariant Cuntz semigroup, extending results of Atiyah and
Segal for equivariant $K$-theory. Finally, we use the equivariant Cuntz semigroup to classify locally
representable actions on direct limits of one-dimensional NCCW-complexes, generalizing work of
Handelman and Rossmann.
\end{abstract}

\maketitle

\tableofcontents

\section{Introduction}

Equivariant $K$-theory for compact groups acting on topological spaces was introduced by Atiyah (the paper
\cite{Seg_EqKThy}, by Segal, contains a basic treatment of the theory). One of the first applications of
this theory was a striking characterization of freeness of a compact Lie group action (\cite{AtiSeg_completion};
see also Theorem~1.1.1 in~\cite{Phi_Book}). Equivariant $K$-theory was later defined and studied for actions of
compact groups on noncommutative \ca s. A fundamental result in this area is Julg's identification
(\cite{Jul_EqKthy}) of the equivariant $K$-theory of a given action, with the ordinary $K$-theory of its
associated crossed product. In a different direction, each of the different statements in Atiyah-Segal's
characterization of freeness, interpreted in the context of \ca s, can be taken as possible definitions of
``noncommutative freeness". This is the approach taken by Phillips in \cite{Phi_Book}. Equivariant $K$-theory
has also been used as an invariant for compact group actions (\cite{HanRos_CptAF}, \cite{Gar_Kir1}),
and its definition has been extended to actions of more general objects, such as quantum groups
(\cite{ZhaZha_EqVBquantum}). Equivariant $K$-theory also plays a role in the so called Baum-Connes
conjecture with coefficients (where one considers the equivariant $K$-theory of a \ca\ other than $\C$.)

On the other hand, the Cuntz semigroup $\Cu(A)$ of a \ca\ $A$, first considered by Cuntz in the 70's (\cite{Cun_DimFun}),
has been intensively studied in the last decade since Toms successfully used it (\cite{Tom_classifpbm}) to
distinguish two non-isomorphic \ca s with identical Elliott invariant (as well as identical real and stable
ranks). Coward, Elliott and Ivanescu (\cite{CowEllIva}) suggested that the Cuntz semigroup could be used
as an invariant for \ca s (in many cases, finer than $K_0$), and this semigroup has since then been used to
obtain positive classification results of not necessarily simple \ca s. The most general result in this direction
is due to Robert (\cite{Rob_classif}): he showed that the Cuntz semigroup is a complete invariant for (not
necessarily simple) unital direct limits of 1-dimensional noncommutative CW-complexes with trivial $K_1$. (This class
contains all AI-algebras.) We refer the reader to \cite{BroPerTom} and \cite{AntPerThi} for thorough developments
of the theory of the Cuntz semigroup.

In this paper, we study an equivariant version of the Cuntz semigroup for compact group actions on \ca s.
For an action $\alpha\colon G\to \Aut(A)$ of a compact group $G$ on a \ca\ $A$, we denote its equivariant
Cuntz semigroup by $\Cu^G(A,\alpha)$.
This is a partially ordered semigroup, with a natural semimodule structure over the representation
semiring $\Cu(G)$ of $G$. We explore some basic properties of the functor $(A,\alpha)\mapsto \Cu^G(A,\alpha)$,
such as continuity, stability, passage to full hereditary subalgebras, cocycle equivalence invariance, etc.
One of the main results of this work
(\autoref{Thm:JulgCuG}) is an analog of Julg's theorem for the Cuntz semigroup: $\Cu^G(A,\alpha)$ is naturally
isomorphic to $\Cu(A\rtimes_\alpha G)$. The induced $\Cu(G)$-semimodule structure on $\Cu(A\rtimes_\alpha G)$
is computed in \autoref{thm:JulgWithCuG} (see \autoref{prop: semimodule structure when G is abelian} for a
simpler description when $G$ is abelian). We prove an analog
of Atiyah-Segal's characterization of freeness using the Cuntz semigroup; see \autoref{thm:freeness}.
Applications to classification of actions are given in the last section; see
\autoref{thm: equiv morphisms of approx rep}.

There are a number of reasons to be interested in an equivariant version of the Cuntz semigroup. Its use for
classification of \ca s suggests that this should also be a useful invariant for group actions. As a semigroup,
$\Cu^G(A,\alpha)$ is just the Cuntz semigroup of the crossed product, but the additional $\Cu(G)$-semimodule
structure makes this a finer invariant. The fact that cocycle equivalent actions have isomorphic equivariant
Cuntz semigroups is a stronger statement than the fact that their crossed products have isomorphic Cuntz semigroups.
In a different direction, our study is a first step towards
a bivariant, equivariant version of the Cuntz semigroup, which is currently being developed by Tornetta; see \cite{Tor_CuCu}.
(A bivariant version of the Cuntz semigroup has been introduced in \cite{BTZ}.)

We have organized this paper as follows. In Section~2, and after reviewing the definition of the Cuntz
semigroup and the Cuntz category $\CCu$, we introduce the equivariant Cuntz semigroup using
positive invariant elements in suitable stabilizations of the algebra (\autoref{df:CuGAa}).
The main result of this section, \autoref{cor:CuGFunctorCCu}, asserts that the equivariant Cuntz semigroup
is a functor from the category of $G$-\ca s (that is, \ca s with an action of $G$), to the category $\CCu$.

In Section~3, we introduce the representation semiring $\Cu(G)$ of $G$ (\autoref{def: rep semiring}), which
corresponds to the equivariant Cuntz semigroup of $G$ acting on $\C$ (\autoref{thm: mu s_mu}), and
define a canonical $\Cu(G)$-action on $\Cu^G(A,\alpha)$; see \autoref{df:SemimodStructCuG}. With this
semimodule structure, the equivariant Cuntz semigroup becomes a functor from $G$-\ca s to a distinguished
category of $\Cu(G)$-se\-mi\-mo\-dules (see \autoref{df:CCuG} and \autoref{thm:EqCtzSmgpCCuG}). We finish this
section by showing that the functor $\Cu^G$ is stable (\autoref{prop:CuGStable}) and preserves countable
inductive limits (\autoref{prop:CuGContinuous}).

Section~4 is devoted to giving two pictures of the equivariant Cuntz semigroup using equivariant Hilbert modules.
In one of these pictures, we identify $\Cu^G(A,\alpha)$ with the ordinary Cuntz semigroup of $\K(\Hi_A)^G$, where
$\Hi_A$ is the universal equivariant Hilbert module for $(A,\alpha)$ introduced by Kasparov in \cite{Kas_HilbMods}
(see \autoref{df:H_A}). In the second picture, we identify $\Cu^G(A,\alpha)$ with the
$\Cu(G)$-semimodule obtained by taking a
suitable equivalence relation (\autoref{def: Hilbert subequivalence}) on the class of equivariant Hilbert
modules; see \autoref{cor:HilbModIsomObject}.

Section~5 contains some of our main results. In \autoref{Thm:JulgCuG}, we construct a natural $\Cu$-isomorphism
between $\Cu^G(A,\alpha)$ and $\Cu(A\rtimes_\alpha G)$. (This is a Cuntz semigroup analog of Julg's theorem
for $K$-theory; see \cite{Jul_EqKthy}.) The induced $\Cu(G)$-semimodule structure on $\Cu(A\rtimes_\alpha G)$
is computed in \autoref{thm:JulgWithCuG}, with an easier description available when $G$ is abelian; see
\autoref{prop: semimodule structure when G is abelian}. As an application, we show that
invariant full hereditary subalgebras have canonically isomorphic equivariant Cuntz semigroups
(\autoref{prop:HeredSubalg}).

Section~6 contains several computations of the equivariant Cuntz semigroups of a number of dynamical systems.
Most of our computations use either the Rokhlin property (which are handled with results from
\cite{GarSan_RokConstrI}) or are pullbacks of dynamical systems (which are handled with \autoref{thm: Lie}).

In Section~7, we apply the theory developed in the previous sections to prove a characterization
of freeness of a compact Lie group action on a compact Hausdorff space, in terms of a certain canonical
map into the equivariant Cuntz semigroup; see \autoref{thm:freeness}. This characterization resembles
(and depends on) Atiyah-Segal's characterization of freeness using equivariant $K$-theory
(\cite{AtiSeg_completion}).

Finally, in Section~8, we use
the equivariant Cuntz semigroup to classify certain direct limit actions on a class of stably finite
\ca s containing all AI-algebras; see \autoref{thm: equiv morphisms of approx rep}. Our results extend the
work of Handelman and Rossmann (\cite{HanRos_CptAF}) on locally representable actions on AF-algebras.

\vspace{0.3cm}

Throughout the paper, topological groups are always assumed to be Hausdorff.
If $G$ is a locally compact group, we will denote by $\widehat G$ the set of
unitary equivalence classes of irreducible representations of $G$. When $G$ is abelian, it is well-known
that every irreducible representation is one dimensional, and hence $\widehat{G}$ has a natural group
structure. When $G$ is compact (but not necessarily abelian), then every irreducible representation is
automatically finite dimensional. Finally, if $G$ is compact and abelian, then its dual group $\widehat{G}$
is discrete. A unitary representation $\mu\colon G\to \U(\Hi_\mu)$ of $G$ on a Hilbert space $\Hi_\mu$ will
usually be abbreviated to $(\Hi_\mu,\mu)$. We say that $(\Hi_\mu,\mu)$ is separable, or finite dimensional,
if $\Hi_\mu$ is. The unitary equivalence class of $(\Hi_\mu,\mu)$ is denoted by $[\mu]$.

For a locally compact group $G$, we denote by $\Cu(G)$ the class of unitary equivalence classes of
unitary representations of $G$ on separable Hilbert spaces. It is easy to see, fixing a separable Hilbert
space and restricting to representations on it, that $\Cu(G)$ is in fact a set. This set has important
additional structure that will not be discussed until it is needed in Section~3.
The set $\Cu(G)$ will play mostly a notational role in the first few sections.

We sometimes make a slight abuse of notation and do not distinguish between elements in $\widehat{G}$
(or $\Cu(G)$) and irreducible (separable) unitary representations of $G$.

If $A$ is a \ca, we denote by $\id_A\colon A\to A$ its identity map. For a unitary $u$ in its multiplier
algebra $M(A)$, we write $\Ad(u)$ for the automorphism of $A$ given by $\Ad(u)(a)=uau^*$ for $a\in A$.
Writing $\U(M(A))$ for the unitary group of $M(A)$, observe that the map $\Ad\colon \U(M(A))\to \Aut(A)$
is a group homomorphism.
For Hilbert spaces $\Hi_1,\Hi_2$, we write $\B(\Hi_1,\Hi_2)$ for the Banach space of bounded, linear
operators from $\Hi_1$ to $\Hi_2$.
We write $\N$ for $\{1,2,\ldots\}$, $\Z_{\geq 0}$ for $\{0,1,2,\ldots\}$, and
$\overline{\Z_{\geq 0}}$ for $\{0,1,2,\ldots,\infty \}$.

For a locally compact group $G$ and a \ca\ $A$, an \emph{action} of $G$ on $A$ will always mean
a group homomorphism $\alpha\colon G\to\Aut(A)$ satisfying the following continuity condition:
for $a\in A$, the map $G\to A$, given by $g\mapsto\alpha_g(a)$ for $g\in G$, is continuous. (This
is usually referred to as ``strong continuity".) We write $A^G$, or $A^\alpha$ if we need to stress
the action $\alpha$, for the $C^*$-subalgebra
\[A^G=\{a\in A\colon \alpha_g(a)=a \ \mbox{ for all } g\in G\}.\]

\indent \textbf{Acknowledgements.} The authors are grateful to Chris Phillips for valuable conversations on
equivariant $K$-theory.
The first named author would like to thank Tron Omland for helpful email correspondence
regarding coactions, and also Gabriele Tornetta for numerous discussions on the topic of this work, as
well as for pointing out that our previous version of \autoref{thm:ConjActsCu} had an unnecessary assumption.
Finally, we thank the anonymous referee for his comments and suggestions, which greatly improved the present
paper, and particularly made it more accessible to the reader.

Most of this work was done while the first named author was visiting the
Westf\"alische Wilhelms-Universit\"at M\"unster in
June-September of 2013, and while both authors were participating in the Thematic Program on Abstract Harmonic
Analysis, Banach and Operator Algebras, at the
Fields Institute for Research in Mathematical Sciences at the University of Toronto, in January-June 2014.
The authors wish to thank both Mathematics departments for their hospitality and for providing a stimulating research environment.

\section{The Equivariant Cuntz semigroup}

In this section, for a continuous action $\alpha\colon G\to\Aut(A)$ of a compact group $G$ on a
\ca\ $A$, we define its equivariant Cuntz semigroup $\Cu^G(A,\alpha)$ (\autoref{df:CuGAa}),
and explore some basic properties. The main result of this section, \autoref{cor:CuGFunctorCCu},
asserts that the equivariant Cuntz semigroup is a functor from the category of $G$-$C^*$-algebras
to the Cuntz category $\CCu$.

\subsection{The Cuntz semigroup and the category $\CCu$.}
This subsection is devoted to reviewing the construction of the ordinary (that is, non
equivariant) Cuntz semigroup, as well as the definition of the Cuntz category $\CCu$.

\vspace{0.3cm}

Let $A$ be a $C^*$-algebra and let $a, b\in A$ be positive elements. We say that $a$ is
\emph{Cuntz subequivalent} to $b$, written $a\precsim b$, if there exists a sequence
$(d_n)_{n\in\N}$ in $A$, such that $\lim\limits_{n\to\infty}\|d_nbd_n^*- a\|=0$.
We say that $a$ is \emph{Cuntz equivalent} to $b$, written $a\sim b$, if $a\precsim b$
and $b\precsim a$.

It is clear that $\precsim$ is a preorder on the set of positive elements of $A$, and
thus $\sim$ is and equivalence relation. We denote by $[a]$ the Cuntz equivalence class
of the element $a\in A_+$.

\begin{df}
The Cuntz semigroup of $A$, denoted by $\Cu(A)$, is defined as the set of equivalence classes
of positive elements of $A\otimes \K$. Addition in $\Cu(A)$ is given by
\[[a]+[b]=\left[\begin{pmatrix}
a & 0\\
0 & b
\end{pmatrix}\right],\]
where the positive element inside the brackets in the right hand side is identified with its
image in $A\otimes \K$ under any $\ast$-isomorphism of $M_2(A\otimes \K)$ with $A\otimes \K$ induced
by an $\ast$-isomorphism of $M_2(\K)$ and $\K$. The set $\Cu(A)$ becomes a partially ordered semigroup when
equipped with the partial order $[a]\le [b]$ if $a\precsim b$.\end{df}

It is easy to see that any $\ast$-homomorphism $\phi\colon A\to B$ induces an
order preserving map $\Cu(\phi)\colon \Cu(A)\to \Cu(B)$ defined by
$\Cu(\phi)([a])=[(\phi\otimes \mathrm{id}_\K)(a)]$ for $a\in (A\otimes\K)_+$.

It was shown in \cite[Theorem 1]{CowEllIva} that $\Cu$ is a functor from the category of
$C^*$-algebras to a subcategory of the category of ordered abelian semigroups. We now
proceed to define this category, which we denote by $\CCu$.

Let $S$ be an ordered semigroup and let $s,t\in S$. We say that $s$ is \emph{compactly
contained} in $t$, and denote this by $s\ll t$, if whenever $(t_n)_{n\in\N}$ is an increasing
sequence in $S$ such that $t\le \sup\limits_{n\in\N} t_n$, there exists $k\in \N$ such that $s\le t_k$.
A sequence $(s_n)_{n\in\N}$ in $S$ is said to be \emph{rapidly increasing}, if $s_n\ll s_{n+1}$ for all $n\in \N$.

\begin{df}\label{def: CCu}
An ordered abelian semigroup $S$ is an object in the category $\CCu$ if it has a zero element and it satisfies the following properties:
\begin{itemize}
\item[(O1)] Every increasing sequence in $S$ has a supremum;
\item[(O2)] For every $s\in S$, there exists a rapidly increasing sequence $(s_n)_{n\in\N}$ in $S$ such that $s=\sup\limits_{n\in\N} s_n$.
\item[(O3)] If $(s_n)_{n\in\N}$ and $(t_n)_{n\in\N}$ are increasing sequences in $S$, then
$$\sup\limits_{n\in\N} s_n +\sup\limits_{n\in\N} t_n=\sup\limits_{n\in\N} (s_n+t_n);$$
\item[(O4)] If $s_1,s_2,t_1,t_2\in S$ satisfy $s_1\ll t_1$ and $s_2\ll t_2$, then $s_1+s_2\ll t_1+t_2$.
\end{itemize}

Let $S$ and $T$ be semigroups in the category $\CCu$. An order preserving semigroup map $\varphi\colon S\to T$ is a morphism in the category $\CCu$ if it preserves the zero element and it satisfies the following properties:
\begin{itemize}
\item[(M1)] If $(s_n)_{n\in\N}$ is an increasing sequence in $S$, then
$$\varphi\left(\sup\limits_{n\in\N} s_n\right)=\sup\limits_{n\in\N} \varphi(s_n);$$
\item[(M2)] If $s, t\in S$ satisfy $s\ll t$, then $\varphi(s)\ll \varphi(t)$.
\end{itemize}
\end{df}

The following observation will be used repeatedly.

\begin{rem}
Let $M$ be a partially ordered semigroup with zero element, and let $S$ be a semigroup in $\CCu$.
Suppose that there exists a semigroup morphism $\varphi\colon M\to S$ (or $\varphi\colon S\to M$)
preserving the zero element and such that:
\be
\item $\varphi$ preserves the order, that is, $x\leq y$ in $M$ implies $\varphi(x)\leq \varphi(y)$ in $S$;
\item $\varphi$ is an order embedding, that is, $\varphi(x)\leq \varphi(y)$ in $S$ implies $x\leq y$ in $M$
(this implies that $\varphi$ is injective); and
\item $\varphi$ is surjective.\ee
Then $M$ belongs to $\CCu$, and $\varphi$ is a $\Cu$-isomorphism. In particular, $\varphi$ automatically
preserves suprema of increasing sequences and the compact containment relation.
\end{rem}

Let $A$ be a $C^*$-algebra and let $a\in A$. Then it can be checked that
$[(a-\varepsilon)_+]\ll [a]$ for all $\varepsilon>0$, and that
$[a]=\sup\limits_{\varepsilon>0} [(a-\varepsilon)_+]$. In particular, it follows that
$\Cu(A)$ satisfies Axiom (O2).

\subsection{The equivariant Cuntz semigroup}
For the rest of this section, we fix a compact group $G$, a \ca\ $A$, and an action
$\alpha\colon G\to\Aut(A)$.

Let $(\Hi_\mu,\mu)$ and $(\Hi_\nu,\nu)$ be separable unitary representations of $G$. We endow
the Banach space $\B(\Hi_\mu,\Hi_\nu)$ with the $G$-action given by
\[g\cdot T=\nu_g\circ T\circ \mu_{g^{-1}},\]
for $g\in G$ and $T\in \B(\Hi_\mu,\Hi_\nu)$. It is clear that $\K(\Hi_\mu,\Hi_\nu)$ is an invariant,
closed subspace, which we will endow with the restricted $G$-action. With these actions, a $G$-invariant
linear map is precisely a map $\Hi_\mu\to \Hi_\nu$ that is equivariant with respect to $\mu$ and $\nu$.

Let $\Hi_1$ and $\Hi_2$ be Hilbert spaces, and let $A$ be a \ca. We will denote by
$\K(\Hi_1,\Hi_2)\otimes A$ the Banach subspace of the \ca\ $\K(\Hi_1\oplus \Hi_2)\otimes A$
of operators of the form
\[\left(
    \begin{array}{cc}
      0 & 0 \\
      \ast & 0 \\
    \end{array}
  \right).\]

With this convention, it is easy to check that for $x\in \K(\Hi_1,\Hi_2)\otimes A$ and
$y\in \K(\Hi_2,\Hi_3)\otimes A$, then the product (composition) $xy$ belongs to
$\K(\Hi_1,\Hi_3)\otimes A$. (We point out that what we denote by $\K(\Hi_1,\Hi_2)\otimes A$
can be canonically identified with $\K(\Hi_1\otimes A,\Hi_2\otimes A)$, where $\Hi_1\otimes A$
and $\Hi_2\otimes A$ denote the exterior tensor products of the Hilbert modules. See, for
example, \cite{Lan_Book}.)

\begin{rem}
Let $G$ be a locally compact group.
If $\mu\colon G\to\U(\Hi_\mu)$ and $\nu\colon G\to \U(\Hi_\nu)$ are unitary representations, then
$\delta=\mu\oplus \nu$ is a unitary representation on $\Hi=\Hi_\mu\oplus\Hi_\nu$, and conjugation
by $\delta$ defines an action $\Ad(\delta)\colon G\to\Aut(\K(\Hi))$.
If moreover $\alpha\colon G\to\Aut(A)$ is an action, there is a natural action $\gamma$ of $G$ on the \ca\
$\K(\Hi)\otimes A$ given by $\gamma_g=\Ad(\delta_g)\otimes\alpha_g$ for all $g\in G$. It is easy to check
that $\K(\Hi_\mu,\Hi_\nu)\otimes A$ is invariant with respect to $\gamma$. We will therefore write
$(\K(\Hi_\mu,\Hi_\nu)\otimes A)^G$ for those elements in $\K(\Hi_\mu,\Hi_\nu)\otimes A$ that are fixed under
$\gamma$.
\end{rem}

\begin{df}\label{df:GCtzeq}
Let $(\Hi_\mu,\mu)$ and $(\Hi_\nu,\nu)$ be separable unitary representations of $G$,
and let $a\in (\K(\Hi_\mu)\otimes A)^G$ and $b\in  (\K(\Hi_\nu)\otimes A)^G$ be positive elements.
We say that $a$ is $G$-{\it Cuntz subequivalent} to $b$, and denote this by $a\precsim_G b$,
if there exists a sequence $(d_n)_{n\in\N}$ in $(\K(\Hi_\mu, \Hi_\nu)\otimes A)^G$ such that
$\lim\limits_{n\to\I}\|d_nbd_n^*- a\|=0$.
We say that $a$ is $G$-{\it Cuntz equivalent} to $b$, and denote this by $a\sim_G b$, if
$a\precsim_G b$ and $b\precsim_G a$. The $G$-Cuntz equivalence class of a positive element
$a\in (\K(\Hi_\mu)\otimes A)^G$ will be denoted by $[a]_G$.
\end{df}

We claim that the relation $\precsim_G$ is transitive.
To see this, let $(\Hi_\mu,\mu)$, $(\Hi_\nu,\nu)$ and $(\Hi_\lambda,\lambda)$ be separable unitary representations
of $G$, and let $a\in (\K(\Hi_\mu)\otimes A)^G$, $b\in (\K(\Hi_\nu)\otimes A)^G$, and
$c\in (\K(\Hi_\lambda)\otimes A)^G$ satisfy $a\precsim_G b$ and $b\precsim_G c$. Fix $\ep>0$, and find
$x\in(\K(\Hi_\mu, \Hi_\nu)\otimes A)^G$ such that
$\|a-xbx^*\|<\frac{\varepsilon}{2}$.
Also, since $b\precsim_G c$ there exists $y\in (\K(\Hi_\nu, \Hi_\lambda)\otimes A)^G$ such that
$\|b-ycy^*\|<\frac{\varepsilon}{2\|x\|}$.
The element $z=xy$ belongs to $(\K(\Hi_\mu, \Hi_\lambda)\otimes A)^G$, and is easily seen to satisfy
$\|a-zcz^*\|<\varepsilon$. Since $\varepsilon>0$ is arbitrary, this implies that $a\precsim_G c$.
In particular, it follows that $\sim_G$ is an equivalence relation.

The following lemma is a simple corollary of \cite[Lemma 2.4]{KirRor_absOI} and the definition of $G$-Cuntz subequivalence.

\begin{prop}\label{lem: G Cuntz relation}
Let $(\Hi_\mu,\mu)$ and $(\Hi_\nu,\nu)$ be separable unitary representations of $G$,
and let $a\in (\K(\Hi_\mu)\otimes A)^G$ and $b\in  (\K(\Hi_\nu)\otimes A)^G$ be positive elements.
The following are equivalent:
\be
\item  $a\precsim_G b$.
\item  For every $\varepsilon>0$, there exists $d\in (\K(\Hi_\mu,\Hi_\nu)\otimes A)^G$ such that
$$(a-\varepsilon)_+=dbd^*.$$
\item  For every $\varepsilon>0$, there exist $\delta>0$ and $d\in (\K(\Hi_\mu,\Hi_\nu)\otimes A)^G$, such that
$$(a-\varepsilon)_+=d(b-\delta)_+d^*.$$
\ee
\end{prop}



\begin{df}\label{df:CuGAa}
Let $G$ be a compact group, let $A$ be a \ca, and let $\alpha\colon G\to \Aut(A)$ be a continuous action.
Define the {\it equivariant Cuntz semigroup} $\Cu^G(A,\alpha)$, of the dynamical system $(G,A,\alpha)$,
to be the set of $G$-Cuntz equivalence classes of positive elements in all of the algebras of the form
$(\K(\Hi_\mu)\otimes A)^G$, where $(\Hi_\mu,\mu)$ is a separable unitary representation of $G$.

We define addition on $\Cu^G(A,\alpha)$ as follows. Let $(\Hi_\mu,\mu)$ and $(\Hi_\nu,\nu)$ be separable
unitary representations of $G$, and let $a\in (\K(\Hi_\mu)\otimes A)^G$ and $b\in  (\K(\Hi_\nu)\otimes A)^G$
be positive elements. Denote by $a\oplus b$ the positive element
\[a\oplus b= \begin{pmatrix}
a & 0\\
0 & b
\end{pmatrix}\]
in $(\K(\Hi_\mu\oplus \Hi_\nu)\otimes A)^G$, and
set $[a]_G+[b]_G=[a\oplus b]_G$. (One must check that the definition is independent of the representatives,
but this is routine.)

Finally, we endow $\Cu^G(A,\alpha)$ with the partial order given by $[a]_G\le [b]_G$ if $a\precsim_G b$.
(One has to again check that the order is well defined; we omit the proof.)
\end{df}

It is clear that if $\beta\colon G\to\Aut(B)$ is another continuous action of $G$ on a \ca\ $B$, and if
$\psi\colon A\to B$ is an equivariant $\ast$-homomorphism, then $\psi$ induces an ordered semigroup homomorphism
$\Cu^G(\psi)\colon \Cu^G(A,\alpha)\to \Cu^G(B,\beta)$, given by
\[\Cu^G(\psi)([a]_G)=[(\id_{\K(\Hi_\mu)}\otimes\psi)(a)]_G\]
for $a\in (\K(\Hi_\mu)\otimes A)^G$.

\vspace{0.3cm}

The rest of this section is devoted to proving that the equivariant Cuntz semigroup is a functor from
the category of $G$-$C^*$-algebras to the category $\CCu$ (\autoref{def: CCu}).
This will be accomplished in \autoref{cor:CuGFunctorCCu}.

We point out that in Section~3, we will show that the equivariant
Cuntz semigroup has additional structure, and that $\Cu^G(A,\alpha)$ belongs to a certain category of
semimodules; see \autoref{df:CCuG} and \autoref{thm:EqCtzSmgpCCuG}.


\begin{lma}
\label{lem:EquivHs}
Let $(\Hi_{\mu}, \mu)$ and $(\Hi_\nu, \nu)$ be separable unitary representations of $G$,
and let $a\in (\K(\Hi_\mu)\otimes A)^G$ be a positive element.
Suppose that there exists $V\in (\mathcal{B}(\Hi_{\mu},\Hi_\nu)\otimes A)^G$ satisfying
$V^*V=\id_{\Hi_\mu}\otimes \id_A$.
Then $VaV^*$ is a positive element in $(\K(\Hi_\nu)\otimes A)^G$, and $a\sim_G VaV^*$.

Moreover, if $W\in (\mathcal{B}(\Hi_{\mu},\Hi_\nu)\otimes A)^G$
is another element satisfying $W^*W=\id_{\Hi_\mu}\otimes \id_A$, then $WaW^*\sim_G VaV^*$.
\end{lma}
\begin{proof} It is clear that $VaV^*$ is a $G$-invariant element in
$\K(\Hi_\nu)\otimes A$. Likewise, for $n\in\N$, we have
$a^{\frac{1}{n}}V^*\in (\K(\Hi_\nu)\otimes A)^G$. Now,
\[\lim_{n\to\I}\left\|\left(a^{1/n}V^*\right)(VaV^*)\left(a^{1/n}V^*\right)^*-a\right\|=
\lim_{n\to\I}\left\|a^{1/n}aa^{1/n}-a\right\|=0,\]
so $a\precsim_G VaV^*$. Similarly, $Va^{\frac{1}{n}}\in (\K(\Hi_\mu)\otimes A)^G$, and one shows that
\[\lim_{n\to\I}\left\|\left(Va^{1/n}\right)a\left(Va^{1/n}\right)^*-VaV^*\right\|=0.\]
We conclude that $a\sim_G VaV^*$, as desired.

The last part of the statement is immediate. \end{proof}

Let $(\Hi_\mu, \mu)$ and $(\Hi_\nu,\nu)$ be separable unitary representations of $G$, such that
$(\Hi_\mu, \mu)$ is unitarily equivalent to a subrepresentation of $(\Hi_\nu,\nu)$. Then there exists
$W_{\mu, \nu}\in \B(\Hi_\mu, \Hi_\nu)^G$ satisfying $W_{\mu, \nu}^*W_{\mu, \nu}=\id_{\Hi_\mu}$.
Set
\[V^A_{\mu, \nu}=W_{\mu, \nu}\otimes \id_A\in \B(\Hi_\mu, \Hi_\nu)\otimes A.\]
It is clear that $V^A_{\mu,\nu}$ is $G$-invariant and that $(V^A_{\mu, \nu})^*V^A_{\mu, \nu}=\id_{\Hi_\mu}\otimes \id_A$.

Set
\[\iota^A_{\mu,\nu}=\Ad(V^A_{\mu,\nu})\colon (\K(\Hi_\mu)\otimes A)^G\to (\K(\Hi_\nu)\otimes A)^G.\]
Then $\iota^A_{\mu,\nu}$ is a $\ast$-homomorphism, since $V^A_{\mu,\nu}$ is an isometry. Let
\[j^A_{\mu,\nu}= \Cu(\Ad(V^A_{\mu,\nu}))\colon \Cu((\K(\Hi_\mu)\otimes A)^G)\to \Cu((\K(\Hi_\nu)\otimes A)^G)\]
be the $\CCu$-morphism given by $j^A_{\mu,\nu}=\Cu(\iota^A_{\mu,\nu})$. Whenever $A$ and $\alpha$ are
clear from the context, we will write $V_{\mu,\nu}$ for $V^A_{\mu,\nu}$, and similarly for $\iota_{\mu,\nu}$
and $j_{\mu,\nu}$.

\begin{lma}\label{lem:jmunu}
Adopt the notation from the discussion above.
\be\item The map $j_{\mu,\nu}$ is independent
of the choice of $V_{\mu,\nu}$.
\item If $(\Hi_\nu,\nu)$ is unitarily equivalent to a subrepresentation of the separable representation
$(\Hi_\lambda,\lambda)$, then $j_{\nu,\lambda}\circ j_{\mu, \nu}=j_{\mu,\lambda}$.\ee
\end{lma}
\begin{proof}
The first part is an immediate consequence of \autoref{lem:EquivHs}.
The second one is straightforward.
\end{proof}

Next, we show that $\Cu^G(A,\alpha)$ is an object in $\CCu$.

\begin{prop} \label{prop:CuGAainCCu}
Let $\alpha\colon G\to \Aut(A)$ be an action of a compact group $G$ on a \ca\ $A$.
Then $\Cu^G(A,\alpha)$ is a semigroup in $\CCu$.\end{prop}
\begin{proof}
We will check axioms (O1) through (O4) from \autoref{def: CCu}. We divide the proof into four claims.

\emph{Claim 1: $\Cu^G(A,\alpha)$ satisfies (O1)}.
Let $(s_n)_{n\in\N}$ be an increasing sequence in $\Cu^G(A,\alpha)$. Choose separable representations
$(\Hi_{\mu_n},\mu_n)$ of $G$ and positive elements $a_n\in (\K(\Hi_{\mu_n})\otimes A)^G$ with $[a_n]_G=s_n$
for $n\in\N$.

Set $(\Hi_\mu,\mu)=\bigoplus\limits_{n\in\N}(\Hi_{\mu_n},\mu_n)$, which is a separable representation.
Since $\mu_n$ is a subrepresentation of $\mu$, there exists a natural map
\[j_{\mu_n,\mu}\colon \Cu((\K(\Hi_{\mu_n})\otimes A)^G)\to \Cu((\K(\Hi_{\mu})\otimes A)^G)\]
in the category $\CCu$. Since $\Cu((\K(\Hi_{\mu})\otimes A)^G)$ is an object in $\CCu$,
the supremum $s=\sup\limits_{n\in\N} j_{\mu_n,\mu}(s_n)$ exists in $\Cu((\K(\Hi_{\mu})\otimes A)^G)$.
Choose a positive element $a\in (\K(\Hi_{\mu})\otimes A)^G$ with $[a]=s$. We claim that $[a]_G$
is the supremum of $(s_n)_{n\in\N}$ in $\Cu^G(A,\alpha)$.

Let $t\in \Cu^G(A,\alpha)$ satisfy $s_n\leq t$ for all $n\in\N$. Choose a separable representation
$(\Hi_\nu,\nu)$ and a positive element $b\in(\K(\Hi_{\nu})\otimes A)^G$ representing $t$.
Then
\[j_{\mu_n,\mu\oplus\nu}(s_n)\leq j_{\mu_n,\mu\oplus\nu}(t)\]
for all $n\in\N$. It is clear that $a\precsim b$ in $(\K(\Hi_{\mu})\otimes A)^G$, so
$[a]_G\leq [b]_G$ in $\Cu^G(A,\alpha)$, as desired.

\emph{Claim 2: $\Cu^G(A,\alpha)$ satisfies (O2).}
Given $s\in \Cu^G(A,\alpha)$ choose $\mu\in \Cu(G)$ and a positive element
$a\in (\K(\Hi_\mu)\otimes A)^G$ with $[a]_G=s$. Then $([(a-\frac{1}{n})_+])_{n\in\N}$ is a rapidly
increasing sequence in $\Cu((\K(\Hi_\mu)\otimes A)^G)$, and its supremum is $[a]$. Using the
description of suprema from the above paragraph, it follows that the same is true for the elements
$[(a-\frac{1}{n})_+]_G$ in $\Cu^G(A,\alpha)$, as desired.

\emph{Claim 3: $\Cu^G(A,\alpha)$ satisfies (O3).}
Let $(s_n)_{n\in\N}$ and $(t_n)_{n\in\N}$ be increasing sequences in $\Cu^(G,\alpha)$. For $n\in\N$, choose
separable representations $(\Hi_{\mu_n},\mu_n)$ and $(\Hi_{\nu_n},\nu_n)$ of $G$, and positive elements
$a_n\in (\K(\Hi_{\mu_n})\otimes A)^G$ and $b_n\in (\K(\Hi_{\nu_n})\otimes A)^G$ satisfying $[a_n]_G=s_n$
and $[b_n]_G=t_n$. Set
\[(\Hi_\mu,\mu)=\bigoplus\limits_{n\in\N}(\Hi_{\mu_n},\mu_n) \ \mbox{ and } \
(\Hi_\nu,\nu)=\bigoplus\limits_{n\in\N}(\Hi_{\nu_n},\nu_n),\]
which are separable representations of $G$. Similarly to what was one in Claim 1, find positive elements
$a\in (\K(\Hi_{\mu})\otimes A)^G$ and $b\in (\K(\Hi_{\nu})\otimes A)^G$ satisfyin
\[ [a]=\sup_{n\in\N} j_{\mu_n,\mu}([a_n])\in \Cu((\K(\Hi_\mu)\otimes A)^G) \ \mbox{ and } \
[b]=\sup_{n\in\N} j_{\nu_n,\nu}([b_n])\in \Cu((\K(\Hi_\nu)\otimes A)^G).\]
Then $[a]_G=\sup\limits_{n\in\N} s_n$ and $[b]_G=\sup\limits_{n\in\N} t_n$, by the proof of Claim 1. Using
that $\Cu((\K(\Hi_{\mu\oplus\nu})\otimes A)^G)$ is an object in $\CCu$ at the second step, we get
\begin{align*}
j_{\mu,\mu\oplus\nu}([a])+j_{\nu,\mu\oplus\nu}([b])&= \sup_{n\in\N} j_{\mu_n,\mu\oplus\nu}([a_n])+\sup_{n\in\N} j_{\nu_n,\mu\oplus\nu}([b_n])\\
&\sup_{n\in\N} \left[j_{\mu_n,\mu\oplus\nu}([a_n])+ j_{\nu_n,\mu\oplus\nu}([b_n])\right].
\end{align*}
It is then clear that $[a]_G+[b]_G$ is the supremum of $(s_n+t_n)_{n\in\N}$ in $\Cu^G(A,\alpha)$, and the claim is proved.

\emph{Claim 4: $\Cu^G(A,\alpha)$ satisfies (O4).}
Let $s_1,s_2,t_1,t_2\in \Cu^G(A,\alpha)$ satisfy $s_j\ll t_j$ for $j=1,2$. In order to check that $s_1+s_2\ll t_1+t_2$,
let $(r_n)_{n\in\N}$ be an increasing sequence in $\Cu^G(A,\alpha)$ satisfying $t_1+t_2\leq \sup\limits_{n\in\N} r_n$. Find
a large enough separable representation $(\Hi_\mu,\mu)$ of $G$, and positive elements
\[a_1,a_2,b_1,b_2,c_n\in (\K(\Hi_\mu)\otimes A)^G,\]
satisfying $[a_j]_G=s_j, [b_j]_G=t_j$ and $[c_n]_G=r_n$ for $j=1,2$ and for $n\in\N$.
Then $[a_j]\ll [b_j]$, for $j=1,2$, and $[b_1]+[b_2]\leq \sup\limits_{n\in\N}[c_n]$ in $\Cu((\K(\Hi_\mu)\otimes A)^G)$.
Since $\Cu((\K(\Hi_\mu)\otimes A)^G)$ is an object in $\CCu$, there exists $m\in\N$ such that
$[a_1]+[a_2]\leq [c_m]$. It follows that
\[s_1+s_2=[a_1]_G+[a_2]_G\leq [c_m]_G=r_m,\]
as desired. This finishes the proof of the claim and the proposition.
\end{proof}

\begin{rem} It follows from the above proof that the supremum of an increasing sequence of elements
in $\Cu^G(A,\alpha)$ can be computed in a single space of the form $(\K(\Hi_\mu)\otimes A)^G$.
Likewise, compact containment can also be verified in a single separable representation.\end{rem}

Let us consider the action on $\K(\ell^2(\N)\otimes \Hi_\mu\otimes A)$ induced by tensor product of the trivial action of $G$ on $\ell^2(\N)$ and the given action of $G$ on $\Hi_\mu\otimes A$. Then $\K(\ell^2(\N)\otimes \Hi_\mu\otimes A)^G$ is naturally isomorphic to $\K(\ell^2(\N))\otimes ((\K(\Hi_\mu)\otimes A)^G)$. Thus, the Cuntz semigroup of $(\K(\Hi_\mu)\otimes A)^G$ can be naturally identified with the set of (ordinary) Cuntz equivalences classes of positive elements in $(\K(\ell^2(\N)\otimes \Hi_\mu)\otimes A)^G$.

Fix $[\mu]\in \Cu(G)$. Then the inclusion
\[(\K(\ell^2(\N)\otimes \Hi_\mu)\otimes A)^G\hookrightarrow \bigsqcup_{[\nu]\in\Cu(G)} (\K(\Hi_\nu)\otimes A)^G\]
induces a semigroup homomorphism
\begin{align}\label{eq: i}
i_\mu\colon \Cu((\K(\Hi_\mu)\otimes A)^G)\to \Cu^G(A,\alpha),
\end{align}
which is given by $i_{\mu}([a])=[a]_G$ for a positive element $a\in (\K(\ell^2(\N)\otimes \Hi_\mu)\otimes A)^G$.
By \autoref{lem:EquivHs}, the map $i_\mu$ satisfies
\[i_\mu\circ j_{\nu, \mu}=i_{\nu}\]
whenever $\nu$ is equivalent to a subrepresentation of $\mu$.
It is also clear that $i_\mu$ preserves suprema of increasing sequences,
the compact containment relation, and that it is an order embedding.

Define a preorder $\leq$ on $\Cu(G)$ by setting $[\mu]\leq [\nu]$ if $\mu$ is equivalent to a subrepresentation
of $\nu$. It is clear that $(\Cu(G),\leq)$ is a directed set.
For use in the next theorem, we recall that by Corollary~3.1.11 in~\cite{AntPerThi}, the category
$\CCu$ is closed under direct limits indexed over an arbitrary directed set.

\begin{thm}\label{thm: indlim}
The direct limit of the direct system
\[\left((\Cu((\K(\Hi_\mu)\otimes A)^G))_{[\mu]\in \Cu(G)}, (j_{\mu, \nu})_{[\mu], [\nu]\in \Cu(G),[\mu]\leq[\nu]}\right),\]
in the category $\CCu$, is naturally isomorphic to the pair
\[(\Cu^G(A,\alpha), (i_{\mu})_{\mu\in \Cu(G)}).\]
\end{thm}
For ease of notation, we will denote elements of $\Cu(G)$ and representatives with the same symbols.
\begin{proof}
We will show that $(\Cu^G(A,\alpha), (i_{\mu})_{\mu\in \Cu(G)})$ satisfies the
universal property of the direct limit in $\CCu$.
Let $(S,(\gamma_{\mu})_{\mu\in\Cu(G)})$ be a pair consisting of a semigroup $S$ in the category $\CCu$
and $\CCu$-morphisms
\[\gamma_{\mu}\colon \Cu((\K(\Hi_\mu)\otimes A)^G)\to S,\]
for $\mu\in\Cu(G)$, satisfying $\gamma_\nu\circ j_{\mu, \nu}=\gamma_{\mu}$ for all $\nu,\mu\in\Cu(G)$ with
$\mu\leq \nu$.
Define a map
\[\gamma\colon \Cu^G(A,\alpha)= \bigcup_{\mu\in\Cu(G)}i_{\mu}(\Cu((\K(\Hi_\mu)\otimes A)^G))\to S\]
by $$\gamma (i_\mu(s))=\gamma_{\mu}(s)$$ for $s\in \Cu((\K(\Hi_\mu)\otimes A)^G)$.

The proof will be finished once we prove that $\gamma$ is a well-defined morphism in $\CCu$. We
divide the proof into a number of claims.

\emph{Claim: $\gamma$ is a well defined order preserving map.} For this,
it is enough to show the following. Given $\mu,\nu\in\Cu(G)$ and given $s\in \Cu((\K(\Hi_\mu)\otimes A)^G)$
and $t\in \Cu((\K(\Hi_\nu)\otimes A)^G)$, if $i_{\mu}(s)\le i_\nu (t)$, then
\[\gamma(i_{\mu}(s))\le \gamma(i_\nu (t)).\]

Let $\mu,\nu, s$ and $t$ be as above. Then
\[i_{\mu\oplus \nu}(j_{\mu, \mu\oplus\nu}(s ))=i_\mu(s)\le i_\nu(s)= i_{\mu\oplus \nu}(j_{\nu, \mu\oplus\nu}(t )).\]
Since $i_{\mu\oplus \nu}$ is an order embedding, we deduce that
$j_{\mu, \mu\oplus\nu}(s )\le j_{\nu, \mu\oplus\nu}(t )$. Hence,
\[\gamma(i_\mu(s ))=\gamma_\mu(s )=\gamma_{\mu\oplus\nu}(j_{\mu,\mu\oplus\nu}(s ))\le \gamma_{\mu\oplus\nu}(j_{\nu,\mu\oplus\nu}(t ))=\gamma_\nu(t )=\gamma(i_\nu(t )).\]
The claim is proved.

\emph{Claim: $\gamma$ is a semigroup homomorphism.}
Given $\mu,\nu\in\Cu(G)$, given $s\in \Cu((\K(\Hi_\mu)\otimes A)^G)$,
and given $t\in \Cu((\K(\Hi_\nu)\otimes A)^G)$, we have
\begin{align*}
\gamma(i_{\mu}(s )+i_\nu(t ))&=\gamma_{\mu\oplus\nu}(j_{\mu, \mu\oplus\nu}(s )+j_{\nu, \mu\oplus\nu}(t ))\\
&=\gamma_{\mu\oplus\nu}(j_{\mu, \mu\oplus\nu}(s ))+\gamma_{\mu\oplus\nu}(j_{\nu, \mu\oplus\nu}(t ))\\
&=\gamma(i_{\mu}(s ))+\gamma(i_\nu(t )),
\end{align*}
so the claim follows.

\emph{Claim: $\gamma$ preserves suprema of increasing sequences (condition (M1) in
\autoref{def: CCu}).}
Let $(x_n)_{n\in\N}$ be an increasing sequence in $\Cu^G(A,\alpha)$, and let $x\in \Cu^G(A,\alpha)$ be its supremum.
For each $n\in\N$, choose $[\mu_n]\in\Cu(G)$ and an element
$s_n\in \Cu((\K(\Hi_{\mu_n})\otimes A)^G)$ such that $i_{\mu_n}(s_n)=x_n$. Likewise, choose $[\mu]\in\Cu(G)$ and an element $s\in \Cu((\K(\Hi_\mu)\otimes A)^G)$ such that $i_{\mu}(s)=x$.

Set $\nu=\mu\oplus \bigoplus\limits_{n=1}^\infty \mu_n$. Then
\[i_{\nu}(j_{\mu_n, \nu}(s_n ))=i_{\mu_n}(s_n )\le i_{\mu_{n+1}}(s_{n+1} )=i_{\nu}(j_{\mu_{n+1}, \nu}(s_{n+1} ))\]
for all $n\in\N$. It follows that $j_{\mu_n, \nu}(s_n )\le j_{\mu_{n+1}, \nu}(s_{n+1} )$ for all $n\in \N$, since $i_{\nu}$ is an order embedding. In other words, $(j_{\mu_n, \nu}(s_n ))_{n\in\N}$ is an increasing sequence in $\Cu((\K(\Hi_{\nu})\otimes A)^G)$. Since suprema of increasing sequences exist in $\Cu((\K(\Hi_{\nu})\otimes A)^G)$ and $i_\nu$ and $\gamma_\nu$ are maps in $\CCu$ we get
\begin{align*}
\gamma(x)=&\gamma(i_\mu(s ))
=\gamma\left(\sup\limits_{n\in\N} i_{\mu_n}(s_n )\right)
=\gamma\left(\sup\limits_{n\in\N} i_\nu(j_{\mu_n, \nu}(s_n ))\right)\\
&=\gamma(i_\nu\left(\sup\limits_{n\in\N} j_{\mu_n, \nu}(s_n )\right))
=\gamma_\nu(\sup\limits_{n\in\N} j_{\mu_n, \nu}(s_n ))
=\sup\limits_{n\in\N} \gamma_\nu( j_{\mu_n, \nu}(s_n ))\\
&=\sup\limits_{n\in\N} \gamma(i_{\mu_n}(s_n ))
=\sup_{n\in\N}\gamma(x_n).
\end{align*}
Hence $\gamma(x)$ is the supremum of $(\gamma(x_n))_{n\in\N}$, proving the claim.

\emph{Claim: $\gamma$ preserves the compact containment relation (condition (M2) in
\autoref{def: CCu}).}
Given $\mu,\nu\in\Cu(G)$, and given $s\in \Cu((\K(\Hi_\mu)\otimes A)^G)$
and $t\in \Cu((\K(\Hi_\nu)\otimes A)^G)$, suppose that $i_{\mu}(s )\ll i_{\nu}(t )$. Then
\[i_{\mu\oplus\nu}(j_{\mu, \mu\oplus\nu}(s ))\ll i_{\nu}(j_{\nu, \mu\oplus\nu}(t )).\]
Since $i_{\mu\oplus\nu}$ is a morphism in the category $\CCu$ and it is an order embedding, we deduce that
$j_{\mu, \mu\oplus\nu}(s )\ll j_{\nu, \mu\oplus\nu}(t )$. Hence,
\[\gamma(i_{\mu}(s ))=\gamma_{\mu\oplus\nu}(j_{\mu, \mu\oplus\nu}(s ))\ll \gamma_{\nu\oplus\nu}(j_{\nu, \mu\oplus\nu}(t ))=\gamma(i_{\nu}(t )).\]

We conclude that $\gamma$ is a morphism in $\CCu$, so the proof is complete.\end{proof}




We can now show that semigroup homomorphisms between the equivariant Cuntz semigroups
induced by equivariant $\ast$-homomorphisms are morphisms in $\CCu$.

\begin{prop}\label{prop:EquivMorphInduceCCu}
Let $\beta\colon G\to\Aut(B)$ be an action of $G$ on a \ca\ $B$, and let $\phi\colon A\to B$ be
an equivariant $\ast$-homomorphism.
Then the induced map $\Cu^G(\phi)\colon \Cu^G(A,\alpha)\to \Cu^G(B,\beta)$ is a morphism in the category $\CCu$.
\end{prop}
\begin{proof}
For $[\mu]\in\Cu(G)$, set
\[\phi_\mu=\id_{\K(\Hi_\mu)}\otimes\phi\colon(\K(\Hi_\mu)\otimes A, \Ad(\mu)\otimes\alpha)\to(\K(\Hi_\mu)\otimes B, \Ad(\mu)\otimes\beta).\]
Then $\phi_\mu$ is equivariant. Its induced map
\begin{align*}
\Cu(\phi_\mu)\colon \Cu((\K(\Hi_\mu)\otimes A)^G)\to \Cu((\K(\Hi_\mu)\otimes B)^G),
\end{align*}
between the Cuntz semigroups of the corresponding fixed point algebras, is a morphism in $\CCu$.

For $[\mu]\leq [\nu]$, we have
\[j^B_{\mu,\nu}\circ \Cu(\phi_\mu)=\Cu(\phi_\nu)\circ j^A_{\mu, \nu}.\]
Consequently, the maps
\[i^B_\mu\circ \Cu(\phi_{\mu})\colon \Cu((\K(\Hi_\mu)\otimes A)^G)\to \Cu^G(B,\beta)\]
satisfy
\[i^B_\mu\circ \Cu(\phi_{\mu})=(i^B_\nu\circ \Cu(\phi_{\nu}))\circ j^A_{\mu, \nu}\]
for $[\mu]\leq [\nu]$. The universal property of the direct limit provides a $\CCu$-morphism
\[\kappa\colon \Cu^G(A,\alpha)\to \Cu^G(B,\beta)\]
satisfying $\kappa\circ i^A_\mu=i^B_\mu\circ \Cu(\phi_{\mu})$ for all $[\mu]\in\Cu(G)$.
For $s\in \Cu((\K(\Hi_\mu)\otimes A)^G)$, we have
\[\kappa(i^A_\mu(s))=i^B_\mu(\Cu(\phi_{\mu})(s))=i^B_\mu(\Cu(\id_{\K(\Hi_\mu)}\otimes \phi)([a]))=\Cu^G(\phi)([a]).\]

We conclude that $\kappa=\Cu^G(\phi)$, and hence $\Cu^G(\phi)$ is a morphism in $\CCu$.
\end{proof}

Since $\Cu^G$ obviously preserves composition of maps, we get the following.

\begin{cor}\label{cor:CuGFunctorCCu}
The equivariant Cuntz semigroup $\Cu^G$
is a functor from the category of $G$-$C^*$-algebras to the category $\CCu$.
\end{cor}

\section{The semiring \texorpdfstring{$\Cu(G)$}{Cu(G)} and the category \texorpdfstring{$\CCu^G$}{CCuG}}

\subsection{The semiring $\Cu(G)$}
Let $G$ be a compact group. Denote by $V(G)$ the semigroup of equivalence classes
of finite dimensional representations of $G$, the operation being given by direct sum. Recall
that the \emph{representation ring} $R(G)$ of $G$ is the Grothendieck
group of $V(G)$. The product structure on $R(G)$ is induced by the tensor
product of representations. The construction of $R(G)$ resembles that of $K$-theory,
while the object we define below is its Cuntz semigroup analog.

Recall that a \emph{semiring} is a set $R$ with two binary operations $+$ and $\cdot$ on $R$,
which satisfy all axioms of a unital ring except for the axiom demanding the existence of
additive inverses.

\begin{df}\label{def: rep semiring}
The \emph{representation semiring} of $G$, denoted by $\Cu(G)$, is the set of all equivalence
classes of unitary representations of $G$ on separable Hilbert spaces.
Addition in $\Cu(G)$ is given by the direct sum of representations, while product in $\Cu(G)$ is
given by the tensor product.
We endow $\Cu(G)$ with the order: $[\mu]\le [\nu]$ if $\mu$ is unitarily equivalent to a
subrepresentation of $\nu$.
\end{df}

Since the tensor product of representations is associative, it is clear that $\Cu(G)$
is indeed a semiring.

\begin{lma}\label{lem:RestrCompactGEquiv}
Let $A$ be a $C^*$-algebra, let $G$ be a compact group, and let $\alpha\colon G\to\Aut(A)$ be an action.
Let $(\Hi_\mu,\mu)$ be a separable unitary representation of $G$,
and let $a\in \K(\Hi_\mu)^G$ be a positive element.
Set $\Hi=\overline{a(\Hi_\mu)}$, and let $a'$ be the restriction of $a$ to $\Hi$. Then $a'$ is a
$G$-invariant strictly positive element in $\K(\Hi)$, and there exists a sequence $(d_n)_{n\in\N}$ in
$\K(\Hi_\mu, \Hi)^G$ such that
\[\lim_{n\to\I}\|d_n^*a'd_n- a\|=0 \mbox{ and } \lim_{n\to\I}\|d_n a d_n^*- a'\|=0.\]
\end{lma}
\begin{proof}
Denote by $B_{\Hi_\mu}$ and $B_\Hi$ the unit balls $\Hi_\mu$ and $\Hi$, respectively.
Since $a'(B_\Hi)\subseteq a(B_{\Hi_\mu})$, it is clear that $a'$ is compact.

Since $\lim\limits_{n\to\I} a^\frac{1}{n}(\xi)= \xi$
for all $\xi\in \Hi$, we conclude that $a'$ is strictly positive.
For $n\in\N$, let $d_n\colon \Hi_\mu \to \Hi$ be the operator defined by restricting the
codomain of $a^\frac{1}{n}$ to $\Hi$. Then $d_n\in \K(\Hi_\mu, \Hi)^G$, and $d_n^*\colon \Hi\to \Hi_\mu$ is given by
$d_n^*(\xi)=a^\frac{1}{n}(\xi)$ for all $\xi\in \Hi$. It is now clear that
\[\lim_{n\to\I}\|d_n^*a'd_n- a\|=0 \mbox{ and } \lim_{n\to\I}\|d_n a d_n^*- a'\|=0,\]
so the proof is complete.
\end{proof}

The following observation will be used throughout without particular reference. Recall
that a positive element $x$ in a \ca\ $A$ is said to be \emph{strictly positive} if
$\tau(x)>0$ for every positive linear map $\tau\colon A\to\C$.

\begin{rem}
Let $G$ be a compact group, let $A$ be a \ca, and let $\alpha\colon G\to\Aut(A)$ be an action.
Denote by $\mu$ the normalized Haar measure on $G$.
If $x\in A$ is a strictly positive element, then $y=\int\limits_G \alpha_g(x)\ d\mu(g)$
is strictly positive in $A$. Indeed, let $\tau\colon A\to \C$ be a positive linear map. For $g\in G$,
the map $\tau\circ\alpha_g\colon A\to \C$ is also linear and positive, and so $\tau(\alpha_g(x))>0$.
Since $g\mapsto \tau(\alpha_g(x))$ is continuous, we deduce that
\[\tau(y)=\tau\left(\int\limits_G \alpha_g(x)\ d\mu(g)\right)=\int\limits_G \tau(\alpha_g(x))\ d\mu(g)>0,\]
so $y$ is strictly positive, as desired.
\end{rem}

Let $(\Hi_\mu,\mu)$ be a separable unitary representation of $G$.
Since $\Hi_\mu$ is separable, $\K(\Hi_\mu)$ has a strictly positive element $\widetilde{s}_\mu$. Moreover,
by integrating $g\cdot \widetilde{s}_\mu$ over $G$, we get an invariant strictly positive element $s_\mu$
of $\K(\Hi_\mu)$.

\begin{thm}\label{thm: mu s_mu}
Adopt the notation from the comments above.
Then the map $s\colon \Cu(G)\to \Cu^G(\C)$ given by
$s([\mu])=[s_\mu]$, for $[\mu]\in \Cu(G)$, is well defined.
Moreover, it is an isomorphism of ordered semigroups.
\end{thm}
\begin{proof}
We begin by showing that $s$ is well defined. Let $(\Hi_\mu, \mu)$ and $(\Hi_\nu, \nu)$ be separable
unitary representations of $G$, with $[\mu]\le [\nu]$. Then there exists
$V\in \mathcal{B}(\Hi_\mu, \Hi_\nu)^G$ such that $V^*V=\id_{\Hi_\mu}$. By \autoref{lem:EquivHs},
we have $s_\mu\sim_G Vs_\mu V^*$. Since $s_\nu$ is strictly positive,
we also have $Vs_{\mu}V^*\precsim_G s_\nu$. Thus, $s_\mu\precsim_G s_\nu$. It follows that $s$ is well
defined and order preserving.

We now show that $s$ is an order embedding. Let $s_\mu\in \K(\Hi_\mu)^G$ and $s_\nu\in \K(\Hi_\nu)^G$
be strictly positive elements such that $s_\mu\precsim_G s_\nu$. By \cite[Proposition 2.5]{CiuEllSan_TypeI}
applied to $\K(\Hi_\mu\oplus\Hi_\nu)$ (with the convention from before \autoref{df:GCtzeq}),
there exists $x\in \K(\Hi_\mu, \Hi_\nu)$ such that $s_\mu=x^*x$ and $xx^*\in \K(\Hi_\nu)^G$. Also, by simply inspecting the proof of that proposition, one sees that $x$ can be taken in $\K(\Hi_\mu, \Hi_\nu)^G$.
Let $x=v(x^*x)^\frac{1}{2}$ be the polar decomposition of $x$. Then $v$ belongs to
$\mathcal{B}(\Hi_\mu, \Hi_\nu)^G$, and $v^*v=\id_{\Hi_\mu}$. This implies that $[\mu]\le [\nu]$. In particular,
$s$ is injective.

To finish the proof, we show that $s$ is surjective.
Let $(\Hi_\mu, \mu)$ be a separable unitary representation of $G$, and let $a$ be a strictly positive
element in $\K(\Hi_\mu)^G$.
Set $\Hi_\nu=\overline{a(\Hi_\mu)}$, let $\nu$ be the restriction of $\mu$ to $\Hi_\nu$, and
let $a'\colon \Hi_\nu\to \Hi_\nu$ be the restriction of $a$.
By \autoref{lem:RestrCompactGEquiv}, $a'$ is a positive element in $\K(\Hi_\nu)^G$,
and $a'\sim_G a$. It follows that $s([\nu])=[a]$, and the proof is complete.
\end{proof}

In particular, the above theorem shows that $\Cu(G)$ is a $\Cu$-semiring in the sense
of Definition~7.1.1 in~\cite{AntPerThi}.

\begin{cor}\label{cor: sup ll}
The semigroup $\Cu(G)$ is an object in $\CCu$. In addition,
\be
\item If $([\mu_n])_{n\in \N}$ is an increasing sequence in $\Cu(G)$,
then $[\mu]$ is the supremum of $([\mu_n])_{n\in\N}$
if and only if $[s_\mu]=\sup\limits_{n\in\N}[s_{\mu_n}]$;
\item $[\mu]\ll [\nu]$ if and only if $[s_\mu]\ll [s_\nu]$.
\ee
\end{cor}

Recall that when $G$ is compact, every unitary representation of $G$ is equivalent to a direct sum of
finite dimensional representations.

\begin{cor}
Let $(\Hi_\mu, \mu)$ be a separable unitary representation of $G$.
Let $(\Hi_{\nu_k},\nu_k)_{k\in\N}$ be a family of non-zero finite dimensional representations of $G$ such that
\[(\Hi_\mu,\mu)\cong \bigoplus\limits_{k\in \N}(\Hi_{\nu_k},\nu_k).\]
For $n\in\N$, set $\mu_n=\bigoplus\limits_{k=1}^n\nu_k$. Then $[\mu_n]\ll [\mu]$ for all $n\in\N$,
and
\[[\mu]=\sup\limits_{n\in\N}[\mu_n].\]
\end{cor}
\begin{proof}
Let $s_\mu$ be a strictly positive element of $\K(\Hi_\mu)^G$. For each $n\in \N$, let $p_n$ be
the unit of $\K(\Hi_{\mu_n})$. Then $[s_\mu]=\sup\limits_{n\in\N} [p_n]$ since
$\Hi_\mu\cong \bigoplus\limits_{k\in \N}\Hi_{\nu_k}$. Also, $[p_n]\ll [s_\mu]$ for all $n\in\N$,
because $p_n$ is a projection. The result then follows from \autoref{cor: sup ll}.
\end{proof}

\subsection{The $\Cu(G)$-semimodule structure on $\Cu^G(A,\alpha)$}

Throughout the rest of this section, we fix a compact group $G$, a \ca\ $A$, and a
continuous action $\alpha\colon G\to\Aut(A)$.

Recall that a (left) semimodule over a semiring $R$, or an $R$-semimodule, is a
commutative monoid $S$ together with a function $\cdot\colon R\times S \to S$ satisfying all the axioms
of a module over a ring, except for the axiom demanding the existence of additive inverses.

In this subsection, we show that $\Cu^G(A,\alpha)$ has a natural $\Cu(G)$-semimodule structure, which
moreover satisfies a number of additional regularity properties. It follows that the equivariant Cuntz
semigroups belong to a distinguished class of partially ordered semirings over $\Cu(G)$. We begin by defining
this category, and then show that $\Cu^G(A,\alpha)$ belongs to it; see \autoref{thm:EqCtzSmgpCCuG}.

\begin{df} \label{df:CCuG}
Denote by $\CCu^G$ the category defined as follows. The objects in $\CCu^G$ are partially ordered
$\Cu(G)$-semimodules $(S,+,\cdot)$ such that:
\begin{itemize}
\item[(O1)] $S$ is an object in $\CCu$;
\item[(O2)] if $x,y\in S$ and $r,s\in \Cu(G)$ satisfy $x\leq y$ and $r\leq s$, then $r\cdot x\leq s\cdot y$;
\item[(O3)] if $x,y\in S$ and $r,s\in \Cu(G)$ satisfy $x\ll y$ and $r\ll s$, then $r\cdot x\ll s\cdot y$;
\item[(O4)] if $(x_n)_{n\in \N}$ is an increasing sequence in $S$, and $(r_n)_{n\in \N}$ is an increasing
sequence in $\Cu(G)$, then
\[\sup_{n\in\N} (r_n\cdot x_n)=\left(\sup_{n\in\N} r_n\right)\cdot\left(\sup_{n\in\N} x_n\right).\]
\end{itemize}
The morphisms in $\CCu^G$ between two $\Cu(G)$-semimodules $S$ and $T$
are all $\Cu(G)$-semimodule homomorphisms $\varphi\colon S\to T$ in the category $\CCu$.
\end{df}

In particular, a $\Cu(G)$-semimodule is a $\Cu$-semimodule in the sense of Definition~7.1.3
of~\cite{AntPerThi}.

\begin{lma}\label{lem:AxO4}
Axiom (O4) in \autoref{df:CCuG} is equivalent to the following.
If $(x_n)_{n\in \N}$ is an increasing sequence in $S$, and $(r_n)_{n\in \N}$ is an increasing
sequence in $\Cu(G)$, then

\[\sup_{n\in\N} (r_n\cdot x)=\left(\sup_{n\in\N} r_n\right)\cdot x \ \mbox{ and } \
\sup_{n\in\N} (r\cdot x_n)=r\cdot\left(\sup_{n\in\N} x_n\right)\]
for all $r\in \Cu(G)$ and for all $x\in S$.
\end{lma}
\begin{proof} That Axiom (O4) implies the condition in the statement is immediate. Conversely,
suppose that $S$ satisfies Axioms (O1), (O2) and (O3), and the condition in the statement.
Let $(x_n)_{n\in \N}$ be an increasing sequence in $S$, and let $(r_n)_{n\in \N}$ be an increasing
sequence in $\Cu(G)$. For $m\in\N$, we have
$r_m\cdot x_m\leq \sup\limits_{n\in\N} r_n\cdot\sup\limits_{n\in\N} x_n$ by
Axiom (O2), so
\[\sup_{m\in\N} (r_m\cdot x_m)\leq \left(\sup_{n\in\N} r_n\right)\cdot\left(\sup_{n\in\N} x_n\right).\]

For the opposite inequality, given $m\in\N$ we have
\[\sup_{n\in\N} (r_n\cdot x_n)\geq \sup_{n\in\N} (r_n\cdot x_m)=\left(\sup_{n\in\N}r_n\right) x_m.\]
By taking $\sup\limits_{m\in\N}$, we conclude that Axiom (O4) is also satisfied.
\end{proof}

We will need to know the following:

\begin{thm}\label{thm:CCuGClosedIndLim}
The category $\CCu^G$ is closed under countable direct limits.
\end{thm}
\begin{proof}
Let $(S_n,\varphi_n)_{n\in\N}$ be a direct system in the category $\CCu^G$, with $\CCu^G$-mor\-phisms
$\varphi_n\colon S_n \to S_{n+1}$. For $m\geq n$, we write $\varphi_{m,n}\colon S_n\to S_{m+1}$ for
the composition $\varphi_{m,n}=\varphi_m\circ\cdots\circ\varphi_n$.
By Theorem~2 in~\cite{CowEllIva}, the limit of this
direct system exists in the category $\CCu$, and we denote it by $(S,(\psi_n)_{n\in\N})$,
where $\psi_n\colon S_n\to S$ is a $\CCu$-morphism satisfying
$\psi_{n+1}\circ \varphi_n=\psi_{n}$ for all $n\in\N$. We will use the description
of the direct limit given in the proof of Theorem~2 in~\cite{CowEllIva}, in the form
given in Proposition~2.2 in~\cite{GarSan_RokConstrI}.

We define a $\Cu(G)$-semimodule structure on $S$ as follows. Let $x\in S$, and choose
elements $x_n\in S_n$, for $n\in\N$, such that $\varphi_n(x_n)\ll x_{n+1}$
and $\sup\limits_{n\in\N}\psi_n(x_n)=x$. Given $r\in\Cu(G)$, set
$r\cdot x=\sup\limits_{n\in\N}\psi_n(r\cdot x_n)$.

\emph{Claim: the $\Cu(G)$-semimodule structure is well-defined and satisfies Axiom (O2).}
It is clearly enough to check Axiom (O2). Let $x,y\in S$ with $x\leq y$, and let $r,s\in\Cu(G)$
with $r\leq s$. Choose
elements $x_n, y_n\in S_n$, for $n\in\N$, satisfying $\varphi_n(x_n)\ll x_{n+1}$
and $\sup\limits_{n\in\N}\psi_n(x_n)=x$, as well as $\varphi_n(y_n)\ll y_{n+1}$
and $\sup\limits_{n\in\N}\psi_n(y_n)=y$.

Given $n\in\N$, we have $\psi_n(x_{n})\ll x\leq y=\sup_{m\in\N}\psi_m(y_m)$, so there exists $m_0\in\N$
such that $\psi_n(x_{n})\leq \psi_m(y_m)$ for all $m\geq m_0$. Without loss of generality, we may assume
$m_0\geq n$. Since $\varphi_{n}(x_{n})\ll x_{n+1}$,
part~(ii) of Proposition~2.2 in~\cite{GarSan_RokConstrI} implies that there exists $n_0\in\N$ with $n_0\geq m$
such that $\varphi_{n_0,n}(x_n)\leq \varphi_{n_0,m}(y_m)$.
It follows that
\[\varphi_{n_0,n}(r\cdot x_n)\leq \varphi_{n_0,m}(s\cdot y_m)\]
for all $k\geq n_0$. Composing with $\psi_{n_0}$, we deduce that $\psi_n(r\cdot x_n)\leq \psi_m(s \cdot y_m)$
for all $m\geq m_0$. Taking first the supremum over $m$, and then the supremum over $n$, we conclude that
\[\sup_{n\in\N}\psi_n(r\cdot x_n)\leq \sup_{m\in\N}\psi_m(s\cdot y_m).\]
The claim is proved.

\emph{Claim: $S$ satisfies Axiom (O3).}

Let $x,y\in S$ with $x\ll y$, and let $r,s\in\Cu(G)$
with $r\ll s$. Then there exist $n \in\N$, and $x',y' \in S_n$ such that
$x'\ll y'$ and $x\ll \psi_n(x') \ll \psi_n(y') \ll y$. Then $r\cdot x' \ll s\cdot y'$,
and hence
\[r\cdot x\leq \psi_n(r\cdot x') \ll \psi_n(s\cdot y') \leq s\cdot y, \]
as desired.

\emph{Claim: $S$ satisfies Axiom (O4).}
It suffices to check the conditions in the statement of \autoref{lem:AxO4}. Let $(r_n)_{n\in\N}$
be an increasing sequence in $\Cu(G)$, and let $x\in S$. Choose
elements $x_m\in S_m$, for $m\in\N$, such that $\varphi_m(x_m)\ll x_{m+1}$
and $\sup\limits_{m\in\N}\psi_m(x_m)=x$. Then
\[\sup_{n\in\N}\left(r_n\cdot x\right)=\sup_{n\in\N}\sup_{m\in\N}\psi_m(r_n\cdot x_m)
=\sup_{n\in\N}r_n\cdot \left(\sup_{m\in\N}\psi_m(x_m)\right)= \left(\sup_{n\in\N}r_n\right)\cdot x,\]
as desired. The other property in \autoref{lem:AxO4} is checked identically. This concludes the proof.
\end{proof}

We now define a $\Cu(G)$-semimodule structure on $\Cu^G(A,\alpha)$.

\begin{df}\label{df:SemimodStructCuG}
Let $(\Hi_\mu,\mu)$ and $(\Hi_\nu,\nu)$ be separable unitary representations of $G$, and let
$a\in (\K(\Hi_\mu)\otimes A)^G$ be a positive element.
In this definition, and to stress the role played by $\mu$, we write
$[(\Hi_\mu,\mu,a)]_G$ for the $G$-Cuntz equivalence class of $a$.
Use separability of $\Hi_\nu$ to choose a $G$-invariant strictly positive element $s_\nu\in\K(\Hi_\nu)^G$.
We set
\[[\nu]\cdot [(\Hi_\mu,\mu,a)]_G=\left[\left(\Hi_\nu\otimes \Hi_\mu,\nu\otimes \mu,s_\nu \otimes a\right)\right]_G.\]
\end{df}

The following is one of the main results in this section. For use in its proof,
we recall that any tensor product of \ca s respects Cuntz subequivalence.

\begin{thm}\label{thm:EqCtzSmgpCCuG}
The $\Cu(G)$-semimodule structure from \autoref{df:SemimodStructCuG} is well defined.
Moreover, with this structure, the semigroup $\Cu^G(A,\alpha)$ becomes an object in $\CCu^G$, and
the equivariant Cuntz semigroup is a functor from the category of $G$-$C^*$-algebras to the
category $\CCu^G$.
\end{thm}
\begin{proof}
We will prove that the $\Cu(G)$-semimodule structure is well defined together with
condition O2 in \autoref{df:CCuG}. So let $[\mu], [\nu]\in \Cu(G)$ and $[a]_G, [b]_G\in \Cu^G(A,\alpha)$
satisfy $[\mu]\le [\nu]$ and $[a]_G\le [b]_G$. By \autoref{thm: mu s_mu}, we have
$s_\mu\precsim_G s_\nu$. Since we also have $a\precsim_G b$, we get
$s_\mu\otimes a\precsim_G s_\nu\otimes b$.
Hence,
\[[\mu]\cdot[a]_G=[s_\mu\otimes a]_G\le [s_\nu\otimes b]_G= [\nu]\cdot[b]_G,\]
as desired.

It is immediate that
\[[\mu]\cdot ([a]_G+[b]_G)=[\mu]\cdot[a]_G+[\mu]\cdot[b]_G\]
and
\[[\mu]\cdot([\nu]\cdot[a]_G)=([\mu]\cdot[\nu])\cdot[a]_G,\]
for all $[\mu], [\nu]\in \Cu(G)$ and for all $[a]_G, [b]_G\in \Cu^G(A,\alpha)$.
We conclude that $\Cu^G(A,\alpha)$ is a $\Cu(G)$-semimodule, and that it satisfies condition O2
in \autoref{df:CCuG}.

We now proceed to show that $\Cu^G(A,\alpha)$ is an object in $\CCu^G$. We already showed in
\autoref{thm: indlim} that it is an object in $\CCu$, so condition O1 in \autoref{df:CCuG} is
satisfied.

We check condition O3.
Suppose that $[\mu], [\nu]\in \Cu(G)$ and $[a]_G, [b]_G\in \Cu^G(A,\alpha)$ satisfy
$[\mu]\ll [\nu]$ and $[a]_G\ll [b]_G$. By \autoref{thm: mu s_mu}, we get
$[s_\mu]_G\ll [s_\nu]_G$ in $\Cu^G(\C)$. Without loss of generality, we may assume that $s_\nu$ and
$b$ are contractions.
Recall that
\[[s_\nu]_G=\sup\limits_{\varepsilon>0} [(s_\nu-\varepsilon)_+]_G \mbox{ and }
[b]_G=\sup\limits_{\varepsilon>0} [(b-\varepsilon)_+]_G.\]
Using the definition of the compact containment relation, find $\varepsilon>0$ such that
\[[s_\mu]_G\le [(s_\nu-\varepsilon)_+]_G \mbox{ and } [a]_G\le [(b-\varepsilon)_+]_G.\]

Use $\|s_\nu\|\leq 1$ and $\|b\|\leq 1$ at the third step to get
\begin{align*}
[\mu]\cdot[a]_G&=[s_\mu\otimes a]_G\\
&\le [(s_\nu-\varepsilon)_+\otimes(b-\varepsilon)_+]_G\\
&\le [(s_\nu\otimes b-\varepsilon^2)_+]_G\\
&\ll [s_\nu\otimes b]_G=[\nu]\cdot [b]_G,\end{align*}
so condition O3 is satisfied.

We now check condition O4. Let $([\mu_n])_{n\in\N}$ and $([a_n]_G)_{n\in \N}$ be increasing sequences
in $\Cu(G)$ and $\Cu^G(A,\alpha)$, respectively, and set
$[\mu]=\sup\limits_{n\in\N}[\mu_n]$ and $[a]_G=\sup\limits_{n\in\N}[a_n]_G$.
Without loss of generality, we may assume that $a$ is a contraction.
For $n\in\N$, denote by $s_{\mu_n}$ an invariant strictly positive element in $\K(\Hi_{\mu_n})$,
and denote by $s_\mu$ an invariant strictly positive element in $\K(\Hi_\mu)$.
By part~(1) of \autoref{cor: sup ll},
we have $[s_\mu]_G=\sup\limits_{n\in\N}[s_{\mu_n}]_G$ in $\Cu^G(\C)$.
As before, we may assume that $s_\mu$ is a contraction. Then the sequence
$([s_{\mu_n}\otimes a_n]_G)_{n\in \N}$ in $\Cu^G(A,\alpha)$ is increasing. Set
\[[c]_G=\sup\limits_{n\in\N} [s_{\mu_n}\otimes a_n]_G.\]

We claim that $[c]_G=[s_\mu\otimes a]_G$.
It is clear that $[c]_G\le [s_\mu\otimes a]_G$. To check the opposite inequality, let $\varepsilon>0$.
Then there exists $n\in \N$ such that
\[[(s_\mu-\varepsilon)_+]\le [s_{\mu_n}]_G \mbox{ and } [(a-\varepsilon)_+]_G\le [a_n]_G.\]
It follows that
\[[(s_\mu\otimes a-\varepsilon)_+]_G\le [(s_\mu-\varepsilon)_+\otimes (a-\varepsilon)_+]_G\le [s_{\mu_n}\otimes a_n]_G.\]
Hence, $[(s_\mu\otimes a-\varepsilon)_+]\le [c]$. Since
$[s_\mu\otimes a]_G=\sup\limits_{\varepsilon>0}[(s_{\mu}\otimes a-\varepsilon)_+]_G$, we deduce that
$[s_\mu\otimes a]_G\le [c]_G$, as desired. We have checked condition O4.

Since $\Cu^G(A,\alpha)$ is an object in $\CCu$ by \autoref{thm: indlim},
we conclude that $\Cu^G(A,\alpha)$ is an object in $\CCu^G$.

It remains to argue that $\Cu^G$ is a functor into $\CCu^G$.
Let $\beta\colon G\to\Aut(B)$ be a continuous action of $G$ on a \ca\ $B$, and let $\phi\colon A\to B$
be an equivariant $\ast$-homomorphism. By \autoref{cor:CuGFunctorCCu}, $\Cu^G(\phi)$ is a morphism in $\CCu$,
so we only need to check that it is a morphism of $\Cu(G)$-semimodules. This is immediate, so the proof
is complete.
\end{proof}

We mention here that naturality of the isomorphism in \autoref{thm: indlim} implies that said isomorphism
becomes a $\CCu^G$-isomorphism when
\[\varinjlim \left((\Cu((\K(\Hi_\mu)\otimes A)^G))_{\mu\in \Cu(G)}, (j_{\mu, \nu})_{\mu, \nu\in \Cu(G),\mu\leq\nu}\right)\]
is endowed with the following $\Cu(G)$-action. For separable representations $(\Hi_\mu,\mu)$ and $(\Hi_\nu,\nu)$
of $G$, and for $x\in \Cu((\K(\Hi_\mu)\otimes A)^G)$, we set
$[\nu]\cdot x = j_{\nu\otimes \mu,\mu}(x)$.

\subsection{Functorial properties of \texorpdfstring{$\Cu^G$}{CuG}}



In this subsection, we prove two functoriality properties of the equivariant Cuntz semigroup.
\autoref{prop:CuGStable} asserts that this functor is stable when the compact operators are given
the trivial $G$-action. We will generalize this result in \autoref{cor:StabUnirep}, where we
replace the trivial action on $\K$ with an arbitrary inner action. (Stability fails if the
action on $\K$ is not inner; see \autoref{eg:NonStableNotInner}.) Then, in \autoref{prop:CuGContinuous}, we show
that the equivariant Cuntz semigroup preserves equivariant inductive limits of sequences. Both
propositions in this subsection will be needed in Section~5.

\begin{prop} \label{prop:CuGStable}
Let $q$ be any rank one projection on $\ell^2(\N)$, and denote by
\[\iota_q\colon A\to A\otimes \K(\ell^2(\N))\]
the inclusion obtained by identifying $A$ with $A\otimes q\K(\ell^2(\N))q$. In other words,
$\iota_q(a)=a\otimes q$ for $a\in A$.

Give $\ell^2(\N)$ the trivial $G$-representation. Then $\iota_q$ induces a natural $\CCu^G$-iso\-mor\-phism
\[\Cu^G(\iota_q)\colon \Cu^G(A,\alpha)\to  \Cu^G(A\otimes \K(\ell^2(\N)),\alpha\otimes \id_{\K(\ell^2(\N))}).\]
\end{prop}
\begin{proof}
We abbreviate $\K(\ell^2(\N))$ to $\K$.
By \autoref{thm:EqCtzSmgpCCuG}, $\Cu^G(\iota_q)$ is a morphism in $\CCu^G$. It thus suffices to check that it is
an isomorphism in $\CCu$.

Let $(\Hi_\mu,\mu)$ be a separable unitary representation of $G$. Denote by
\[\kappa^q_\mu\colon \Cu((\K(\Hi_\mu)\otimes A)^G)\to \Cu((\K(\Hi_\mu)\otimes A)^G\otimes \K)\]
the $\CCu$-morphism induced by the inclusion as the corner associated to $q$. Then $\kappa^q_\mu$ is an
isomorphism (see Appendix~6 in~\cite{CowEllIva}). With the notation from \autoref{thm:EqCtzSmgpCCuG}, it is clear
that
\[j_{\mu,\nu}^{A\otimes\K}\circ \kappa^q_\nu = \kappa^q_\mu \circ j_{\mu,\nu}^{A}\]
for all $\nu\in\Cu(G)$ with $\mu\leq \nu$.

By the universal property of the direct limit in $\CCu$, applied to the object $\Cu^G(A\otimes\K, \alpha\otimes\id_\K)$
and the maps $i^{A\otimes\K}_\mu\circ \kappa^q_\mu$, for $\mu\in \Cu(G)$, it follows that there exists a
$\CCu$-morphism
\[\kappa^q\colon  \Cu^G(A,\alpha)\to  \Cu^G(A\otimes \K,\alpha\otimes \id_{\K})\]
satisfying $\kappa^q\circ i^A_\mu=i_\mu^{A\otimes\K}\circ \kappa^q_\mu$. Since $\kappa^q$ is induced
by $\iota_q$, we must have $\kappa^q=\Cu^G(\iota_q)$. Finally, since $\kappa^q_\mu$ is an isomorphism
for all $\mu$, the same holds for $\kappa^q$, so the proof is complete.
\end{proof}

\begin{prop} \label{prop:CuGContinuous}
Let $(A_n,\iota_n)_{n\in\N}$ be a direct system of \ca s with connecting maps $\iota_n\colon A_n\to A_{n+1}$.
For $n\in\N$, let $\alpha^{(n)}\colon G\to \Aut(A_n)$ be a continuous action, and suppose
that $\alpha^{(n+1)}\circ\iota_n=\iota_{n}\circ\alpha^{(n)}$ for all $n\in\N$. Set $A=\varinjlim (A_n,\iota_n)$,
and $\alpha= \varinjlim \alpha^{(n)}$. Then there exists a natural $\CCu^G$-isomorphism
\[\varinjlim \Cu^G(A_n, \alpha^{(n)})\cong \Cu^G(A,\alpha).\]
\end{prop}
\begin{proof}
By functoriality of $\Cu^G$ (see \autoref{cor:CuGFunctorCCu}), the equivariant inductive
system $(A_n,\alpha^{(n)},\iota_n)_{n\in\N}$ induces the inductive system
\[\left(\Cu^G(A_n,\alpha^{(n)}),\Cu^G(\iota_n)\right)_{n\in\N}\]
in $\CCu^G$. By \autoref{thm:CCuGClosedIndLim}, its inductive limit exists in $\CCu^G$, and we will
denote it by $\varinjlim \Cu^G(A_n,\alpha^{(n)})$.

For $n\in\N$, denote by $\iota_{\I,n}\colon A_n\to A$ the equivariant map into the direct limit.
Then $\Cu^G(\iota_{\I,n})$ is a morphism in $\CCu^G$, and the universal property of inductive limits,
there exists a $\CCu^G$-morphism $\varphi\colon \varinjlim \Cu^G(A_n,\alpha^{(n)})\to \Cu^G(A,\alpha)$.

We claim that $\varphi$ is an isomorphism. For this, it is enough to check that it is an isomorphism
in $\CCu$. Fix a separable representation $(\Hi_\mu,\mu)$ of $G$, and for $n\in\N$, denote by
\[\psi_{n,\mu}\colon \Cu((\K(\Hi_\mu)\otimes A_n)^G)\to \Cu((\K(\Hi_\mu)\otimes A_{n+1})^G)\]
the $\CCu$-morphism induced by $\iota_{n}$. For $\nu\in\Cu(G)$ with $\mu\leq \nu$, we have
\[j_{\mu,\nu}\circ\psi_{n,\nu}=\psi_{n,\nu}\circ j_{\mu,\nu}.\]
Denote by
\[\psi_n\colon \varinjlim_{\mu\in\Cu(G)}\Cu((\K(\Hi_\mu)\otimes A_n)^G)\to \varinjlim_{\mu\in\Cu(G)} \Cu((\K(\Hi_\mu)\otimes A_{n+1})^G)\]
the resulting $\Cu$-morphism, and regard it with a $\Cu$-morphism $\varphi_n\colon \Cu^G(A_n,\alpha^{(n)})\to \Cu^G(A_{n+1},\alpha^{(n+1)})$.
It is clear that $\varphi_n=\Cu^G(\iota_{n})$.

Using Theorem~2 in~\cite{CowEllIva}, the direct limit of $\Cu((\K(\Hi_\mu)\otimes A_n)^G)$ is (naturally) isomorphic, in the category $\CCu$,
to $\Cu((\K(\Hi_\mu)\otimes A)^G)$. This isomorphism can be identified with the map $\varphi$, and this shows that $\varphi$ is an
isomorphism in $\CCu$. This finishes the proof.
\end{proof}

We point out that in the previous proposition, we may allow direct limits over arbitrary directed sets,
using Corollary~3.1.11 in~\cite{AntPerThi} instead of Theorem~2 in~\cite{CowEllIva}.

\section{A Hilbert module picture of \texorpdfstring{$\Cu^G(A,\alpha)$}{CuGAa}}

In analogy with the non-equivariant case, the equivariant Cuntz semigroup can be constructed in terms
of equivariant Hilbert modules. The goal of this section is to present this construction and identify
it with $\Cu^G(A,\alpha)$ in a canonical way.

The description of $\Cu^G(A,\alpha)$ provided in this section will be needed in Section~5, where we will
prove that $\Cu^G(A,\alpha)$ can be naturally identified with $\Cu(A\rtimes_\alpha G)$ (\autoref{Thm:JulgCuG}).

\subsection{Equivariant Hilbert C*-modules}\label{subsec: equivariant}

Throughout this section, we fix a \ca\ $A$, a compact group $G$, and an action $\alpha\colon G\to \Aut(A)$.
All modules will be right modules and a Hilbert $A$-module will mean a Hilbert C*-module over $A$.
The reader is referred to \cite{Lan_Book} for the basics of Hilbert C*-modules.

Given Hilbert $A$-modules $E$ and $F$, we let $\mathcal{L}(E, F)$ and $\K(E, F)$
denote the spaces of adjointable operators and compact operators from $E$ to $F$,
respectively. We write $\U(E,F)$ for the set of unitaries between $E$ and $F$.
When $E=F$, we write $\mathcal{L}(E)$, $\K(E)$, and $\U(E)$ for $\mathcal{L}(E, E)$, $\K(E, E)$, and $\U(E,E)$,
respectively.

\begin{df} A \emph{Hilbert $(G,A,\alpha)$-module} is a pair $(E,\rho)$ consisting of
\begin{enumerate}
\item a Hilbert $A$-module $E$, and
\item a strongly continuous group homomorphism $\rho\colon G\to \U(E)$, satisfying
\begin{enumerate}
\item $\rho_g(x\cdot a)=\rho_g(x)\cdot\alpha_g(a)$ for all $g\in G$, all $x\in E$ and all $a\in A$, and
\item $\langle \rho_g(x),\rho_g(y)\rangle_E=\alpha_g\left(\langle x,y\rangle_E\right)$ for all $g\in G$ and all $x,y\in E$.
\end{enumerate}
\end{enumerate}
(The continuity condition for $\rho$ means that for $x\in E$, the map $G\to E$ given by
$g\mapsto \rho_g(x)$ is continuous.)\end{df}

\begin{df}
If $(E,\rho)$ is a Hilbert $(G,A,\alpha)$-module and $F$ is a Hilbert submodule of $E$ satisfying
$\rho_g(F)\subseteq F$ for all $g\in G$, we will write $\rho|_F$ for the (co-restricted)
group homomorphism $\rho|_F\colon G\to \U(F)$ given by $(\rho|_F)_g(z)=\rho_g(z)$ for all $g\in G$ and all $z\in F$.
The pair $(F,\rho|_F)$ will be called a \emph{Hilbert $(G,A,\alpha)$-submodule} of $(E, \rho)$.

We say that $E$ is \emph{countably generated} if there exists a countable subset
$\{\xi_n\}_{n\in \N}\subseteq E$ such that
\[\left\{\sum_{n=1}^k\xi_n a_n\colon a_n\in A, k\in \N\right\}\]
is dense in $E$.
\end{df}

We will sometimes call Hilbert $(G,A,\alpha)$-modules $G$-Hilbert $(A,\alpha)$-modules, or just $G$-Hilbert
$A$-modules if the action $\alpha$ is understood.

\begin{eg} It is easy to check that if $\beta\colon G\to\Aut(B)$ is an action of $G$ on a \ca\ $B$, then
the pair $(B,\beta)$ is a $G$-Hilbert $B$-module.\end{eg}

Given $G$-Hilbert $A$-modules $(E,\rho)$ and $(F,\eta)$, we let $\mathcal{L}(E,F)^G$ and $\K(E, F)^G$ denote the subsets of $\mathcal{L}(E,F)$ and $\K(E, F)$, respectively, consisting of the equivariant operators. That is,
\begin{align*}
&\mathcal{L}(E,F)^G=\{T\in \mathcal{L}(E,F)\colon T\circ\rho_g=\eta_g\circ T \mbox{ for all } g\in G\},\\
&\K(E,F)^G=\{T\in \K(E,F)\colon T\circ\rho_g=\eta_g\circ T \mbox{ for all } g\in G\}.
\end{align*}
(Note that $\mathcal{L}(E,F)^G$ is the set of fixed points of $\mathcal{L}(E,F)$, where for an adjointable
operator $T\colon E\to F$ and $g\in G$, we set $g\cdot T=\eta_g\circ T \circ \rho_{g^{-1}}$.)

As before, $\mathcal{L}(E)^G$ and $\K(E)^G$ denote $\mathcal{L}(E,E)^G$ and $\K(E,E)^G$, respectively.

\begin{df}\label{def: Hilbert subequivalence}
Let $(E,\rho)$ and $(F,\eta)$ be $G$-Hilbert A-modules. We say that $(E, \rho)$ is \emph{isomorphic} to $(F,\eta)$,
in symbols $(E,\rho)\cong (F,\eta)$, if there exists a unitary in $\mathcal{L}(E,F)^G$.
We say that $(E,\rho)$ is \emph{subequivalent} to $(F,\eta)$, in symbols $(E,\rho)\preceq (F,\eta)$,
if $(E,\rho)$ is isomorphic to a direct summand of $(F, \eta)$.
(That is, if there exists $V\in \mathcal{L}(E, F)^G$ such that $V^*V=\id_E$.)
\end{df}

Let $I$ be a set and let $(E_j, \rho_j)_{j\in I}$ be a family of Hilbert $A$-modules. Then the Hilbert direct sum $\left(\bigoplus\limits_{j\in I} E_j, \bigoplus\limits_{j\in I}\rho_j\right)$ is the completion of the corresponding
algebraic direct sum with respect to the norm defined by the scalar product
\[\left\langle\bigoplus\limits_{j\in I}\xi_j, \bigoplus\limits_{j\in I}\zeta_j\right\rangle=
\sum_{j\in I}\langle \xi_j, \zeta_j\rangle.\]

Let $\Hi$ be a Hilbert space. By convention, the scalar product on $\Hi$ is linear in the second argument and conjugate linear
in the first one. Let $\Hi\otimes A$ denote the exterior tensor product of $\Hi$ and $A$, where $A$ is considered as a right $A$-module over itself (\cite[Chapter 4]{Lan_Book}). That is, $\Hi\otimes A$ is the completion of the algebraic tensor product $\Hi\otimes_{\mathrm{alg}} A$ in the norm given by the $A$-valued product
\[\langle\xi_1\otimes a_1, \xi_2\otimes a_2 \rangle=\langle\xi_1, \xi_2\rangle a_1^*a_2\]
for $\xi_1,\xi_2\in \Hi$ and $a_1,a_2\in A$.

\begin{df}\label{df:H_A}
For each element $[\pi]\in \widehat{G}$, choose a representative $\pi\colon G\to \U(\Hi_\pi)$.
Denote by $\Hi_\C$ the Hilbert space direct sum
\[\Hi_\C=\bigoplus\limits_{[\pi]\in \widehat G}\bigoplus\limits_{n=1}^\infty \Hi_\pi,\]
and let $\pi_\C\colon G\to \U(\Hi_\C)$ be the unitary representation given by
\[\pi_\C=\bigoplus\limits_{[\pi]\in \widehat G}\bigoplus\limits_{n=1}^\infty \pi.\]
The unitary representation $(\Hi_\C,\pi_\C)$ is easily seen not to depend on the choices
of representatives $\pi\colon G\to \U(\Hi_\pi)$ up to unitary equivalence.

We define the \emph{universal $G$-Hilbert $(A,\alpha)$-module}
$(\Hi_A,\pi_A)$ to be $\Hi_A=\Hi_\C\otimes A$ and $\pi_A=\pi_\C\otimes\alpha$.
\end{df}

\begin{rem}
It is a classical result of Kasparov that when $G$ is second countable, then every countably generated
$G$-Hilbert $A$-module is isomorphic to a direct summand of $(\Hi_A, \pi_A)$; see \cite[Theorem 2]{Kas_HilbMods}.
\end{rem}

\begin{rem}
It is easy to check, using the Peter-Weyl theorem,
that $(\Hi_\C,\pi_\C)$ is (unitarily equivalent) to the representation
$(L^2(G)\otimes\ell^2(\N),\lambda\otimes \id_{\ell^2(\N)})$.
\end{rem}

An equivalent presentation of $\Hi_\C$ (and therefore of $\Hi_A$), when $G$ is second countable, is
\[\Hi_\C=\bigoplus\limits_{[\mu]\in \Cu(G)} \Hi_\mu,\]
with $\pi_\C=\bigoplus\limits_{[\mu]\in \Cu(G)} \mu$.

\begin{lma}\label{lem: submodules of H_A}
Suppose that $G$ is second countable, and let $(E, (\pi_A)|_E)$ be a countably generated $G$-Hilbert
$A$-submodule of $(\Hi_A, \pi_A)$. Then there exists a separable subrepresentation $(\Hi_\mu, \mu)$ of $(\Hi_\C, \pi_\C)$ such that
$E\subseteq \Hi_\mu\otimes A$ and $(\pi_A)|_E=(\mu\otimes \alpha)|_E$.
\end{lma}
\begin{proof}
Since $\Hi_A=\bigoplus\limits_{[\nu]\in\Cu(G)}\left(\Hi_\nu\otimes A\right)$ for any $\xi \in \Hi_A$, there exists
a countable set $X_\xi\subseteq \Cu(G)$ such that $\xi$ belongs to
$\bigoplus\limits_{[\nu]\in X_\xi}\left(\Hi_\nu\otimes A\right)$.

Now let $\{\xi_n\}_{n\in\N}$ be a countable generating subset of $E$.
Then $X=\bigcup\limits_{n\in\N}X_{\xi_n}$ is a countable subset of $\Cu(G)$. Set
\[(\Hi_\mu, \mu)=\bigoplus\limits_{[\nu]\in X}\left(\Hi_{\nu},\nu\right).\]
Then $\Hi_\mu$ is separable. It is immediate that $E\subseteq \Hi_\mu\otimes A$ and
$(\pi_A)|_E=(\mu\otimes \alpha)|_E$, so the proof is complete.
\end{proof}

Let $(E, \rho)$ be a $G$-Hilbert $A$-module.
Then the action $G$ on $E$ induces an action of $G$ on the $C^*$-algebra $\K(E)$ by conjugation.
The fixed point algebra of this action will be denoted by $\K(E)^G$.
When $(E, \rho)$ is the $G$-Hilbert $A$-module $(\Hi_\mu\otimes A, \mu\otimes \alpha)$,
for some separable unitary representation $(\Hi_\mu, \mu)$ of $G$, then the induced action on
$\K(\Hi_\mu\otimes A)$ will be denoted by $\Ad(\mu\otimes \alpha)$.


Let $E$ be a Hilbert $A$-module, and let $\xi,\zeta\in E$. We denote by $\Theta_{\xi,\zeta}\colon E\to E$
the $A$-rank one operator given by $\Theta_{\xi,\zeta}(\eta)=\xi \cdot \langle\zeta,\eta\rangle$ for $\eta\in E$.

\begin{lma}\label{lem: range}
Let $a\in \K(\Hi_A)^G$, and set $E=\overline{\mathrm{span}\{a(\Hi_A)\cup a^*(\Hi_A)\}}$, endowed with the
restricted $G$-representation $\rho$.
Then $(E,\rho)$ is a countably generated $G$-Hilbert $A$-module and $a|_E\in \K(E)^G$.
\end{lma}
\begin{proof}
It is clear that $E$ is invariant under $(\pi_A)_g$ for all $g\in G$; thus $(E,\rho)$
is a $G$-Hilbert $A$-module. Let $\varepsilon>0$. Since $a\in \K(\Hi_A)^G$, there exist $k\in\N$, and
$\xi_1,\ldots,\xi_k,\zeta_1,\ldots,\zeta_k\in \Hi_A$ satisfying
$$\left\|a-\sum_{j=1}^k\Theta_{\xi_j, \zeta_j}\right\|<\frac{\varepsilon}{8}.$$
Use that $a\in \overline{(aa^*)\K(\Hi_A)}$ and $a\in \overline{\K(\Hi_A)(a^*a)}$ to choose $n\in \N$ with \[\left\|aa^*\right\|^\frac{1}{n}<2 \mbox{ and } \left\|a^*a\right\|^\frac{1}{n}<2\]
such that, in addition,
\[\left\|a-(aa^*)^{1/n}a(a^*a)^{1/n}\right\|<\frac{\varepsilon}{2}.\]
It follows that
\begin{align*}
\left\|a-\sum_{j=1}^k\Theta_{(aa^*)^{1/n}(\xi_j), (a^*a)^{1/n}(\zeta_j)}\right\|&=
\left\|a-\sum_{j=1}^k(aa^*)^{1/n}\Theta_{\xi_j, \zeta_j}(a^*a)^{1/n}\right\|\\
&\le \left\|a-(aa^*)^{1/n}a(a^*a)^{1/n}\right\|\\
& \ \ \ \ +\left\|aa^*\right\|^{1/n}
\left\|a-\sum_{j=1}^k\Theta_{\xi_j, \zeta_j}\right\|\left\|a^*a\right\|^{1/n}\\
&<\varepsilon.
\end{align*}

For $j=1,\ldots,k$, the map $\Theta_{(aa^*)^\frac{1}{n}(\xi_j), (a^*a)^\frac{1}{n}(\zeta_j)}$ leaves $E$
invariant, so its restriction to $E$ is a rank one operator in $\K(E)$.
It follows that $a|_E$ is the limit of a sequence finite rank operators on $E$, and hence it is compact.
In particular, $E$ is countably generated, because the range of each finite rank operator is finitely generated.
Finally, it is clear that $a$ is invariant, so $a\in \K(E)^G$.
\end{proof}

Let $(E,\rho)$ be a $G$-Hilbert $A$-module, and let $(F,\rho|_F)$ be a $G$-Hilber submodule of $E$. Then there
is a canonical inclusion equivariant $\iota\colon \K(F) \to \K(E)$, which is defined as follows. For $\xi,\eta\in F$, we
denote by $\Theta_{\xi,\eta}^F\in \K(F)$ the corresponding rank-one operator, and likewise for $\Theta_{\xi,\eta}^E\in \K(E)$.
Then $\iota$ is given by $\iota(\Theta^F_{\xi,\eta})=\Theta_{\xi,\eta}^E$. One must check that $\iota$ extends to an
injective homomorphism $\K(F)\to \K(E)$, and we omit the straightforward verification. That $\iota$ is equivariant is also
clear. 

In particular, in the above situation, $\iota$ restricts to a canonical embedding $\iota\colon \K(F)^G\to \K(E)^G$. This fact
will be used in the proof of the following lemma. 

\begin{lma}\label{lem: submodules}
Suppose that $G$ is second countable. 
Let $(E,\rho)$ be a countably generated $G$-Hilbert $A$-module, and let $(F, \rho|_F)$ be a countably
generated $G$-Hilbert submodule of $E$.
Then there exists $a\in \K(E)^G$ such that $a|_F$ is strictly positive and belongs to $\K(F)$,
and $F=\overline{a(F)}=\overline{a(E)}$.
\end{lma}
\begin{proof}
Use \cite[Proposition 6.7]{Lan_Book} to choose a strictly positive element $c\in \K(F)$.
Using the normalized Haar measure on $G$, it is easy to check that the element $b=\int_G (g\cdot c) \ dg$ 
is strictly positive and $G$-invariant.
Moreover, one has $c(F)=b(F)$.

Using strict positivity of $b$, choose a sequence $(b_n)_{n\in \N}$ in $\K(F)$ such that
\[\lim\limits_{n\to\I}\|bb_n-b^{\frac 1 n}\|=0.\]
By \cite[Equation 1.5]{Lan_Book}, we have $\lim\limits_{n\to\I}bb_n\xi=\xi$
for all $\xi \in F$. It follows that $\overline{b(F)}=F$.

Denote by $\iota\colon \K(F)^G\to \K(E)^G$ the canonical inclusion described in the comments before this lemma.
Set $a=\iota(c)$. Then $a$ is strictly positive and invariant, and $a(E)=a(F)$ is dense in $F$. This finishes
the proof. 
\end{proof}

\subsection{The Hilbert module picture of $\Cu^G(A,\alpha)$}

We now define the relevant equivalence and subequivalence relations of $G$-Hilbert modules that will give rise to a different description of the equivariant Cuntz semigroup. 

\begin{df}\label{df:CptContHilbMod}
Let $(E,\rho)$ be a $G$-Hilbert A-module, and let $(F,\eta)$ be a $G$-Hilbert $A$-submodule. We say that
$(F,\eta)$ is \emph{$G$-compactly contained} in $(E,\rho)$,
if there exists $b\in \K(E)$ whose restriction to $F$ is $\id_F$.
\end{df}

We claim that the operator $b$ in the definition above can be taken to be contractive and to belong to the fixed point
algebra $\K(E)^G$. Contractivity is easy to arrange: simply divide by $\|b\|$. To choose $b$ in the fixed point algebra, 
first note that if $\xi\in F$, then
\[(g\cdot b)(\xi)=\rho_g(b(\rho_{g^{-1}}(\xi)))=\xi.\]
With $dg$ denoting the normalized Haar measure on $G$, it follows that $b'=\int\limits_G (g\cdot b)\ dg$
is invariant and its restriction to $F$ is the identity.

\begin{df}\label{equivalences for Hilbert modules}
Let $(E,\rho)$ and $(F,\eta)$ be $G$-Hilbert A-modules. We say that $(E,\rho)$ is \emph{$G$-Cuntz subequivalent}
to $(F,\eta)$, and denote this by $(E,\rho)\precsim_G (F,\eta)$, if every compactly contained $G$-Hilbert
submodule of $(E,\rho)$ is unitarily equivalent to a $G$-Hilbert submodule of $(F,\eta)$.

We say that $(E,\rho)$ is \emph{$G$-Cuntz equivalent} to $(F,\eta)$, and denote this by $(E,\rho)\sim_G (F,\eta)$,
if $(E,\rho)\precsim_G (F,\eta)$ and $(F, \mu)\precsim_G (E, \nu)$. The $G$-Cuntz equivalence class of the
$G$-Hilbert $A$-module $(E,\rho)$ is denoted by $[(E,\rho)]$.

We denote by $\Cu^G_\Hi(A,\alpha)$ the set of Cuntz equivalence classes of $G$-Hilbert $A$-modules.
\end{df}

It is easy to check that the direct sum of $G$-Hilbert $A$-modules induces a well defined operation on
$\Cu^G_\Hi(A,\alpha)$. Endow $\Cu^G_\Hi(A,\alpha)$ with the partial order given by $[(E,\rho)]\le [(F,\eta)]$
if $(E,\rho)\precsim_G (F,\eta)$. With this structure, it is clear that $\Cu^G_\Hi(A,\alpha)$ is a partially
ordered abelian semigroup.



We now define a $\Cu(G)$-semimodule structure on $\Cu^G_\Hi(A,\alpha)$. For $[(E,\rho)]\in \Cu^G_\Hi(A,\alpha)$ and $[\mu]\in \Cu(G)$, we set
$$[\mu]\cdot [(E,\rho)]=[(\Hi_\mu\otimes E, \mu\otimes \rho)].$$

Similarly, $\Cu(\K(\Hi_A)^G)$ has
a natural $\Cu(G)$-semimodule structure (see \autoref{df:H_A} for the definition of $\Hi_A$).
Let $a\in \K(\Hi_A)^G$ be a positive element. For a separable
unitary representation $(\Hi_\mu,\mu)$ of $G$, let $s_\mu\in \K(\Hi_\mu)^G$ be a strictly
positive element. Identify $\Hi_\mu\otimes\Hi_A$ with a submodule of $\Hi_A$ using the product
in $\Cu(G)$, and set
\[[\mu]\cdot [a]=[s_\mu\otimes a].\]

Let $(\Hi_\mu,\mu)$ be a separable unitary representation of $G$.
Then $(\Hi_\mu,\mu)$ is unitarily equivalent to a subrepresentation of $(\Hi_\C, \pi_\C)$, and
hence there exists an operator
\[V_{\mu}=V_{\mu, \pi_A}\in \mathcal{L}(\Hi_\mu\otimes A, \Hi_A)^G\]
satisfying $V_{\mu}^*V_{\mu}=\id_{\Hi_\mu}$.

Define a map $\chi\colon \Cu^G(A,\alpha)\to \Cu(\K(\Hi_A)^G)$ as follows. Given a separable unitary
representation $(\Hi_\mu,\mu)$ of $G$, and given a positive element $a\in (\K(\Hi_\mu)\otimes A)^G$,
set
\[\chi([a]_G)=[V_{\mu} aV_{\mu}^*].\]


\begin{prop}\label{prop: chi}
The map $\chi\colon \Cu^G(A,\alpha)\to \Cu(\K(\Hi_A)^G)$, described above, is well defined.
Moreover, it is an isomorphism in $\CCu^G$.
\end{prop}
\begin{proof}
We divide the proof into a number of claims.

\emph{Claim: $\chi$ is well defined, and it preserves the order.}
Let $(\Hi_\mu,\mu)$ and $(\Hi_\nu,\nu)$ be separable unitary representations of $G$, and let
$a\in (\K(\Hi_\mu)\otimes A)^G$ and $b\in (\K(\Hi_\nu)\otimes A)^G$ satisfy $a\precsim_G b$.
Then there exists a sequence $(d_n)_{n\in \N}$ in $(\K(\Hi_\nu,\Hi_\mu)\otimes A)^G$, such that
$\lim\limits_{n\to\I}\|d_nbd_n^*- a\|=0$. It follows that $V_\mu d_nV_\nu^*$ belongs to $\K(\Hi_A)^G$,
and
\[(V_\mu d_nV_\nu^*)(V_\nu bV_\nu^*)(V_\mu d_nV_\nu^*)^*=V_\mu d_nbd_nV_\mu^*\to V_\mu  aV_\mu^*,\]
in the norm of $(\K(\Hi_\mu)\otimes A)^G$, as $n\to\I$.
This shows that $ V_\mu aV_\mu^*\precsim V_\nu bV_\nu^*$, and the claim is proved.

\emph{Claim: $\chi$ is an order embedding.}
Let $(\Hi_\mu,\mu)$ and $(\Hi_\nu,\nu)$ be separable unitary representations of $G$, and
let $a\in (\K(\Hi_\mu)\otimes A)^G$ and $b\in (\K(\Hi_\nu)\otimes A)^G$ satisfy
$\chi([a]_G)\le \chi([b]_G)$. Set $a'=V_\mu aV_\mu^*$ and $b'=V_\nu bV_\nu^*$.
Then there exists a sequence $(d_n')_{n\in \N}$ in $\K(\Hi_A)^G$ such that
$\lim\limits_{n\to\I}\|d_n'b'(d_n')^*- a'\|=0$.

For each $n\in\N$, set $E_n=\overline{\mbox{span}}\left(d_n'(\Hi_A)\cup (d_n')^*(\Hi_A)\right)$. Then $E_n$ is a
countably generated $G$-Hilbert $A$-module by \autoref{lem: range}. Use \autoref{lem: submodules of H_A}
to choose a separable subrepresentation $(\Hi_n, (\pi_\C)_{\Hi_n})$ of $(\Hi_\C, \pi_\C)$,
satisfying $E\subseteq \Hi_n\otimes A\subseteq \Hi_A$ as $G$-Hilbert $A$-modules.
Let $W_{\mu, \pi_A}\in \mathcal{B}(\Hi_\mu)$ and $W_{\nu, \pi_A}\in \mathcal{B}(\Hi_\nu)$ be the
partial isometries implementing the isomorphisms of $\Hi_\mu$ and $\Hi_\nu$ with Hilbert
subspaces $\Hi_\mu'$ and $\Hi_\nu'$ of $\Hi_\C$, respectively. Set
\[\Hi=\overline{\mbox{span}}\left(\Hi_\mu'\cup \Hi_\nu'\cup \bigcup_{n\in\N}\Hi_n\right)\subseteq \Hi_\C.\]
Then $\Hi$ is separable and the operators $a'$, $b'$, $d_n'$, and $(d_n')^*$, for $n\in\N$, map
$\Hi\otimes A$ to itself. Moreover, the restrictions $a''$, $b''$, $d_n''$ of $a'$, $b'$, $d_n'$, for
$n\in\N$, to $\Hi\otimes A$, belong to $\K(\Hi\otimes A)^G$. Moreover, we have
\[\lim_{n\to\I}\|d_n''b''(d_n'')^*- a''\|=0,\]
and thus $a''\precsim b''$ in $\K(\Hi \otimes A)^G$. Consequently, $a''\precsim_G b''$.
\autoref{lem:EquivHs} implies that $a\sim_G a''$ and $b\sim_G b''$. We
conclude that $a\precsim_G b$, as desired.

\emph{Claim: $\chi$ is surjective.} Let $a\in \K(\Hi_A)^G$. By \autoref{lem: submodules of H_A}, there exists
a subrepresentation $(\Hi_\mu, \mu)$ of $(\Hi_\C, \pi_\C)$ such that
$\overline{a(\Hi_A)}\subseteq \Hi_\mu\otimes A$ as $G$-Hilbert $A$-modules. Let $a'$ be the restriction of $a$ to $\Hi_\mu\otimes A$. It is then clear that $\chi([a']_G)=[a]$, so the claim is proved.

It follows that $\chi$ is a $\CCu$-isomorphism.

\emph{Claim: $\chi$ is a $\Cu(G)$-semimodule morphism (and hence a $\CCu^G$-isomorphism).}
It is enough to check that $\chi$ preserves the $\Cu(G)$-action. This is immediate from the definitions.
\end{proof}





For $a\in \K(\Hi_A)^G$, we denote by $\Hi_{A,a}$ the $G$-Hilbert $A$-module $\overline{a(\Hi_A)}$,
and we let $\pi_{A,a}$ be the compression of $\pi_A$ to $\overline{a(\Hi_A)}$.

\begin{thm}\label{thm:tau}
Suppose that $G$ is second countable.
Then the map
\[\tau\colon \Cu(\K(\Hi_A)^G)\to \Cu^G_\Hi(A,\alpha),\]
defined by $\tau([a])=[(\Hi_{A,a}, \pi_{A,a})]$
for a positive element $a\in \K\otimes (\K(\Hi_A)^G)$, is a well defined natural isomorphism in $\CCu^G$.
\end{thm}
\begin{proof}
We divide the proof into a number of claims.

\emph{Claim: $\tau$ is well defined and it preserves the partial order.} To show this, it suffices to prove
that if $a,b\in\K(\Hi_A)^G$ are positive elements with $a\precsim_G b$, then
$$(\Hi_{A,a}, \pi_{A,a})\precsim_G (\Hi_{A,b}, \pi_{A,b}).$$
(See \autoref{def: Hilbert subequivalence}.)

Let $(E, (\pi_A)|_E)$ be a countably generated $G$-Hilbert $A$-module which is compactly contained
in $(\Hi_{A,a}, \pi_{A,a})$. By \autoref{lem: submodules}, there exists
\[c\in \K\left(\Hi_{A,a}\right)^G\cong \left(\overline{a(\K(\Hi_{\C})\otimes A)a}\right)^{G}\]
such that $c|_E$ is strictly positive and $\overline{c(\Hi_{A,a})}=E$. Then $\K(E)\cong \overline{c(\K(\Hi_{\C})\otimes A)c}$.
Use the definition of compact containment of $G$-Hilbert
$A$-modules (\autoref{df:CptContHilbMod}) to choose
$d\in \left(\overline{a(\K(\Hi_{\C})\otimes A)a}\right)^{G}$ with $\|d\|=1$ and $dc=c$. In particular, we
have $(d+c-1)_+=c$ and $d+c\precsim_G a$. Apply \autoref{lem: G Cuntz relation} with $\varepsilon=1$ to the elements
$d+c\precsim_G b$ to find $f\in \K\left(\Hi_{A,b}, \Hi_{A,a}\right)^G$ such that $(d+c-1)_+=fbf^*$.
Set $x=b^\frac{1}{2}f^*$, which is an element in $\K\left(\Hi_{A,a}, \Hi_{A,b}\right)^G$.
Since $(d+c-1)_+=c$, we have
\[x^*x=c \mbox{ and } xx^*\in \left(\overline{b(\K(\Hi_{\C})\otimes A)b}\right)^G.\]

Set $F=\overline{x(E)}$, and let $y\colon E\to F$ be the operator obtained from $x$ by restricting
its domain to $E$ and its codomain to $F$. Since $x$ is invariant, we have
$y\in \mathcal{L}(E, F)^G$. It is clear that $y$ has dense range. Moreover,
$y^*=(x^*)|_F\in\mathcal{L}(F, E)^G$, and hence
\[\overline{y^*(F)}=\overline{y^*(\overline{x(E)})}=\overline{y^*(x(E))}=\overline{x^*x(E)}=\overline{c(E)}=E.\]

It follows that both $y$ and $y^*$ have dense range. By \cite[Proposition 3.8]{Lan_Book}, it
follows that $E$ and $F$ are unitarily equivalent. Moreover, it can be seen from the proof of that
proposition that the unitary can be chosen in $\mathcal{L}_G(E, F)$. This shows that $(E,(\pi_A)|_E)$
is $G$-equivalent to a submodule of $(\Hi_{A,b},\pi_{A,b})$, as desired. This proves the claim.

\emph{Claim: $\tau$ is an order embedding.} Let $a,b\in \K(\Hi_A)^G$ satisfy
$$\left(\Hi_{A,a}, \pi_{A, a}\right)\precsim_G \left(\Hi_{A,b}, \pi_{A,b}\right).$$
Let $\varepsilon>0$ and let $f\in \mathrm C_0(0,\|a\|]$ be a function that is linear on
$[0,\varepsilon]$ and constant equal to $1$ on $[\varepsilon, \|a\|]$. Then $f(a)$ belongs to
$\left(\overline{a \K(\Hi_A) a}\right)^G$ and satisfies $f(a)(a-\varepsilon)_+=(a-\varepsilon)_+$. It follows that
$\left(\Hi_{A,(a-\varepsilon)_+}, \pi_{A,(a-\varepsilon)_+}\right)$ is compactly contained in
$\left(\Hi_{A,a}, \pi_{A, a}\right)$, so there exists an equivariant unitary
\[U\colon\left(\Hi_{A,(a-\varepsilon)_+}, \pi_{A,(a-\varepsilon)_+}\right)\to \left(\Hi_{A,b}, \pi_{A,b}\right).\]

Set $x=(a-\varepsilon)_+U^*$, which is an element in $\K(\Hi_{A, (a-\varepsilon)_+}, \overline{b(\Hi_A)})$.
Then
$$(a-\varepsilon)_+=xx^* \mbox{ and } x^*x=U(a-\varepsilon)_+U^*\in \K(\overline{b\Hi_A})^G.$$

It follows that $\lim\limits_{n\to\I} \left\|b^\frac{1}{n}x^*xb^{\frac 1 n}- x^*x\right\|=0$.
Therefore $x^*x\precsim_G b$. Since we also have $(a-\varepsilon)_+=xx^*\sim_G x^*x$, it follows that
$(a-\varepsilon)_+\precsim_G b$. Since $\varepsilon>0$ is arbitrary, we conclude that
\[[a]=\sup\limits_{\varepsilon>0} [(a-\varepsilon)_+]\le [b],\]
and the claim is proved.

\emph{Claim: $\tau$ is surjective.} Let $(E, (\pi_A)|_E)$ be a countably generated $G$-Hilbert $A$-module.
Since $G$ is assumed to be second countable, $(E, (\pi_A)|_E)$ is isomorphic to a $G$-Hilbert submodule
of $(\Hi_A, \pi_A)$, by \cite[Theorem 2]{Kas_HilbMods}.
Use \autoref{lem: submodules} to find $a\in \K(\Hi_A)^G$ such that
$(E, (\pi_A)|_E)\cong (\Hi_{A,a}, \pi_{A,a})$.
It follows that
\[\tau([a])=[(\Hi_{A,a}, \pi_{A,a})]=[(E, (\pi_A)|_E)],\]
and the claim follows.

We deduce that $\tau$ is an isomorphism in $\CCu$.

\emph{Claim: $\tau$ is a $\Cu(G)$-semimodule morphism (and hence an isomorphism in $\CCu^G$).}
We only check that $\tau$ preserves the $\Cu(G)$-action.
Let $(\Hi_\mu,\mu)$ be a separable unitary representation of $G$, and let $s_\mu$ be a strictly
positive element in $\K(\Hi_{\mu})^G$. For $a\in \K(\Hi_A)^G$, we have
$$\tau([\mu]\cdot [a])=\tau([s_\mu\otimes a])=[\overline{(s_\mu\otimes a)(\Hi_\mu\otimes \Hi_A)}]=[\Hi_\mu\otimes \Hi_{A,a}]=[\mu]\cdot [\Hi_{A,a}].$$
This concludes the proof of the claim and of the theorem.
\end{proof}

The following is the main result of this section.

\begin{cor} \label{cor:HilbModIsomObject}
Suppose that $G$ is second countable. Then there exists a natural $\CCu^G$-isomorphism
\[\delta\colon \Cu^G(A,\alpha)\to \Cu^G_\Hi(A,\alpha).\]
\end{cor}
\begin{proof}
By \autoref{thm:tau} and \autoref{prop: chi}, the map $\delta=\tau\circ\chi$ is the desired
$\CCu^G$-isomorphism.
\end{proof}

\section{Julg's theorem and the \texorpdfstring{$\Cu(G)$}{Cu(G)}-semimodule structure on
\texorpdfstring{$\Cu(A\rtimes_\alpha G)$}{Cu(AtimesG)}}

The goal of this section is to prove that if $\alpha\colon G\to\Aut(A)$ is a continuous
action of a compact group $G$ on a \ca\ $A$, then there exists a natural $\CCu$-isomorphism between
its equivariant Cuntz semigroup $\Cu^G(A,\alpha)$ and $\Cu(A\rtimes_\alpha G)$; see
\autoref{Thm:JulgCuG}. This isomorphism allows us to endow $\Cu(A\rtimes_\alpha G)$ with a canonical
$\Cu(G)$-semimodule structure, and we compute it explicitly in \autoref{thm:JulgWithCuG}.

When $G$ is abelian, this
semimodule structure is particularly easy to describe: it is given by the dual action of
$\alpha$; see \autoref{prop: semimodule structure when G is abelian}. We will prove these
results using the equivariant Hilbert module picture of $\Cu^G(A,\alpha)$ studied in the
previous section.

\subsection{Julg's Theorem}

For the rest of this section, we fix a compact group $G$, a \ca\ $A$, and a continuous action
$\alpha\colon G\to\Aut(A)$. The goal of this section is to prove the Cuntz semigroup analog of Julg's
theorem; see \autoref{Thm:JulgCuG}. Most of the work has already been done in the previous section,
and the only missing ingredients are \autoref{rem: theta}, which is essentially
the Peter-Weyl theorem, and \autoref{prop:FixPtCoaction}, which is noncommutative
duality.

\vspace{0.3cm}

Let $L^2(G)$ denote the Hilbert space of square integrable functions on $G$ with respect to its
normalized Haar measure,
and let $\lambda\colon G\to \U(L^2(G))$ denote the left regular representation.

\begin{rem}\label{rem: theta}
By the Peter-Weyl Theorem (\cite[Theorem 5.12]{Fol_Book}), the $G$-Hilbert module $(\Hi_\C, \pi_\C)$ is
unitarily equivalent (see \autoref{def: Hilbert subequivalence}) to
\[(\ell^2(\N)\otimes L^2(G), \id_{\ell^2(\N)}\otimes \lambda).\]
Therefore there exists an equivariant $\ast$-isomorphism
\[(\K(\Hi_A),\pi_A)\to (\K(\ell^2(\N)\otimes L^2(G)\otimes A),\Ad(\id_{\ell^2(\N)}\otimes \lambda\otimes\alpha)).\]
It follows that there exists a natural $\ast$-isomorphism
\[\theta\colon \K(\Hi_A)^G\to \K(\ell^2(\N)\otimes L^2(G)\otimes A)^G.\]
\end{rem}

The following result is standard, and it is a consequence of Landstad's duality. See, for example,
Theorem~2.7 in~\cite{KalOmlQui_ThreeVersions}, and specifically Example~2.9
in~\cite{KalOmlQui_ThreeVersions}. (The result can also be derived from Katayama's duality;
see Theorem~8 in~\cite{Kat_TakesakisDuality}.)

\begin{prop}\label{prop:FixPtCoaction}
Let $G$ be a locally compact group, let $B$ be a \ca, and let $\delta\colon B\to M(B\otimes C^*(G))$ be
a normal coaction. Denote by $B\rtimes_\delta G$ the corresponding cocrossed product, and by
$\widehat{\delta}\colon G\to\Aut(B\rtimes_\delta G)$ the dual action. Then there exists a canonical
$\ast$-isomorphism
\[\psi\colon (B\rtimes_\delta G)^{\widehat{\delta}}\to B,\]
which is moreover $(\delta^{j_G},\delta)$-equivariant (see Definition~2.8
in~\cite{KalOmlQui_ThreeVersions}).\end{prop}

The following result is an analog of Julg's Theorem (\cite{Jul_EqKthy}; see also \cite[Theorem 2.6.1]{Phi_Book})
for the equivariant Cuntz semigroup. It asserts that the equivariant Cuntz semigroup is naturally isomorphic, in
the category $\CCu$, to the Cuntz semigroup of the crossed product. The isomorphism is obtained as a composition
of the isomorphisms from \autoref{prop: chi}, \autoref{rem: theta}, Appendix~6 in~\cite{CowEllIva} (this is
stability of the functor $\Cu$), and \autoref{prop:FixPtCoaction}; see the commutative diagram at the end
of the proof of \autoref{Thm:JulgCuG}. In some applications of this theorem, particularly in \autoref{thm:JulgWithCuG},
the explicit construction of the isomorphism will be relevant.

\begin{thm}\label{Thm:JulgCuG}
Let $G$ be a compact group, let $A$ be a \ca, and let $\alpha\colon G\to\Aut(A)$ be a continuous action.
Then there exists a natural $\CCu$-isomorphism
\[\sigma\colon \Cu^G(A,\alpha)\to\Cu(A\rtimes_\alpha G).\]
\end{thm}
\begin{proof}
Endow $\K(L^2(G))$ with the action of conjugation by the left regular representation of $G$, endow
$\K(\ell^2(\N))$ with the trivial $G$-action, and endow $\K(L^2(G))\otimes A$ and
$\K(\ell^2(\N))\otimes \K(L^2(G))\otimes A$ with the corresponding tensor product actions.
Then there exists a natural identification
\[(\K(\ell^2(\N))\otimes \K(L^2(G))\otimes A)^G= \K(\ell^2(\N))\otimes (\K(L^2(G))\otimes A)^G.\]
Since $\Cu$ is a stable functor (see Appendix~6 in~\cite{CowEllIva}), there exists a natural $\CCu$-isomorphism
\[\kappa\colon \Cu(\K(\ell^2(\N)\otimes L^2(G)\otimes A)^G)\to \Cu(\K(L^2(G)\otimes A)^G).\]

By \autoref{rem: theta}, there exists a natural $\ast$-isomorphism
\[\theta\colon \K(\Hi_A)^G\to \K(\ell^2(\N)\otimes L^2(G)\otimes A)^G.\]

Denote by $\psi\colon (\K(L^2(G))\otimes A)^G\to A\rtimes_\alpha G$ the natural $\ast$-isomorphism
obtained from \autoref{prop:FixPtCoaction} for $B=A\rtimes_\alpha G$ and $\delta=\widehat{\alpha}$.
(Recall that coactions of compact groups are automatically normal; see, for example,
the end of the proof of Lemma~4.8 in \cite{Gar_CptRok}.)
With $\chi\colon \Cu^G(A,\alpha)\to \Cu(\K(\Hi_A)^G)$ denoting the natural $\CCu$-isomorphism
given by \autoref{prop: chi}, define $\sigma$ to be the following composition:
\beqa\xymatrix{
\Cu^G(A,\alpha)\ar@{-->}[dr]_-{\sigma}\ar[r]^-\chi &\Cu(\K(\Hi_A)^G)\ar[r]^-{\Cu(\theta)} & \Cu((\K(\ell^2(\N))\otimes\K(L^2(G))\otimes A)^G)
\ar[d]^-\kappa \\
& \Cu(A\rtimes_\alpha G)& \Cu((\K(L^2(G))\otimes A)^G) \ar[l]^-{\Cu(\psi)}.}\eeqa

It is clear that $\sigma$ is a natural isomorphism in the category $\CCu$.\end{proof}

\subsection{Semimodule structure on the crossed product}

\autoref{Thm:JulgCuG} provides an isomorphism $\Cu^G(A,\alpha)\cong \Cu(A\rtimes_\alpha G)$ as $\CCu$-semigroups.
We can give $\Cu(A\rtimes_\alpha G)$ the unique $\Cu(G)$-semimodule structure that makes this
isomorphism into a $\CCu^G$-isomorphism. To make this result useful, we must describe this semimodule
structure internally. This takes some work, and we will need a series of intermediate results. This
subsection is based, to some extent, on \cite{Phi_Book}. We assume that $G$ is a second countable group.

The main technical difficulties are the absence of short exact sequences in the context
of semigroups, and the fact that in the construction of the equivariant Cuntz
semigroup, representations of the group $G$ on infinite dimensional Hilbert spaces are allowed.
Compactness of $G$ is crucial in overcoming the latter.

We need a standard definition. For a \ca\ $A$, we denote by $M(A)$ its multiplier algebra.

\begin{df}\label{df:ExtEq}
Let $\alpha,\beta\colon G\to \Aut(A)$ be continuous actions of a locally compact group $G$ on a \ca\ $A$.
We say that $\alpha$ and $\beta$ are \emph{cocycle equivalent}, if there exists a function
$\omega\colon G\to \U(M(A))$ satisfying:
\be
\item $\omega_{gh}=\omega_g\beta_g(\omega_h)$ for all $g,h\in G$;
\item $\alpha_g=\Ad(\omega_g)\circ\beta_g$ for all $g\in G$;
\item For $a\in A$, the map $G\to A$ given by $g\mapsto \omega_g a$ is continuous.\ee
\end{df}

Condition (1) in the definition above ensures that if $\alpha$ is defined by (2), then $\alpha_g\circ\alpha_h=\alpha_{gh}$
for all $g,h\in G$.

\begin{rem} It is well known that cocycle equivalent actions have isomorphic associated
crossed products. Nevertheless, it does not follow from this that cocycle
equivalent actions have isomorphic equivariant Cuntz semigroups, because we do not know
how to compute the $\Cu(G)$-semimodule structure of the crossed products. (That this is
indeed the case is a consequence of \autoref{thm:JulgWithCuG}.)
In order to prove the result, however, we do need to know that some specific cocycle
conjugate actions yield isomorphic equivariant Cuntz semigroups;
see \autoref{prop: ext equivalence invariance}.
\end{rem}

For the rest of the subsection, we fix a continuous action $\alpha\colon G\to\Aut(A)$ of a compact
group $G$ on a \ca\ $A$.

\begin{prop}\label{prop: ext equivalence invariance}
Let $\beta$ be an action of $G$ on $A$ which is cocycle equivalent to $\alpha$. Suppose that
$A$ has an increasing countable approximate identity consisting of projections which are invariant for both
$\alpha$ and $\beta$. Then there exists a natural $\CCu^G$-isomorphism
$\ep\colon \Cu^G(A,\alpha)\to \Cu^G(A,\beta)$.

When $A$ is unital, the map $\ep$ can be described as follows. Choose a cocycle $\omega\colon G\to \U(A)$
such that $\alpha_g=\Ad(\omega_g)\circ\beta_g$ for all $g\in G$. For a countably generated
$G$-Hilbert $(A,\alpha)$-module $(E,\rho)$, the map $\ep$ is given by $\ep([(E,\rho)])=([E,\rho^\omega)]$,
where $\rho^\omega\colon G\to\U(E)$ is given by $\rho^\omega_g(x)=\rho_g(x)\omega_g$ for all $g\in G$ and
all $x$ in $E$.
\end{prop}
\begin{proof}
Suppose first that $A$ is unital. We must first argue why the map $\ep$ described in the statement is well-defined,
and why it is an isomorphism in $\CCu^G$. We fix a cocycle $\omega\colon G\to \U(A)$
such that $\alpha_g=\Ad(\omega_g)\circ\beta_g$ for all $g\in G$, and we fix a countably generated
$G$-Hilbert $(A,\alpha)$-module $(E,\rho)$. Let $\rho^\omega\colon G\to \U(E)$ be defined as in the statement.

We claim that $(E,\rho^\omega)$ is a countably generated $G$-Hilbert $(A,\beta)$-module.
Since $E$ was chosen to be countably generated to begin with, we shall only check that $\rho^\omega$
is compatible with $\beta$ and with the Hilbert module structure. Given $a\in A$, $x\in E$ and
$g\in G$, we have
$$\rho^\omega_g(xa)=\rho_g(xa)\omega_g=\rho_g(x)\alpha_g(a)\omega_g=\rho_g(x)\omega_g\beta_g(a)=\rho^\omega_g(x)\beta_g(a),$$
as desired. Moreover, for $g\in G$ and $x,y\in E$, we have
\begin{align*} \langle \rho^\omega_g(x),\rho^\omega_g(y)\rangle_E &=\langle \rho_g(x)\omega_g,\rho_g(y)\omega_g\rangle_E\\
                                                  &=\omega_g^*\langle \rho_g(x),\rho_g(y)\rangle_E\omega_g\\
                                                  &=(\Ad(\omega_g^*)\circ\alpha_g)(\langle x,y\rangle_E)\\
                                                  &=\beta_g(\langle x,y\rangle_E),\end{align*}
thus proving the claim.

The assignment $(E,\rho)\mapsto (E,\rho^\omega)$ is clearly surjective, and hence every $G$-Hilbert $(A,\beta)$-module has the form $(E,\rho^\omega)$ for some $G$-Hilbert $(A,\alpha)$-module $(E,\rho)$.

We claim that the assignment $(E,\rho)\mapsto (E,\rho^\omega)$ preserves $G$-Cuntz subequivalence.
Let $(E,\rho)$ and $(E',\rho')$ be countably generated $G$-Hilbert $(A,\alpha)$-modules, and suppose that $(E,\rho)\precsim_G (E',\rho')$. If $(F,\eta^\omega)$ is a $G$-Hilbert $(A,\beta)$-module that is compactly contained in $(E,\rho^\omega)$, then it is straightforward to check that $(F,\eta)$ is compactly contained in $(E,\rho)$. If $(F',\eta')$ is a $G$-Hilbert $(A,\alpha)$-module compactly contained in $(E',\rho')$ such that $(F',\eta')\cong (F,\eta)$, then one readily checks that $(F',(\eta')^\omega)\cong (F,\eta^\omega)$. This shows that $(E,\rho^\omega)\precsim_G (E',(\rho')^\omega)$, and proves
the claim.

Denote by $\varphi\colon \Cu^G(A,\alpha)\to \Cu^G(A,\beta)$ the map given by $[(E,\rho)]\mapsto [(E,\rho^\omega)]$, where $\rho^\omega$ is given by $\rho^\omega_g(x)=\rho_g(x)\omega_g$ for all $g\in G$ and all $x$ in $E$. We claim that $\varphi$ is a $\CCu^G$-morphism.

We already showed that $\varphi$ preserves the order. It is also easy to show that is preserves compact containment and suprema of increasing sequences. We show that it is a morphism of $\Cu(G)$-semimodules.
Given a separable unitary representation $(\Hi_\mu,\mu)$ of $G$, we must show that the diagram
\begin{align*} \xymatrix{\Cu^G(A,\alpha)\ar[r]^\varphi\ar[d]_-{[\mu]\cdot} & \Cu^G(A,\beta) \ar[d]^-{[\mu]\cdot }\\
\Cu^G(A,\alpha)\ar[r]_-{\varphi}& \Cu^G(A,\beta)}\end{align*}
commutes. Given a countably generated $G$-Hilbert $(A,\alpha)$-module $(E,\rho)$, we have
\[[(\Hi_\mu,\mu)]\cdot \varphi([(E,\rho)])=\left[(E\otimes \Hi_\mu,\rho^\omega\otimes\mu)\right].\]

On the other hand, the element $ \varphi\left([(\Hi_\mu,\mu)]\cdot[(E,\rho)]\right)$ is represented by the $G$-Hilbert $(A,\beta)$-module $(E\otimes \Hi_\mu,(\rho\otimes\mu)^\omega)$, so it is enough to check that both actions on
$\Hi_\mu\otimes E$ agree.
Let $g\in G$, let $x\in E$ and let $\xi\in \Hi_\mu$. Then
\begin{align*} \left((\rho\otimes\mu)^\omega \right)_g(x\otimes \xi)&=(\rho_g(x)\otimes \mu_g(\xi))\omega_g\\
&=(\rho_g(x)\omega_g)\otimes \mu_g(\xi)\\
&=\rho^\omega_g(x)\otimes\mu_g(\xi)\\
&=(\rho^\omega\otimes\mu)_g(x\otimes\xi),\end{align*}
thus showing that $\varphi$ is a $\CCu^G$-morphism. Since $\varphi$ is clearly bijective, it follows that it is an isomorphism. Naturality is also clear. This proves the unital case.

\vspace{0.3cm}

For the general case, let $(e_n)_{n\in \N}$ be an increasing approximate identity in $A$ consisting of projections
that are invariant for both $\alpha$ and $\beta$. For $n\in\N$, let $\alpha^{(n)}\colon G\to\Aut(e_nAe_n)$ be the action given by $\alpha^{(n)}_g(a)=\alpha_g(a)$ for all $g\in G$ and $a\in e_nAe_n$, and similarly for $\beta^{(n)}\colon G\to\Aut(e_nAe_n)$. We claim that $\alpha^{(n)}$ and $\beta^{(n)}$ are cocycle equivalent for all $n\in\N$.

Choose a cocycle $\omega\colon G\to\U(M(A))$ as in \autoref{df:ExtEq}.
For $g\in G$ and $n\in\N$, one checks that
$$\omega_ge_n \omega_g^*=(\Ad(\omega_g)\circ\alpha_g)(e_n)=\beta_g(e_n)=e_n.$$

Define $\omega^{(n)}\colon G\to \U(e_nAe_n)$ by $\omega^{(n)}_g=e_n\omega_ge_n$ for $g\in G$. One readily checks that $\Ad(\omega^{(n)}_g)\circ\beta^{(n)}_g=\alpha^{(n)}_g$ for all $n\in\N$ and all $g\in G$. The
cocycle condition is also easy to verify, so the claim is proved.

Note that there exist natural equivariant $\ast$-isomorphisms
$$(A,\alpha)=\varinjlim (e_nAe_n,\alpha^{(n)}) \mbox{ and } (A,\beta)=\varinjlim (e_nAe_n,\beta^{(n)}).$$
By \autoref{prop:CuGContinuous}, there exist natural $\Cu^G$-isomorphisms
\[\Cu^G(A,\alpha)\cong \varinjlim \Cu^G(e_nAe_n,\alpha^{(n)}) \mbox{ and } \Cu^G(A,\beta)\cong \varinjlim \Cu^G(e_nAe_n,\beta^{(n)}).\]
For $n\in\N$, denote by $\varphi^{(n)}\colon \Cu^G(e_nAe_n,\alpha^{(n)}) \to \Cu^G(e_nAe_n,\beta^{(n)})$
the natural $\CCu^G$-isomorphism provided by the unital case of this proposition.
Naturality implies that the following diagram in $\CCu^G$ is commutative, where the top row
is exact by Theorem~5 in~\cite{CiuRobSan_CuIdeals}:
\begin{align*}
\xymatrix{
\Cu^G(e_1Ae_1,\alpha^{(1)})\ar[r]\ar[d]_-{\varphi^{(1)}} & \Cu^G(e_2Ae_2,\alpha^{(2)})\ar[r]\ar[d]_-{\varphi^{(2)}} & \cdots\ar[r] & \Cu^G(A,\alpha)\ar@{-->}[d]^-{\varphi}\\
\Cu^G(e_1Ae_1,\beta^{(1)})\ar[r]& \Cu^G(e_2Ae_2,\beta^{(2)})\ar[r] & \cdots\ar[r] & \Cu^G(A,\beta)
.}
\end{align*}
The universal property of the inductive limit in $\CCu^G$ shows that there exists a natural $\CCu^G$-morphism
$\varphi\colon \Cu^G(A,\alpha)\to \Cu^G(A,\beta)$. This map is easily seen to be an isomorphism in $\CCu^G$, so
the proof is complete.
\end{proof}

The following result states that the equivariant Cuntz semigroup is a stable functor, as long as the algebra
of compact operators is given an inner action. The case of the trivial action was already prove in \autoref{prop:CuGStable}.
The result fails in general if the action on the compacts is not inner; see \autoref{eg:NonStableNotInner}.

\begin{cor} \label{cor:StabUnirep}
Let $A$ be a unital $C^*$-algebra, and let $(\Hi_\mu,\mu)$ be a separable unitary representation of $G$.
Let $q$ be any rank one projection on $\Hi_\mu$, and consider the map $\iota_q\colon A\to A\otimes\K(\Hi_\mu)$,
given by $\iota_q(a)=a\otimes q$ for $a\in A$. Let
\[\varepsilon\colon \Cu^G(A\otimes\K(\Hi_\mu),\alpha\otimes\Ad(\mu))\to \Cu^G(A\otimes\K(\Hi_\mu),\alpha\otimes\id_{\K(\Hi_\mu)})\]
be the natural $\CCu^G$-isomorphism given by \autoref{prop: ext equivalence invariance}.
Then
\[\ep^{-1}\circ\Cu^G(\iota_q)\colon \Cu^G(A,\alpha)\to \Cu^G(A\otimes\K(\Hi_\mu),\alpha\otimes\Ad(\mu))\]
is a natural $\CCu^G$-isomorphism. \end{cor}
\begin{proof}
We must first argue why the assumptions of \autoref{prop: ext equivalence invariance} are met, so that the map $\ep$
really does exist. Firxt, since $\mu$ can be decomposed as a direct sum of finite dimensional representations by the Peter-Weyl Theorem,
it follows that $A\otimes \K(\Hi_\mu)$ has a countable approximate identity consisting of projections that are
invariant under both $\alpha\otimes\Ad(\mu)$ and $\alpha\otimes\id_{\K(\Hi_\mu)}$.
And second, the actions $\alpha\otimes\Ad(\mu)$ and $\alpha\otimes\id_{\K(\Hi_\mu)}$ are cocycle equivalent via
the cocycle $\omega\colon G\to \U(A\otimes\K(\Hi_\mu))$ given by $\omega_g=1_A\otimes \mu_g$ for all $g\in G$.
We conclude that \autoref{prop: ext equivalence invariance} applies and that $\ep$ is a natural $\CCu^G$-isomorphism.

By \autoref{prop:CuGStable}, the map
\[\Cu^G(\iota_q)\colon \Cu^G(A,\alpha)\to \Cu^G(A\otimes\K(\Hi_\mu),\alpha\otimes\id_{\K(\Hi_\mu)})\]
is natural $\CCu^G$-isomorphism.
It follows that $\varepsilon^{-1}\circ \Cu^G(\iota_q)$ is a natural isomorphism in $\CCu^G$.
\end{proof}

As mentioned before, the result above fails if the action on $\K(\Hi)$ is not inner, as the following
example shows. We write $\Z_2=\{-1,1\}$.

\begin{eg}\label{eg:NonStableNotInner}
Define an action of $G=\Z_2\times\Z_2$ on $A=M_2$ by letting $(-1,1)$ act by conjugation by $\left(
                                                                                              \begin{array}{cc}
                                                                                                1 & 0 \\
                                                                                                0 & -1 \\
                                                                                              \end{array}
                                                                                            \right),$
and by letting $(1,-1)$ act by conjugation by $\left(
                                                                                              \begin{array}{cc}
                                                                                                0 & 1 \\
                                                                                                1 & 0 \\
                                                                                              \end{array}
                                                                                            \right)$. Denote this action by $\alpha$.
(These unitaries commute up to a sign, so conjugation by them really does define an action of
$\Z_2\times\Z_2$.)
We claim that there is no isomorphism between $\Cu^G(M_2,\alpha)$ and $\Cu^G(M_2,\id_{M_2})$.
For this, it is enough to note that since there is an isomorphism
\[M_2\rtimes_{\id_{M_2}}(\Z_2\times\Z_2)\cong M_2\oplus M_2\oplus M_2\oplus M_2,\]
the semigroup $\Cu^G(M_2,\id_{M_2})$ is isomorphic to $\bigoplus\limits_{j=1}^4\overline{\Z_{\geq 0}}$.
On the other hand, one can compute $\Cu^G(M_2,\alpha)$ with elementary methods. For instance, one easily
checks that $M_2^{\Z_2\times\Z_2}$ is simply the scalar matrices, and a similar computation shows that
\[(M_2\otimes\K(\ell^2(\Z_2\times\Z_2)))^{\alpha\otimes\Ad(\lambda)}\cong M_4.\]
Thus, $\Cu^G(M_2,\alpha)=\overline{\Z_{\geq 0}}$, and the claim follows.

We conclude that the analog of \autoref{cor:StabUnirep}, where $\K(\Hi)$
does not have an inner action, fails in general.
\end{eg}

Denote by $V(G)\subseteq \Cu(G)$ the subsemigroup consisting
of equivalence classes of \emph{finite dimensional} unitary representations of $G$ (so that the
Grothendieck group of $V(G)$ is the represenation ring $R(G)$).
The following proposition describes the action of $V(G)$ on $\Cu(A\rtimes_\alpha G)$. Upon taking
the supremum of an increasing sequence of finite dimensional representations, this result also
leads to a method for computing the action of an arbitrary element in $\Cu(G)$. 
For use in its proof, we recall from the proof of \autoref{prop: ext equivalence invariance} that
for $A$ unital and cocycle equivalent actions $\alpha$ and $\alpha^\omega$ of $G$ on $A$,
the $\CCu^G$-isomorphism
\[\varepsilon\colon \Cu^G(A,\alpha)\to \Cu^G(A,\alpha^\omega)\]
was constructed using $G$-Hilbert modules, and $\ep([(E,\rho)])=[(E,\rho^\omega)]$, where the
unitary representation
$\rho^\omega\colon G\to \U(E)$ is given by $\rho^\omega_g=\rho_g\omega_g$
for $g\in G$.

\begin{prop}\label{prop:SemimodStrBad}
Suppose that $A$ is unital and let $(\Hi_\mu,\mu)$ be a separable unitary representation of $G$.
Let $p,q\in\K(\Hi_\mu)$ be a $\mu$-invariant projections, with $q$ of rank one.
Define equivariant $\ast$-homomorphisms
$$\varphi\colon (A,\alpha) \to (A\otimes\K(\Hi_\mu),\alpha\otimes\Ad(\mu)) \mbox{ and } \psi\colon (A,\alpha) \to (A\otimes\K(\Hi_\mu),\alpha\otimes\id_{\K(\Hi_\mu)})$$
by $\varphi(a)= a\otimes p$ and $\psi(a)=a\otimes q$ for all $a\in A$. Let
$$\varepsilon\colon \Cu^G(A\otimes\K(\Hi_\mu),\alpha\otimes\Ad(\mu))\to \Cu^G(A\otimes\K(\Hi_\mu),\alpha\otimes\id_{\K(\Hi_\mu)})$$
be the cocycle equivalence isomorphism given by \autoref{prop: ext equivalence invariance}, associated
to the cocycle equivalence $g\mapsto 1_A\otimes\mu_g$ between $\alpha\otimes\Ad(\mu)$ and $\alpha\otimes\id_{\K(\Hi_mu)}$. Then the following $\CCu^G$-diagram commutes:

\begin{align*}
\xymatrix{
\Cu^G(A,\alpha)\ar[rr]^-{[\mu|_{p\mathcal{H}}]\cdot }\ar[d]_-{\Cu^G(\varphi)} & & \Cu^G(A,\alpha)\ar[d]^-{\Cu^G(\psi)}\\
\Cu^G(A\otimes\K(\Hi_\mu),\alpha\otimes\Ad(\mu))\ar[rr]_-{\varepsilon} & &
\Cu^G(A\otimes\K(\Hi_\mu),\alpha\otimes\id_{\K(\Hi_\mu)}).
}
\end{align*}
\end{prop}
\begin{proof}
Observe that $\Cu^G(\psi)$ is a natural isomorphism in $\CCu^G$ by \autoref{prop:CuGStable}.
We only need to check that
\[\Cu^G(\psi)^{-1}\circ \varepsilon\circ \Cu^G(\varphi)\colon \Cu^G(A,\alpha)\to \Cu^G(A,\alpha)\]
is multiplication by $[\mu|_{p\Hi_\mu}]$.

Denote by $(\Hi_A,\pi_A)$ the canonical countably generated $G$-Hilbert $(A,\alpha)$-module from \autoref{df:H_A}.
Let $(E,\rho)$ be a countably generated $G$-Hilbert $(A,\alpha)$-module. Use Kasparov's absorption theorem
(Theorem~2 in~\cite{Kas_HilbMods}) to choose a $\pi_A$-invariant projection $r$ in $\mathcal{L}(\Hi_A)^G$ such
that $(E,\rho)\cong (r\Hi_A,\pi_A|_{r\Hi_A})$. Then
\begin{align*}
\Cu^G(\varphi)([(E,\rho)])&=\left[\left((r\otimes p)\left(\K(\Hi_A)\otimes \K(\Hi_\mu)\right),
\pi_A|_{r\Hi_A}\otimes\mu|_{p\Hi_\mu} \right)\right]\\
&=[E\otimes p\K(\Hi_\mu),\rho\otimes \mu|_{p\Hi_\mu}].\end{align*}

Denote the unitary representation $\Ad(\mu)^\mu$ (see the comments before this proposition)
of $G$ on $p\K(\Hi_\mu)$ by $\widetilde{\mu}$. In other words,
$\widetilde{\mu}_g(x)=(\mu_g x\mu_g^*)\mu_g=\mu_g x$ for all $g\in G$ and all $x\in p\K(\Hi_\mu)$.
By the computation above, we have
\[(\varepsilon \circ\Cu^G(\varphi))([E,\rho])=[(E\otimes p\K(\Hi_\mu),\rho\otimes\widetilde{\mu})].\]
We must compare the class of $(E\otimes p\K(\Hi_\mu),\rho\otimes\widetilde{\mu})$ with the class of
$\Cu^G(\psi)([E\otimes p\K(\Hi_\mu), \rho\otimes \Ad(\mu)])$, and show that they agree.

One checks that $\Cu^G(\psi)\left([(p\Hi_\mu, \mu|_{p\Hi_\mu})]\cdot [(E,\rho)]\right)$ is represented by
$$\left(E\otimes p\Hi_\mu\otimes q\K(\Hi_\mu),\rho\otimes\Ad(\mu)\otimes\id_{\K(\Hi_\mu)}\right).$$
Evaluating the maps in the diagram in the statement at $(E,\rho)$, we get
\begin{align*}
\xymatrix{
(E,\rho)\ar@{|->}[rr]\ar@{|->}[dd] && \left(E\otimes p\Hi_\mu,\rho\otimes\mu\right) \ar@{|->}[d]\\
 && \left(E\otimes p\Hi_\mu\otimes q\K(\Hi_\mu),\rho\otimes\Ad(\mu)\otimes\id_\K\right)\\
\left(E\otimes p\K(\Hi_\mu),\rho\otimes\Ad(\mu)\right)\ar@{|->}[rr] & & \left(E\otimes p\K(\Hi_\mu),\rho\otimes\widetilde\mu\right).
}\end{align*}

It is therefore enough to check that
$$\left(p\Hi_\mu\otimes q\K(\Hi_\mu),\mu|_{p\Hi_\mu}\otimes \id_{q\K(\Hi_\mu)}\right)\cong \left(p\K(\Hi_\mu),\widetilde\mu\right)$$
as $G$-Hilbert $\K(\Hi_\mu)$-modules. Fix a unit vector $\xi^{(0)}\in \Hi_\mu$ in the range of $q$ and define
$$\sigma\colon p\Hi_\mu\otimes q\K(\Hi_\mu)\to p\K(\Hi_\mu)$$
by $\sigma(\xi\otimes b)(\eta)=\langle b^*(\xi^{(0)}),\eta\rangle \xi$ for all $\xi\in p\Hi_\mu$, for all
$b\in q\K(\Hi_\mu)$ and for all $\eta\in \Hi_\mu$, and extended linearly and continuously.
Note that
$$\left(p\circ(\sigma(\xi\otimes b))\right)(\eta)=\langle b^*(\xi^{(0)}),\eta\rangle p(\xi)=\langle b^*(\xi^{(0)}),\eta\rangle \xi,$$
so the range of $\sigma$ is really contained in $p\K(\Hi_\mu)$.

\emph{Claim: $\sigma$ is injective.} Since $p$ is a projection in $\K(\Hi_\mu)$, it has finite rank. For
$m=\dim(p\Hi_\mu)$, choose an orthonormal basis $\xi_1,\ldots,\xi_m$ of $p\Hi_\mu$.
Given $x\in p\Hi_\mu\otimes q\K(\Hi_\mu)$,
there exist $b_1,\ldots, b_m\in\K(\Hi_\mu)$ such that $x=\sum\limits_{j=1}^m \xi_j\otimes b_j$.
Assume that $\sigma(x)=0$. By orthogonality of the vectors $\xi_j$,
it follows that $\sigma(\xi_j\otimes b_j)=0$ for all $j=1,\ldots,m$. Thus
$\langle \eta,b_j^*(\xi^{(0)})\rangle =0$ for all $\eta\in \Hi_\mu$, and hence $b_j^*(\xi^{(0)})=0$.
This shows that $b_j^*$ vanishes on $\mbox{span}\{\xi^{(0)}\}=q\Hi_\mu$, and so $b_j=0$ for $j=1,\ldots,m$. We
conclude that $x=0$ and $\sigma$ is injective.

\emph{Claim: $\sigma$ is surjective.} Given $a$ in $\K(\Hi_\mu)$, the map $pa\colon \Hi_\mu\to p\Hi_\mu$
is map and has finite rank. For $j=1,\ldots,m$, denote by $p_j\colon \Hi_\mu\to \C\times \xi_j$ the orthogonal
projection. It follows from the Riesz Representation Theorem that there
exist unit vectors $\xi^{(0)}_j$ in the range of $p_j$ and linear maps $c_j\in q_j\K(\Hi_\mu)$, for
$j=1,\ldots,m$, such that
$$pa(\eta)=\sum_{j=1}^m \langle c^*_j(\xi^{(0)}_j),\eta\rangle \xi_j$$
for all $\eta\in \Hi_\mu$. (Recall our convention that inner products are linear on the second
variable.) Since any two rank one projections are unitarily equivalent, it follows that there exist linear maps $b_1,\ldots,b_m\in q\K(\Hi_\mu)$ such that
$$pa(\eta)=\sum_{j=1}^m \langle b^*_j(\xi^{(0)}),\eta\rangle \xi_j$$
and thus $pa=\sigma\left(\sum\limits_{j=1}^m\xi_j\otimes b_j\right)$, showing that $\sigma$ is surjective.

\emph{Claim: $\sigma$ is a $G$-Hilbert $\K(\Hi_\mu)$-homomorphism.} Let $\xi\in p\Hi_\mu$, $b\in q\K(\Hi_\mu)$,
$c\in \K(\Hi_\mu)$, and $\eta\in \Hi_\mu$. Then
\begin{align*} \sigma(\xi\otimes b)(c\eta)&= \langle c^*b^*(\xi^{(0)}),\eta\rangle \xi\\
                                          &= \langle (bc)^*(\xi^{(0)}),\eta\rangle \xi\\
                                          &= \sigma(\xi\otimes bc)(\eta).\end{align*}

Finally, for $g\in G$, $b\in q\K(\Hi_\mu)$, $c\in \K(\Hi_\mu)$, $\xi\in p\Hi_\mu$, and $\eta\in \Hi_\mu$,
one has
\begin{align*}
\sigma\left((\mu|_{p\Hi_\mu}\otimes\id_{q\K(\Hi_\mu)})_g(\xi\otimes b)\right)
&= \sigma\left(\mu_g\xi\otimes b\right)(\eta)\\
&= \langle b^*(\xi^{(0)}),\eta\rangle \mu_g\xi\\
&= \mu_g\langle b^*(\xi^{(0)}),\eta\rangle \xi\\
&= \widetilde\mu_g(\sigma(\xi\otimes b))(\eta),\end{align*}
which shows that $\sigma$ is equivariant. This finishes the proof.
\end{proof}

The above result leads to a method for computing the $\Cu(G)$-semimodule structure on $\Cu(A\rtimes_\alpha G)$.
This description makes essential use of the cocycle equivalence isomorphism $\varepsilon$, and similarly to what happens
with equivariant $K$-theory, it is inconvenient when trying to use it. To remedy this, we give an alternative
description of the $\Cu(G)$-action (\autoref{df: semimod structure}), which, even though it is not as
transparent as the one in \autoref{prop:SemimodStrBad},
has the advantage that all $\CCu$-maps involved are induced by $\ast$-homomorphisms.

Let $(\Hi_\mu,\mu)$ be a \emph{finite dimensional} unitary representation of $G$, and denote by
$\mu^+\colon G\to \mathcal{U}(\mathcal{H}_\mu\oplus\C)$  its direct sum with the trivial
representation on $\C$. Let $p_\mu,q_\mu \in \K(\Hi_\mu\oplus\C)$ be the projections onto
$\Hi_\mu$ and $\C$, respectively.
Define equivariant $\ast$-homomorphism $\iota_{p_\mu},\iota_{q_\mu}\colon A\to A\otimes\K(\Hi_\mu\oplus\C)$ by
$$\iota_{p_\mu}(a)=a\otimes p_{\mu} \mbox{ and } \iota_{q_\mu}(a)=a\otimes p_\C$$
for $a\in A$. Denote by $\widehat{\iota_{p_\mu}}$ and $\widehat{\iota_{q_\mu}}$ the corresponding
maps on the crossed products by $G$.
The $\Cu(\iota_{q_\mu})$ is invertible by \autoref{prop:CuGStable}, and it corresponds to multiplication
by the class of the trivial representation by \autoref{prop:SemimodStrBad}.
By considering these maps at the level of the Cuntz
semigroups, we have
\begin{align*}
\xymatrix@=1.5em{ \Cu(A\rtimes_\alpha G)\ar[rr]^-{\Cu(\widehat{\iota_{p_\mu}})}& & \Cu\left((A\otimes \K(\mathcal{H}_\mu\oplus\C))\rtimes_{\alpha\otimes\Ad(\mu^+)}G\right)\ar@<1ex>[rr]^-{\Cu(\widehat{\iota_{q_\mu}})^{-1}} & &\Cu(A\rtimes_\alpha G)\ar@<1ex>[ll]^-{\Cu(\widehat{\iota_{q_\mu}})}.
}
\end{align*}

\begin{df}\label{df: semimod structure}
Adopt the notation from the discussion above.
We define a $\Cu(G)$-semimodule structure on $\Cu(A\rtimes_\alpha G)$ as follows. For a finite dimensional
representation $(\Hi_\mu,\mu)$, and for $s \in \Cu(A\rtimes_\alpha G)$, we set
$$[\mu]\cdot s=\left(\Cu(\widehat{\iota_{q_\mu}})^{-1}\circ \Cu(\widehat{\iota_{p_\mu}})\right)(s).$$

For an arbitrary separable unitary representation $(\Hi_\nu,\nu)$, use compactness of $G$ to choose
irreducible representations $(\Hi_{\mu_n},\mu_n)$ of $G$ such that
$\nu\cong \bigoplus\limits_{n=1}^\infty \mu_n$. For $m\in\N$, set $\nu_m=\bigoplus\limits_{n=1}^m\mu_n$. For
$s \in \Cu(A\rtimes_\alpha G)$, we set
$$[\nu]\cdot s=\sup\limits_{m\in\N}\left([\nu_m]\cdot s\right).$$
\end{df}

The following lemma shows that the above $\Cu(G)$-semimodule structure is well-defined.

\begin{lma} Let $(\Hi_\nu,\nu)$ be a separable unitary representation of $G$, and find finite
dimensional unitary representations $(\Hi_{\mu_n},\mu_n)$ of $G$ as in \autoref{df: semimod structure}.
For $m\in\N$, set $\nu_m=\bigoplus\limits_{n=1}^m\mu_n$. Let $s\in\Cu(A\rtimes_\alpha G)$.
\begin{enumerate}
\item The sequence $\left([\nu_m]\cdot s\right)_{n\in\N}$ is increasing in $\Cu(A\rtimes_\alpha G)$.
\item The element $[\nu]\cdot s=\sup\limits_{m\in\N}\left([\nu_m]\cdot s\right)$ is independent of
the decomposition $\nu\cong \bigoplus\limits_{n\in\N}\mu_n$.
\end{enumerate}
\end{lma}
\begin{proof}
We only prove (1), since (2) follows the same idea. It suffices to show that if $\mu$ and $\nu$ are
finite dimensional representations, then $[\mu\oplus\nu]\cdot s\leq [\mu]\cdot s$ for all
$s\in \Cu(A\rtimes_\alpha G)$. With the notation from the definition above, observe that the maps
\[\varphi_{\mu},\varphi_\nu\colon A\to A\otimes \K(\Hi_\mu\oplus \Hi_\nu\oplus\C)\]
are orthogonal. Since orthogonal equivariant homomorphisms induce homomorphisms between the respective
crossed products which are also orthogonal, we deduce that
\[\Cu(\widehat{\varphi}_{\mu\oplus\nu})=\Cu(\widehat{\varphi}_\mu)+\Cu(\widehat{\varphi}_\nu).\]
By evaluating the above identity at $s$, and using \autoref{df: semimod structure}, one deduces
that $[\mu\oplus\nu]\cdot s\leq [\mu]\cdot s$, as desired.
\end{proof}

\begin{lma} \label{lem:CompatTakingSup}
The $\Cu(G)$-semimodule structure on $\Cu(A\rtimes_\alpha G)$ described above is
compatible with taking suprema in $\Cu(G)$.\end{lma}
\begin{proof} Let $(\Hi_{\mu_n},\mu_n)_{n\in\N}$ be a sequence of separable unitary representations of $G$
such that $([\mu_n])_{n\in\N}$ is increasing in $\Cu(G)$. Set $[\mu]=\sup\limits_{n\in\N} [\mu_n]$.
Without loss of generality, we can assume that $\mu_n$ is a subrepresentation of $\mu_{n+1}$ for all $n\in\N$.
In particular, we may assume that $\Hi_\mu=\overline{\bigcup\limits_{n\in\N}\mathcal{H}_{\mu_n}}$ with $\mathcal{H}_{\mu_n}\subseteq \mathcal{H}_{\mu_{n+1}}$ for all $n\in\N$. It follows that
$\mu^+=\sup\limits_{n\in\N} (\mu_n^+)$ as representations of $G$ on $\mathcal{H}_\mu\oplus\C$.
Thus, for $a\in A$, we have
$$\iota_{p_\mu}(a)=a\otimes p_{\Hi_\mu}=\sup\limits_{n\in\N} (a\otimes p_{\Hi_{\mu_n}})=\sup\limits_{n\in\N} \left(\varphi_{\Hi_{\mu_n}}(a)\right).$$
Finally, since $\Cu(\widehat{\iota_{q_\mu}})$ is an isomorphism in $\CCu$, we conclude that
\begin{align*} \sup\limits_{n\in\N} \left([\mu_n]\cdot s\right) &= \sup\limits_{n\in\N} \left(\Cu(\widehat{\iota_{q_\mu}})^{-1}\circ \Cu(\overline{\varphi}_{\Hi_{\mu_n}})(s)\right)\\
&= \Cu(\widehat{\iota_{q_\mu}})^{-1}\left(\sup\limits_{n\in\N} \left(\Cu(\overline{\varphi}_{\Hi_{\mu_n}})(s)\right)\right)\\
&=\Cu(\widehat{\iota_{q_\mu}})^{-1}\circ\Cu(\overline{\varphi}_{\Hi_{\mu}})(s)\\
&= [\mu]\cdot s,\end{align*}
for all $s\in\Cu(A\rtimes_\alpha G)$, as desired.\end{proof}

For later use, we record here the following fact.

\begin{prop}\label{prop:CuGEquivExSeqs}
Let
\[\xymatrix{0\ar[r]& I \ar^-{\iota}[r]&A\ar[r]^-{\pi}& B\ar[r]& 0}\]
be an exact sequence of $G$-\ca s, with actions $\gamma\colon G\to\Aut(I)$ and $\beta\colon G\to \Aut(B)$.
Then $\Cu^G(I,\gamma)$ can be naturally identified with $\ker(\Cu^G(\pi))$, which by definition is
\[\ker(\Cu^G(\pi))=\{s\in \Cu^G(A,\alpha)\colon \Cu^G(\pi)(s)=0\}\subseteq \Cu^G(A,\alpha).\]
\end{prop}
\begin{proof}
Observe that $\ker(\Cu^G(\pi))$ only depends on the structure of $\Cu^G(A,\alpha)$ as a semigroup, and is
independent of the action of $\Cu(G)$. Denote by $\widehat{\pi}\colon A\rtimes_\alpha G\to B\rtimes_\beta G$ the
map induced by $\pi$. Using the natural isomorphisms from \autoref{Thm:JulgCuG} for
$\Cu^G(I,\gamma)$, $\Cu^G(A,\alpha)$ and $\Cu^G(B,\beta)$, it is enough to show that $\Cu(I\rtimes_\gamma G)$
can be naturally identified with the subsemigroup
\[\{s\in \Cu(A\rtimes_\alpha G)\colon \Cu(\widehat{\pi})(s)=0\}\]
of $\Cu(A\rtimes_\alpha G)$. This follows immediately from Theorem~5 in~\cite{CiuRobSan_CuIdeals},
so the proof is finished.
\end{proof}

We have now arrived at the main result of this section.

\begin{thm}\label{thm:JulgWithCuG}
Let $G$ be a compact group, let $A$ be a $C^*$-algebra, and let $\alpha\colon G\to\Aut(A)$ be a
continuous action. Then there exists a natural $\CCu^G$-isomorphism
\[\Cu^G(A,\alpha)\cong \Cu(A\rtimes_\alpha G),\]
where the $\Cu(G)$-semimodule structure on $\Cu(A\rtimes_\alpha G)$ is given by
\autoref{df: semimod structure}.
\end{thm}
\begin{proof}
Assume first that $A$ is unital. Let $(\Hi_\mu,\mu)$ be a finite dimensional unitary representation
of $G$, and let $\iota_{p_\mu}, \iota_{q_\mu}, \widehat{\iota_{p_\mu}}$ and $\widehat{\iota_{q_\mu}}$
be as in \autoref{df: semimod structure}. By naturality of
the isomorphism in \autoref{Thm:JulgCuG}, there exists a commutative diagram
\begin{align*}
\xymatrix@=1.40em{
\Cu^G(A,\alpha)\ar[d]_-{\sigma_A}\ar[rr]^-{\Cu^G(\iota_{p_\mu})}& & \Cu^G\left(A\otimes \K(\Hi_\mu\oplus\C),\alpha\otimes\Ad(\mu^+)\right)\ar[d]& &\Cu^G(A,\alpha)\ar[ll]_-{\Cu^G(\iota_{q_\mu})}\ar[d]^-{\sigma_A}\\
\Cu(A\rtimes_\alpha G)\ar[rr]_-{\Cu(\widehat{\iota_{p_\mu}})}& & \Cu\left((A\otimes \K(\Hi_\mu\oplus\C))\rtimes_{\alpha\otimes\Ad(\mu^+)}G\right) & &\Cu(A\rtimes_\alpha G)\ar[ll]^-{\Cu(\widehat{\iota_{q_\mu}})},
}
\end{align*}
where all vertical arrows are the isomorphisms given by \autoref{Thm:JulgCuG}.
By \autoref{prop:SemimodStrBad}, the map $\Cu^G(\iota_{q_\mu})$
corresponds to multiplication by the class of the trivial representation in the $\Cu(G)$-semimodule
$\Cu^G\left(A\otimes \K(\Hi_\mu\oplus\C),\alpha\otimes\Ad(\mu^+)\right)$.
It follows that $\Cu^G(\iota_{q_\mu})$ is invertible.
Thus $\Cu(\widehat{\iota_{q_\mu}})$ is also invertible, since the vertical arrows are invertible.
By definition, $\Cu(\widehat{\iota_{q_\mu}})^{-1}\circ \Cu(\widehat{\iota_{p_\mu}})$ is multiplication by $[\mu]$ on $\Cu(A\rtimes_\alpha G)$, and \autoref{prop:SemimodStrBad} shows that the composition
$\Cu^G(\iota_{q_\mu})^{-1}\circ\Cu^G(\iota_{p_\mu})$ agrees with multiplication by $[\mu]$ on $\Cu^G(A,\alpha)$.
Commutativity of the diagram implies that the isomorphism $\sigma_A\colon \Cu^G(A,\alpha)\to \Cu(A\rtimes_\alpha G)$,
which appears both in the left and right vertical arrows, commutes with multiplication by $[\mu]$.

Assume now that $(\Hi_\nu,\nu)$ is a separable unitary representation
of $G$. Since $G$ is compact, it follows that $\K(\mathcal{H}_\nu)$ has an increasing approximate
identity $(e_n)_{n\in\N}$ consisting of $G$-invariant projections. For $n\in\N$, denote by
$\mu_n\colon G\to\mathcal{U}(e_n\mathcal{H}_\nu)$ the restriction of $\nu$. It follows that
$$[\nu]=\sup\limits_{n\in\N}[\mu_n]$$
in $\Cu(G)$. Since the $\Cu(G)$-semimodule structure on $\Cu(A\rtimes_\alpha G)$ described above
is compatible with taking suprema in $\Cu(G)$ by \autoref{lem:CompatTakingSup}, it follows that
the isomorphism $\sigma_A\colon \Cu^G(A,\alpha)\to \Cu(A\rtimes_\alpha G)$ commutes with multiplication
by $[\nu]$, since it commutes with multiplication by $[\mu_n]$ for all $n\in\N$ by the above paragraph.
This shows that $\sigma_A$ is a $\Cu(G)$-semimodule homomorphism.

Now suppose that $A$ is non-unital, and denote by $\widetilde{A}$ its unitization. Define an extension
$\widetilde{\alpha}\colon G\to \Aut(\widetilde{A})$ of $\alpha$ to $\widetilde{A}$ by setting $\widetilde{\alpha}_g(a+\lambda 1)=\alpha_g(a)+\lambda 1$
for all $a\in A$ and all $\lambda\in\C$. The short exact sequence of $G$-$C^*$-algebras
$$0\to A\to \widetilde{A} \to \C\to 0$$
induces the short exact sequence of crossed products
$$0\to A\rtimes_\alpha G\to \widetilde{A}\rtimes_{\widetilde{\alpha}} G \to C^*(G)\to 0.$$
(Recall that $\C\rtimes_\id G\cong C^*(G)$ for a locally compact group $G$.)

Denote by $\sigma_A\colon \Cu^G(A,\alpha)\to \Cu(A\rtimes_\alpha G)$, by $\sigma_{\widetilde{A}}\colon \Cu^G(\widetilde{A},\widetilde{\alpha})\to \Cu(\widetilde{A}\rtimes_{\widetilde{\alpha}}G)$,
and by $\sigma_\C\colon \Cu^G(\C)\to \Cu(C^*(G))$, the natural isomorphisms given by \autoref{Thm:JulgCuG}.
Then the following diagram in $\CCu$ is commutative:
\begin{align*} \xymatrix{ \Cu(A\rtimes_\alpha G)\ar[r]\ar[d]_-{\sigma_A} & \Cu(\widetilde{A}\rtimes_{\widetilde{\alpha}} G)\ar[r]^-{\Cu^G(\widehat{\pi})}\ar[d]_-{\sigma_{\widetilde{A}}}&\Cu(C^*(G))\ar[d]_-{\sigma_\C}\\
\Cu^G(A,\alpha)\ar[r] & \Cu^{G}(\widetilde{A},\widetilde{\alpha})\ar[r]_-{\Cu^G(\widehat{\pi})} &\Cu^{G}(\C,\id_\C).}\end{align*}

The maps $\sigma_{\widetilde{A}}$ and $\sigma_{\C}$ are $\Cu(G)$-semimodule homomorphisms by the unital case.
By \autoref{prop:CuGEquivExSeqs}, it follows that $\Cu^G(A,\alpha)$ is the kernel of $\Cu^G(\widehat{\pi})$. By
commutativity of the diagram, $\sigma_A$ is the restriction of $\sigma_{\widehat{A}}$ to $\ker(\Cu^G(\widehat{\pi}))$,
and since $\sigma_{\widehat{A}}$ preserves the $\Cu(G)$-action, so does $\sigma_A$. This finishes the proof. \end{proof}

We illustrate these methods by computing an easy example.
Let $G$ be a compact group, and let $\widehat{G}$ be its dual.
If $A$ is a \ca, then we write $\Cu(G)\otimes \Cu(A)$
for the $\Cu(G)$-semimodule
\[\Cu(G)\otimes \Cu(A)=\{f\colon \widehat{G}\to \Cu(A)\colon f \mbox{ has countable support}\},\]
with pointwise addition and partial order. The $\Cu(G)$-action on $\Cu(G)\otimes \Cu(A)$ can
be described as follows. Given $[\mu]\in \Cu(G)$ and $[\pi]\in\widehat{G}$,
let $m_\pi(\mu)\in\overline{\Z_{\geq 0}}$ be the multiplicity of $\pi$ in $\mu$. Then
\[[\mu]=\sum\limits_{[\pi]\in\widehat{G}}m_\pi(\mu)\cdot [\pi].\]
For $f\in \Cu(G)\otimes\Cu(A)$, we set
\[([\mu]\cdot f)([\pi])=m_\pi(\mu)f([\pi])\]
for $\pi\in\widehat{G}$.

The tensor product notation is justified because of the following. One can check that
\[\Cu(G)\cong \{f\colon \widehat{G}\to \overline{\N}\colon f \mbox{ has countable support}\},\]
with pointwise operations and partial order. Moreover, it is easy to check that
$\Cu(G)\otimes \Cu(A)$ really is the tensor product
in the category $\CCu$ of the semiring $\Cu(G)$ and the semigroup $\Cu(A)$, in the sense of Theorem~6.3.3
in \cite{AntPerThi}.

\begin{prop}\label{prop:EgTrivialAction}
Suppose that $G$ acts trivially on $A$. Then $\Cu^G(A,\id_A)\cong \Cu(G)\otimes \Cu(A)$.
\end{prop}
\begin{proof}
Since $G$ acts trivially on $A$, we have $A\rtimes_\alpha G\cong A\otimes C^*(G)$ canonically.
For $[\pi]\in\widehat{G}$, denote by $d_\pi$ the dimension of $\pi$. Then
$C^*(G)\cong \bigoplus\limits_{[\pi]\in\widehat{G}} M_{d_\pi}$, so
\[A\rtimes_\alpha G\cong \bigoplus_{[\pi]\in\widehat{G}} M_{d_\pi}(A).\]
For $[\tau]\in\widehat{G}$, let
\[\rho_\tau\colon \bigoplus_{[\pi]\in\widehat{G}} M_{d_\pi}(A)\to M_{d_\tau}(A)\]
be the corresponding surjective $\ast$-homomorphism.

We define a map $\varphi \colon \Cu(A\rtimes_\alpha G)\to \Cu(G)\otimes\Cu(A)$ as follows. For
a positive element $A$ in
\[\K\otimes (A\rtimes_\alpha G)\cong \bigoplus\limits_{[\pi]\in\widehat{G}} \K\otimes M_{d_\pi}(A),\]
set $\varphi([a])([\pi])=[\rho_\pi(a)]\in\Cu(A)$ for
$[\pi]\in\widehat{G}$. It is clear that $\varphi([a])$ is independent
of the representative of $[a]$.
We must check that $\varphi([a])$ has countable support. Without loss of generality, assume
that $\|a\|=1$. For $0<\ep<1$, there exists a finite subset $X_\ep$ of $\widehat{G}$
such that the element $(a-\ep)_+$ belongs to $\bigoplus\limits_{[\pi]\in X_\ep} \K\otimes M_{d_\pi}(A)$.
Since
\[\varphi([a])=\sup\limits_{n\in\N}\psi\left(\left[\left(a-\frac{1}{n}\right)_+\right]\right),\]
it follows that $\varphi([a])$ is a supremum of an
increasing sequence of functions $\widehat{G}\to \Cu(A)$ with finite support, so the support of
$\varphi([a])$ is countable.

\emph{Claim: $\varphi$ preserves Cuntz subequivalence.}
Let $a$ and $b$ be positive elements in
$\bigoplus\limits_{[\pi]\in\widehat{G}} \left(\K\otimes M_{d_\pi}(A)\right)$ satisfying $a\precsim b$.
Without loss of generality, we may assume that $\|a\|=\|b\|=1$. Given $n\in\N$, there exists $m\in\N$
such that
$\left(a-\frac{1}{m}\right)_+\precsim \left(b-\frac{1}{n}\right)_+$. Hence,
$\rho_\pi\left(\left(a-\frac{1}{m}\right)_+\right)\precsim \rho_\pi\left(\left(b-\frac{1}{n}\right)_+\right)$
for all $[\pi]\in\widehat{G}$. It follows that
\[\varphi\left(\left[\left(a-\frac{1}{m}\right)_+\right]\right)\leq \varphi\left(\left[\left(b-\frac{1}{n}\right)_+\right]\right)\varphi \psi([b]).\]
By taking supremum in $m$, we conclude that $\varphi([a])\leq \varphi([b])$, and the claim is proved.

It is clear that the restriction of $\varphi$ to the image in $\Cu(A\rtimes_\alpha G)$ of the
positive elements in $\bigoplus\limits_{[\pi]\in\widehat{G}} \left(\K\otimes M_{d_\pi}(A)\right)$
with finitely many nonzero coordinates preserves suprema of increasing sequences, preserves the compact
containment relation, and is an order embedding. We want to show that $\varphi$ is an isomorphism in the category
$\CCu^G$. This will be a consequence of the next three claims, which will complete the proof of the proposition.

\emph{Claim: $\varphi$ is an order embedding.}
Let $a$ and $b$ be positive elements in
\[\bigoplus\limits_{[\pi]\in\widehat{G}} \left(\K\otimes M_{d_\pi}(A)\right),\]
and assume that $\varphi([a])\leq \varphi([b])$ in $\Cu(G)\otimes \Cu(A)$. If there exists
a finite subset $X\subseteq \widehat{G}$ such that $a$ and $b$
belong to $\bigoplus\limits_{[\pi]\in X} \left(\K\otimes M_{d_\pi}(A)\right)$, then it is clear
that we must have $[a]\leq [b]$. For the general case, we can assume without loss of generality
that $\|a\|=\|b\|=1$. The sequence
$\left(\varphi\left(\left[\left(a-\frac{1}{n}\right)_+\right]\right)\right)_{n\in\N}$
is rapidly increasing by the comments before this claim. In particular, for fixed $n\in\N$, we
have $\varphi\left(\left[\left(a-\frac{1}{n}\right)_+\right]\right)\ll \varphi([a])$. Since
$\varphi([b])= \sup\limits_{n\in\N}\varphi\left(\left[\left(b-\frac{1}{n}\right)_+\right]\right)$
by definition of $\varphi$, there exists $n_0\in\N$ such that
\[\varphi\left(\left[\left(a-\frac{1}{n}\right)_+\right]\right)\ll \varphi\left(\left[\left(b-\frac{1}{m}\right)_+\right]\right)\]
for all $m\geq n_0$. It follows that
\[\left[\left(a-\frac{1}{n}\right)_+\right]\leq \left[\left(b-\frac{1}{m}\right)_+\right],\]
because $\left(a-\frac{1}{n}\right)_+$ and $\left(b-\frac{1}{m}\right)_+$ have only finitely many
nonzero coordinates. By taking the supremum over $m$ first, and then over $n$, we deduce that $[a]\leq [b]$, as desired.

\emph{Claim: $\varphi$ is surjective.}
Let $f\colon \widehat{G}\to \Cu(A)$ be a function with countable support. Let $(\pi_n)_{n\in\N}$
be an enumeration of the support of $f$. For $n\in\N$, let $a_n\in \K\otimes A$ be a positive
element with $\|a_n\|= \frac{1}{n}$ satisfying $[a_n]=f(\pi_n)$ in $\Cu(A)$. Let
\[a\in \K\otimes (A\rtimes_\alpha G)\cong \bigoplus_{[\pi]\in\widehat{G}}\left(\K\otimes M_{d_\pi}(A)\right)\]
be the positive element determined by $\rho_{\pi_n}(a)=a_n$ for $n\in\N$, and $\rho_{\pi}(a)=0$
for $\pi\notin\supp(f)$. It is then clear that $\varphi([a])=f$.

\emph{Claim: $\varphi$ is a morphism in $\CCu^G$.}
We need to check conditions (M1) and (M2) in \autoref{def: CCu} and that $\varphi$ is a $\Cu(G)$-semimodule
homomorphism.

To check (M1), let $(s_n)_{n\in\N}$
be an increasing sequence in $\Cu(A\rtimes_\alpha G)$, and set $s=\sup\limits_{n\in\N}s_n$. We
want to show that $\varphi(s)=\sup\limits_{n\in\N}\varphi(s_n)$. Let $t\in \Cu(G)\otimes\Cu(A)$
satisfy $\varphi(s_n)\leq t$ for all $n\in\N$. Since $\varphi$ is surjective (see the previous claim),
there exists $r\in \Cu(A\rtimes_\alpha G)$ with $\varphi(r)=t$. Now, since $\varphi$ is an order embedding
(see claim above), we deduce that $s_n\leq r$ for all $n\in\N$. Hence $s\leq r$ by the definition of supremum.
Thus $\varphi(s)\leq \varphi(r)$, and $\varphi(s)=\sup\limits_{n\in\N}\varphi(s_n)$.

In order to check (M2), let $s,t\in \Cu(A\rtimes_\alpha G)$ satisfy $s\ll t$. We want to show that
$\varphi(s)\ll \varphi(t)$. To this end, let $(r_n)_{n\in\N}$ be an increasing sequence in $\Cu(G)\otimes\Cu(A)$
satisfying $\varphi(t)\leq \sup_{n\in\N}r_n$. For each $n\in\N$, choose $z_n\in\Cu(A\rtimes_\alpha G)$ with
$\varphi(z_n)=r_n$. Since $\varphi$ is an order embedding, we have $t\leq \sup_{n\in\N}z_n$. Hence there
exists $m\in\N$ with $s\leq z_m$, and so $\varphi(s)\leq \varphi(z_m)$. This shows that $\varphi(s)\ll \varphi(t)$,
as desired.

We will now check that $\varphi$ is a $\Cu(G)$-semimodule homomorphism.
Observe first that $A\rtimes_\alpha G\cong C^*(G)\otimes A$, and that the $\Cu(G)$-module structure
on $\Cu(A\rtimes_\alpha G)\cong \Cu(C^*(G)\otimes A)$ is the usual multiplication
on $\Cu(C^*(G))=\Cu(G)$ and trivial on $\Cu(A)$. Let $[\mu],[\tau]\in\widehat{G}$, and
write $[\mu\otimes\rho]$ as a linear combination
\[[\mu\otimes \rho]=\sum\limits_{[\pi]\in\widehat{G}}m_\pi(\mu\otimes\tau)\cdot [\pi].\]
For a positive element
\[a_\tau\in \K\otimes M_{d_\tau}(A)\subseteq \bigoplus\limits_{[\pi]\in\widehat{G}}(\K\otimes M_{d_\pi}(A))\cong \K\otimes (A\rtimes_\alpha G),\]
we have
\[[\mu]\cdot [a_\tau]=\left[(m_\pi(\mu\otimes \tau)a_\tau)_{[\pi]\in\widehat{G}}\right].\]
Hence, for $[\pi]\in \widehat{G}$, we have
\[\varphi([\mu]\cdot [a_\tau])([\pi])=[m_\pi(\mu\otimes \tau)a_\tau]=([\mu]\cdot \varphi([a_\tau]))([\widehat{\pi}]),\]
as desired. This completes the proof of the claim and the proposition.
\end{proof}

Similarly to what happens in equivariant $K$-theory, the $\Cu(G)$-semimodule structure on
$\Cu(A\rtimes_\alpha G)$ has a more concrete expression when $G$ is abelian.

We saw that $\Cu(G)$ consists of the suprema of all finite linear combinations of elements of
$\widehat{G}$ with coefficients in $\Z_{\geq 0}$, with coordinate-wise order, addition and multiplication.
In particular,
it follows that a $\Cu(G)$-semimodule structure on a partially ordered abelian semigroup that is
compatible with suprema is necessarily completely determined by multiplication by the elements
of $\widehat{G}$.

We denote by $\widehat{\alpha}\colon \widehat{G}\to \Aut(A\rtimes_\alpha G)$ the dual action of $\alpha$.
In the following proposition, we use the identification
\[\Hi_A=\left(\bigoplus\limits_{n\in\N}\bigoplus\limits_{\pi\in\widehat{G}}\Hi_\pi\right)\otimes A.\]

\begin{prop}\label{prop: semimodule structure when G is abelian} Let $G$ be a compact abelian group, let $A$ be a $C^*$-algebra, and let $\alpha\colon G\to\Aut(A)$ be a continuous action. Then for $\tau\in \widehat{G}$ and $s\in \Cu(A\rtimes_\alpha G)$, we have $\tau \cdot s=\Cu(\widehat{\alpha}_\tau)(s)$.
More precisely, the following diagram commutes:
\begin{align*}
\xymatrix{
\Cu^G(A,\alpha)\ar[d]_-{\sigma}\ar[rr]^-{ \tau\cdot} && \Cu^G(A,\alpha)\ar[d]^-{\sigma}\\
\Cu(A\rtimes_\alpha G)\ar[rr]_-{\Cu(\widehat{\alpha}_\tau)} && \Cu(A\rtimes_\alpha G)
,}\end{align*}
where $\sigma\colon \Cu^G(A,\alpha)\to \Cu(A\rtimes_\alpha G)$ is the natural $\CCu$-isomorphism given by
\autoref{Thm:JulgCuG}.
\end{prop}
\begin{proof}
Fix $\tau\in\widehat{G}$.
By the construction of the $\Cu$-isomorphism $\sigma\colon \Cu^G(A,\alpha)\to \Cu(A\rtimes_\alpha G)$ from
\autoref{Thm:JulgCuG}, and adopting the notation in its proof,
it is enough to show that the following diagram commutes:
\begin{align*}
\xymatrix{
\Cu(\K(\Hi_A)^G)\ar[d]_-{\Cu(\theta)} \ar[r]^-{ \tau \cdot}
& \Cu(\K(\Hi_A)^G)\ar[d]^-{\Cu(\theta)} \\
\Cu(\K(\ell^2(\N))\otimes \K(L^2(G)\otimes A)^G) \ar[d]_-{\kappa}
& \Cu(\K(\ell^2(\N))\otimes\K(L^2(G)\otimes A)^G)\ar[d]^-{\kappa} \\
\Cu(\K(L^2(G)\otimes A)^G)\ar[r]_-{\Cu(\widehat{\alpha}_\tau)}
& \Cu(\K(L^2(G)\otimes A)^G).}
\end{align*}

The dual action $\widehat{\alpha}\colon \widehat{G}\to\Aut((A\otimes\K(L^2(G)))^G)$
has the following description. For $\chi\in\widehat{G}$, let $u_\chi\in\U(L^2(G))$
be the corresponding multiplication operator, which is given by
\[u_\chi(\xi)(g)=\chi(g)\xi(g)\]
for all $\xi\in L^2(G)$ and for all $g\in G$. Define an automorphism
$\gamma_\chi\colon A\otimes\K(L^2(G))\to A\otimes \K(L^2(G))$ by $\gamma_\chi=\id_A\otimes\Ad(u_\chi)$.
It is clear that $\gamma_\chi$ commutes with $\alpha\otimes\Ad(\lambda_g)$ for all $g\in G$, and
thus $\gamma_\chi$ leaves $(A\otimes\K(L^2(G)))^G$ invariant. Then $\widehat{\alpha}_\chi$ is the
restriction of $\gamma_\chi$ to $(A\otimes\K(L^2(G)))^G$.

For $\chi\in\widehat{G}$, the action on $\K(\Hi_A)^G$ can be described as follows.
Write $\Hi_A=\ell^2(\N)\otimes L^2(G)\otimes A$, and let $u_\chi$ be the unitary
on $L^2(G)$ described above. Then $w_\chi=\id_{\ell^2(\N)}\otimes u_\chi\otimes \id_A$ is
a unitary on $\Hi_A$, and conjugation by $w_\chi$ defines an automorphism of $\K(\Hi_A)$.
This automorphism clearly commutes with the action of $G$ on $\K(\Hi_A)$, so it defines,
by restriction, an automorphism of its fixed point algebra $\K(\Hi_A)^G$.

Using these descriptions of the actions of $\widehat{G}$, it is clear that the diagram
above is commutative, and the result follows.
\end{proof}

We close this section with an application to invariant hereditary subalgebras. The result is
a Cuntz semigroup analog of Proposition~2.9.1 in~\cite{Phi_Book}.

\begin{prop}\label{prop:HeredSubalg}
Suppose that $A$ is separable and $G$ is second countable.
Let $B\subseteq A$ be an $\alpha$-invariant hereditary subalgebra
of $A$, and denote by $\beta\colon G\to\Aut(B)$ the compression of $\alpha$.
If $B$ is full, then the canonical inclusion induces a natural $\CCu^G$-isomorphism
$\Cu^G(B,\beta)\to \Cu^G(A,\alpha)$.
\end{prop}
\begin{proof}
Under the canonical identification given by \autoref{Thm:JulgCuG}, the map in the statement becomes
the map $\Cu(B\rtimes_\beta G)\to \Cu(A\rtimes_\alpha G)$ induced by the inclusion. Now,
Proposition~2.9.1 in~\cite{Phi_Book} shows that $B\rtimes_\beta G$ is a full hereditary
subalgebra of $A\rtimes_\alpha G$. Separability of the objects implies, by Brown's stability
theorem, that they are stably isomorphic.
It follows that the canonical map $\Cu(B\rtimes_\beta G)\to \Cu(A\rtimes_\alpha G)$, which
belongs to $\CCu^G$ by \autoref{thm:EqCtzSmgpCCuG}, is an isomorphism.
\end{proof}

\section{Examples and computations}

In this section, we compute the equivariant Cuntz semigroups of a number of dynamical
systems. In most of our examples, the Rokhlin property for an action of a finite group
comes up as an useful technical
device that makes the computations possible. The Rokhlin property is also implicitly
used in \autoref{thm:ConjActsCu}. We therefore recall its definition here.

\begin{df}\label{df:Rp}
Let $A$ be a $\sigma$-unital \ca, let $G$ be a finite group, and let $\alpha\colon G\to\Aut(A)$
be an action. We say that $\alpha$ has the \emph{Rokhlin property} if for every finite set
$F\subseteq A$ and every $\varepsilon>0$, there exists orthogonal positive elements
$a_g\in A$, for $g\in G$, satisfying
\begin{enumerate}\item $\|(\alpha_g(a_h)-a_{gh})x\|<\varepsilon$ for all $g,h\in G$;
\item $\|a_gx-xa_g\|<\varepsilon$ for all $g\in G$ and all $x\in F$;
\item $\left\|x\left(\sum\limits_{g\in G}a_g\right)-x\right\|<\varepsilon$ for all $x\in F$.\end{enumerate}
\end{df}

Let $G$ be a finite
group, let $S$ be a semigroup in the category $\mathbf{Cu}$, and let $\gamma\colon G\to \mathrm{Aut}(S)$
be an action. Set $S^\gamma=\{s\in S\colon \gamma_g(s)=s \mbox{ for all } g\in G\}$, and set
\[S^\gamma_{\mathbb N}=\left\{s\in S\colon \exists \ (s_n)_{n\in\mathbb N} \mbox{ in } S^\gamma \colon
s_n\ll s_{n+1} \ \forall \ n\in\mathbb N, \, s=\sup\limits_{n\in\mathbb N} s_n\right\}.\]

For use in the following proposition, we observe that if $\alpha\colon G\to\Aut(A)$ is
an action of a \emph{second countable} compact group $G$, then its equivariant Cuntz semigroup
$\Cu^G(A,\alpha)$ can be canonically identified with $\Cu\left((A\otimes\K(L^2(G)))^{\alpha\otimes\Ad(\lambda)}\right)$,
where the $\Cu(G)$-semimodule structure is given by tensor product (where we identify
$(L^2(G)\otimes \Hi_\mu,\lambda\otimes\mu)$ with $(L^2(G),\lambda)$ whenever $(\Hi_\mu,\mu)$ is a separable
unitary representation of $G$).


\begin{prop}\label{prop: fixedpointCuntz}
Let $A$ be a \ca, let $G$ be a finite group, and let $\alpha\colon G\to \mathrm{Aut(A)}$
be an action with the Rokhlin property. Then there exists a natural $\mathbf{Cu}^G$-isomorphism
\[\Cu^G(A,\alpha)\cong \Cu(A)^{\Cu(\alpha)}_{\mathbb N},\]
where the induced $\Cu(G)$-semimodule structure on $\Cu(A)^{\Cu(\alpha)}_{\mathbb N}$
is trivial.
\end{prop}
\begin{proof}
Denote by $e\in\K(\ell^2(G))$ the projection onto the constant functions, and let
$\iota\colon A^\alpha\hookrightarrow (A\otimes \mathcal K(L^2(G)))^{\alpha\otimes\Ad(\lambda)}$
be the inclusion given by $\iota(a)= a\otimes e$ for all $a\in A$.
Since $\alpha$ has the Rokhlin property, $\iota$ induces a $\CCu$-isomorphism between the Cuntz
semigroups $\Cu(A^\alpha)$ and $\Cu((A\otimes \mathcal K(L^2(G)))^{\alpha\otimes\Ad(\lambda)})$
by Proposition~2.6 in~\cite{Gar_CptRokCP}. By the comments above
this proposition, we deduce that
$\Cu^G(A,\alpha)$ is naturally isomorphic to $\Cu(A^\alpha)$. Since $\Cu(A^\alpha)$ can be
canonically identified with $\Cu(A)^{\Cu(\alpha)}_\N$ by Theorem~4.1 in~\cite{GarSan_RokConstrI},
we conclude that the inclusion of
$A^\alpha$ into $A$ induces a $\CCu$-isomorphism of $\Cu(A^\alpha)$ and $\Cu(A)^{\Cu(\alpha)}_{\mathbb N}$.

Finally, the dual action $\widehat{\alpha}$ is approximately representable by part~(i)
of Proposition~4.4 in~\cite{Naw_RpNonunital}. In particular, $\widehat{\alpha}_\tau$ is approximately inner for all
$\tau\in\widehat{G}$, so $\Cu(\widehat{\alpha}_\tau)$ is the identity map. We deduce from
\autoref{thm:JulgWithCuG} that the $\Cu(G)$-semimodule structure on $\Cu^G(A,\alpha)$ is trivial.
This finishes the proof.
\end{proof}


We give two concrete applications of the above computation. Recall that for $m\in\N$, there
is a canonical identification
$\Cu(M_{m^\I})\cong \mathbb Z_{\ge 0}\left[\frac{1}{m}\right]\sqcup \overline{\mathbb R_{>0}}$

\begin{eg}\label{exa: UHF}
Let $G$ be a finite group and set $m=|G|$. Let $\mu^G\colon  G\to \mathrm{Aut}(M_{m^\infty})$ be the action
considered, for example, in \cite[Example 2.1]{GarSan_RokConstrI}. Then there is a $\CCu^G$-isomorphism
\[\Cu^G(M_{|G|^\infty}, \mu^G)\cong \mathbb Z_{\ge 0}\left[\frac{1}{m}\right]\sqcup \overline{\mathbb R_{>0}},\]
where the right-hand side carries the trivial $\Cu(G)$-semimodule structure.
\end{eg}
\begin{proof}
The $\Cu(G)$ -semimodule structure is trivial by \autoref{prop: fixedpointCuntz}, since $\mu^G$
has the Rokhlin property. Also, by \autoref{prop: fixedpointCuntz}, we have
\[\Cu^G(M_{|G|^\infty}, \mu^G)\cong \Cu(M_{|G|^\infty})^{\Cu(\mu^G)}_{\mathbb N}.\]
The computation follows since $\Cu(\mu_g^G)=\mathrm{id}_{\Cu(M_{|G|^\infty})}$ for all $g\in G$.
\end{proof}


\begin{eg}\label{eg:UHF-abs}
Let $n\in \mathbb N$ and set $\omega_n=e^{\frac{2\pi i}{n}}\in\T$.
Let $\gamma\colon \mathbb Z_n\to \mathrm{Aut}(C(\mathbb T))$ be the action given
by $\gamma_k(f)(z)=f(\omega^k_nz)$ for all $z\in \mathbb T$ and all $k\in \mathbb Z_n$.
Let $A$ be any unital $M_{n^\I}$-absorbing \ca, and let $\alpha\colon \Z_n\to \Aut(A)$ be any
action.
Then there is a $\CCu^{\Z_n}$-semimodule isomorphism
\[
\Cu^{\mathbb Z_n}\left(C(\mathbb T, A), \gamma\otimes \alpha\right)\cong
\left\{
f\in \mathrm{Lsc}((\mathbb T, \Cu(A))\colon f(\omega_nz)=f(z) \text{ for all } z\in \T
\right\},
\]
where the $\Cu(\mathbb Z_n)$-semimodule structure on the right-hand side is trivial. \end{eg}
\begin{proof}
Note that $\gamma$ has the unitary Rokhlin property from Definition~3.5
in~\cite{Gar_UHFabs} (take the unitary $u$ in that definition to be $u(z)=z$ for $z\in\T$).
One easily checks that $\gamma\otimes\alpha$ also has the unitary
Rokhlin property. Since $A$ is assumed to absorb $M_{n^\I}$, it follows from Theorem~3.19
in~\cite{Gar_UHFabs} that $\gamma\otimes\alpha$ has the Rokhlin property. We deduce
from \autoref{prop: fixedpointCuntz} that the $\Cu(\Z_n)$-semimodule structure on
$\Cu^{\Z_n}(C(\T,A))$ is trivial. Again by \autoref{prop: fixedpointCuntz}, we have
\[\Cu^{\mathbb Z_n}\left(C(\mathbb T, A), \gamma\otimes \alpha\right)\cong
\Cu(C(\mathbb T, A)^{\Cu(\gamma\otimes \alpha)}_{\mathbb N}.\]
The result then follows from \cite[Theorem 3.4]{AntPerSan_Pullbacks}.
\end{proof}

Let $\mathcal W$ be the stably projectionless simple $C^*$-algebra studied in \cite{Jac_W}.
This algebra has a unique tracial state and trivial $K$-groups.
By \cite[Corollary 6.8]{EllRobSan}, we have $\Cu(\mathcal W)\cong \overline{\mathbb R_{\ge 0}}$.
Moreover, every automorphism of $\mathcal W$ is approximately inner by \cite[Theorem 1.0.1]{Rob_classif}.
Consequently, if $\alpha\colon G\to \mathrm{Aut}(A)$ is an action of a group $G$ on $\mathcal{W}$, then $\Cu(\alpha_g)=\mathrm{id}_{\Cu(\mathcal W)}$ for all $g\in G$.

The following computations are similar to the previous ones, so we will omit them.

\begin{eg}\label{exa: W}
Let $G$ be a finite group and let $\mu\colon G\to \Aut(\mathcal W)$ be the unique (up to conjugacy)
action with the Rokhlin property; see \cite{Naw_RpNonunital}.
Then
\[\Cu^G(\mathcal W, \mu)\cong \overline{\mathbb R_{\ge 0}}\]
and
\[\Cu^{\Z_n}\left(C(\mathbb T,\mathcal W), \gamma\otimes \mu\right)\cong\left\{\mathrm{Lsc}(\mathbb T, \overline{\mathbb R_{\ge 0}})\colon f(\omega_n z)=f(z) \text{ for all } z\in \mathbb{T}\right\}.
\]
Moreover, the corresponding $\Cu(G)$ and $\Cu(\mathbb Z_n)$-semimodule structures are trivial.
\end{eg}

\subsection{Pullbacks of dynamical systems.}

Let $(A, \alpha)$ and $(B, \beta)$ be $G$-$C^*$-algebras.
Let $I$ be an invariant closed two-sided ideal in $A$, and let
$\phi\colon A\to A/I$ denote the quotient map. Note that $\phi$ is equivariant when taking on $A/I$ the action
$\alpha_{A/I}$ induced by $\alpha$.
Let $\psi\colon (B, \beta)\to (A/I, \alpha_{A/I})$ be an equivariant $\ast$-homomorphism.
By \cite[Proposition 6.2]{Ped_PbackPout}, the following pullback exists in the category of $G$-$C^*$-algebras:
\[
\xymatrix{
(G,C,\gamma)\ar[rr]^-{\pi_B}\ar[d]_-{\pi_A} && (B, \beta)\ar[d]^-\psi \\
(A, \alpha)\ar[rr]_-{\phi} && (A/I, \alpha_{A/I}),
}
\]
where we set $C=A\oplus_{A/I} B$, denote by $\pi_A\colon C\to A$ and $\pi_B\colon C\to B$
the corresponding quotient maps, and take $\gamma=(\alpha,\beta)$ to be the pullback action of
$G$ on $C$.
By applying the functor $\Cu^G$ to the diagram above, we get the following commutative diagram in the
category $\CCu^G$:
\[
\xymatrix{\Cu^G(C,\gamma)\ar[rr]^-{\Cu^G(\pi_B)}\ar[d]_-{\Cu^G(\pi_A)} && \Cu^G(B, \beta)\ar[d]^-{\Cu^G(\psi)} \\
\Cu^G(A, \alpha)\ar[rr]_-{\Cu^G(\phi)} && \Cu^G(A/I, \alpha_{A/I}).}
\]

Consider the pullback $\Cu^G(A, \alpha)\oplus_{\Cu^G(A/I, \alpha_{A/I})}\Cu^G(B, \beta)$
in the category of ordered semigroups with the $\Cu(G)$-semimodule structure induced by
the ones in $\Cu^G(A, \alpha)$ and $\Cu^G(B, \beta)$. There is a natural morphism in the
category $\CCu^G$
\[
\rho\colon \Cu^G(C,\gamma)\to \Cu^G(A, \alpha)\oplus_{\Cu^G(A/I, \alpha_{A/I})}\Cu^G(B, \beta)\]
given by $\rho([(a,b)]_G)=([a]_G, [b]_G)$ for $[(a,b)]_G\in \Cu^G(C,\gamma)$.

\begin{prop}
Adopt the notation from the discussion above, and suppose that
$A/I\rtimes_{\alpha_{A/I}}G$ has stable rank one and that each of its closed two-sided
ideals has trivial $K_1$-group. Then $\rho$ is an order embedding.
\end{prop}
\begin{proof}
Denote by $\widehat\pi_A$, $\widehat\pi_B$, $\widehat\phi$, and $\widehat\psi$ the maps at the level
of the crossed products induced by $\pi_A$, $\pi_B$, $\phi$, and $\psi$, respectively.
By Theorem 6.3 of \cite{Ped_PbackPout}, the following diagram is a pullback
\[
\xymatrix{
C\rtimes_\gamma G\ar[rr]^-{\widehat\pi_B}\ar[d]_-{\widehat\pi_A} && B\rtimes_\beta G\ar[d]^-{\widehat\psi} \\ A\rtimes_\alpha G\ar[rr]_-{\widehat\phi} && A/I\rtimes_{\alpha_{A/I}}G.
}
\]
In other words, there is a natural $\ast$-isomorphism
\[C\rtimes_\gamma G\cong A\rtimes_\alpha G\oplus_{A/I\rtimes_{\alpha_{A/I}}G}B\rtimes_\beta G.\]
Therefore, by \autoref{Thm:JulgCuG}, it is enough to show that the map
\[\Cu(A\rtimes_\alpha G\oplus_{A/I\rtimes_{\alpha_{A/I}}G}B\rtimes_\beta G)
\to \Cu(A\rtimes_\alpha G)\oplus_{\Cu(A/I\rtimes_{\alpha_{A/I}}G)}\Cu(B\rtimes_\beta G),\]
given by $[(a,b)]_G\mapsto ([a]_G, [b]_G)$,
is an order embedding. That this is the case is a consequence of \cite[Theorem 3.1]{AntPerSan_Pullbacks},
by the assumptions on $A/I\rtimes_{\alpha_{A/I}}G$.
\end{proof}

We have arrived at the main result of this subsection.

\begin{thm}\label{thm: Lie}
Let $X$ be a compact metric space and let $G$ be a compact Lie group with
$\dim (X)\leq \dim (G)+1$. Let $\gamma\colon G\to \mathrm{Aut}(C(X))$ be an action
induced by a free action of $G$ on $X$, and let $Y$ be an invariant closed subset
of $X$. Let $(A, \alpha)$ and $(B, \beta)$ be $G$-$C^*$-algebras with
$A$ separable, of stable rank one, and such that the $K_1$-groups of all its closed
two-sided ideals are trivial. Let
\[\phi\colon (C(X, A), \gamma\otimes \alpha)\to (C(Y, A), \gamma|_{C(Y)}\otimes \alpha)\]
be the canonical equivariant quotient map, and let
\[\psi \colon (B, \beta)\to (C(Y, A), \gamma|_{C(Y)}\otimes \alpha)
\]
be an equivariant *-homomorphism. Then the ordered semigroup
\begin{align}\label{eq: pullback}
\Cu^G(C(X, A), \gamma\otimes\alpha)\oplus_{\Cu^G(C(Y,A), \gamma|_{C(Y)}\otimes\alpha)}\Cu^G(B, \beta)
\end{align}
belongs to the category $\mathbf{Cu}^G$, and the map
$\rho$ from $\Cu^G(C(X, A)\oplus_{C(Y, A)}B, (\gamma\otimes\alpha, \beta))$ to the semigroup in
\autoref{eq: pullback},
given by $\rho([(a,b)]_G)=([a]_G, [b]_G)$ for $[(a,b)]_G\in  \Cu^G(C(X, A)\oplus_{C(Y, A)}B,(\gamma\otimes\alpha,\beta))$,
is an isomorphism in the category $\mathbf{Cu}^G$.
\end{thm}
\begin{proof}
By \autoref{Thm:JulgCuG}, there are natural isomorphisms
\begin{align*}
\Cu^G(C(X, A), \gamma\otimes\alpha)&\cong\Cu(C(X, A)\rtimes_{\gamma\otimes \alpha}G),\\
\Cu^G(C(Y, A), \gamma|_{C(Y)}\otimes\alpha)&\cong\Cu(C(Y, A)\rtimes_{\gamma|_{C(Y)}\otimes \alpha}G),\\
\Cu^G(B, \beta)&\cong\Cu(B\rtimes_{\beta}G).
\end{align*}
Also, by \cite[Theorem 6.3]{Ped_PbackPout}, the following diagram is a pullback
\beqa
\xymatrix{
(C(X,A)\oplus_{C(Y, A)}B)\rtimes_{(\gamma\otimes \alpha, \beta)}G\ar[r]^-{\widetilde\pi_B}\ar[d]_-{\widehat\pi_{C(X, A)}} & B\rtimes_\beta G\ar[d]^-{\widehat\psi} \\
C(X, A)\rtimes_{\gamma\otimes \alpha}G\ar[r]_-{\widehat\phi} & C(Y, A)\rtimes_{\gamma|_{C(Y)}\otimes \alpha}G,
}
\eeqa
where $\widehat\phi$ and $\widehat\psi$ are the maps induced by $\phi$ and
$\psi$, and the maps $\widehat\pi_{C(X,A)}$ and $\widehat\pi_B$ are the maps
induced by the canonical coordinate projections
$\pi_{C(X, A)}\colon C(X,A)\oplus_{C(Y, A)}B\to C(X,A)$ and $\pi_B\colon C(X,A)\oplus_{C(Y, A)}B\to B$.

Since the action of $G$ on $X$ is free, the crossed product $C(X, A)\rtimes_{\gamma\otimes \alpha}G$ is
a $C(X/G)$-algebra with fibers isomorphic to $A\otimes \K(L^2(G))$ by
Theorem~5.2 in~\cite{GHS_preparation},
and similarly for $C(X, A)\rtimes_{\gamma|_{C(Y)}\otimes \alpha} G$.
The assumptions on $G$ and $X$ imply that $\dim(Y/G)\leq \dim(X/G)\leq 1$.
In addition, $A\otimes\K(L^2(G))$ is separable, has stable rank stable rank one,
and $K_1(J)=0$ for every closed two-sided
ideal $J$ in $A\otimes\K(L^2(G))$. Hence the conditions of \cite[Theorem 2.6]{AntBosPer_CtsFlds}
and \cite[Theorem 3.3]{AntPerSan_Pullbacks} are satisfied, and the statements in the theorem follow.
\end{proof}

We will now use the above theorem to compute the equivariant Cuntz semigroup of some $C^*$-dynamical systems.
In the following example, for $n\in\N$ we identify $\Z_n$ with the subgroup of $\T$ consisting of the
$n$-th roots of unity.

\begin{eg}
Adopt the notation of \autoref{exa: UHF}. Let $(B, \beta)$ be the $\Z_n$-dynamical system given by
$$B=\{f\in C(\mathbb T, M_{n^{\infty}})\colon f(\omega_n^k)=f(1)\in \mathbb C 1_{M_{n^{\infty}}}\text{ for }k\in\Z_n\},$$
and $\beta=( \gamma\otimes\mu^{\mathbb Z_n})|_B$. Then $\Cu^{\mathbb Z_n}(B, \beta)$ is isomorphic,
in the category $\CCu^{\mathbb{Z}_n}$, to the semigroup
\beqa
\left\{(f,g)\in \mathrm{Lsc}\left(\mathbb T, \mathbb Z_{\ge 0}\left[\frac{1}{n}\right]\sqcup \overline{\mathbb R_{>0}}\right)\oplus C(\widehat{\Z_n}, \overline{\mathbb Z_{\ge 0}})\colon
\begin{aligned}
&f(\omega_n z)=f(z)\text{ for } z\in \mathbb T\\
&\ \text{and } \ f(1)=\sum_{\eta\in\widehat{\Z_n}} g(\eta)
\end{aligned}\right\},\eeqa
where for $\tau\in \widehat{\Z_n}$, the corresponding semimodule structure is given by
\[\left[\tau\cdot (f,g)\right](z,\eta)=(f(z),g(\tau+\eta))\]
for all $z\in \T$ and for all $\eta\in\widehat{\Z_n}$.
\end{eg}
\begin{proof}
Set $Y=\Z_n\subseteq \T$. Then $(B, \beta)$ is given by the following pullback diagram:
\begin{align*}
\xymatrix{(B, \beta)\ar[rr]^{}\ar[d]_{} && (\mathbb C, \mathrm{id}_{\mathbb C}) \ar[d]^{\psi}\\ (C(\mathbb T, M_{n^\infty}), \gamma\otimes \mu^{\mathbb Z_n})\ar[rr]_-{\phi} && (C(Y, M_{n^\infty}), \gamma|_{C(Y)}\otimes \mu^{\mathbb Z_n}),}
\end{align*}
where $\phi$ is the restriction map to $Y$, and $\psi$ is the map given by
$\phi(z)(y)=z 1_{M_{n^\infty}}$ for all $z\in \mathbb C$ and for all $y\in Y$.
Recall that $M_{n^{\infty}}$ is simple and its $K_1$-group is trivial.
By \autoref{thm: Lie}, and using that the action of $\Z_n$ on $\T$ induced by $\gamma$ is free,
we get the following pullback diagram in the category $\mathbf{Cu}^{\mathbb Z_n}$:
\begin{align*}
\xymatrix{\Cu^{\mathbb Z_n}(B, \beta)\ar[rr]^{}\ar[d] &&\Cu^{\mathbb Z_n} (\mathbb C, \mathrm{id}_{\mathbb C}) \ar[d]^-{\Cu^{\mathbb Z_n}(\psi)}\\
\Cu^{\mathbb Z_n}(C(\mathbb T, M_{n^\infty}), \gamma\otimes \mu^{\mathbb Z_n})\ar[rr]_-{\Cu^{\mathbb Z_n}(\phi)} && \Cu^{\mathbb Z_n}(C(Y, M_{n^\infty}), \gamma|_{C(Y)}\otimes \mu^{\mathbb Z_n}).}
\end{align*}

The isomorphism in the statement of the example now follows using the definition of pullbacks,
the computations given in \autoref{exa: UHF} and Proposition~5.15 in~\cite{GHS_preparation}, and that the map
\[\Cu^{\mathbb Z_n}(\psi)\colon C(\widehat{\Z_n},\overline{\mathbb Z_{\ge 0}})\to \mathbb Z_{\ge 0}\left[\frac 1 n\right]\sqcup \overline{\mathbb R_{>0}}\]
is given by $\Cu^{\mathbb Z_n}(\psi)(f)=\sum\limits_{j\in\Z_n}f(k)$ for all
$f\in C(\widehat{\Z_n},\overline{\mathbb Z_{\ge 0}})$.  The semimodule structure described in
the example can be computed using the following facts:
\be
\item the semimodule structure
on $\Cu^{\mathbb Z_n}(C(\mathbb T, M_{n^\infty}), \gamma\otimes \mu^{\mathbb Z_n})$ is trivial since $\gamma\otimes \mu^{\mathbb Z_n}$ has the Rokhlin property; and
\item given $\tau\in \widehat{\Z_n}$, the corresponding action on
\[\Cu^{\mathbb Z_n} (\mathbb C, \mathrm{id}_{\mathbb C})\cong C(\widehat{\Z_n},\overline{\mathbb Z_{\ge 0}})\]
is given by $(\tau\cdot g)(\eta)=g(\tau+\eta)$ for all $g\in C(\widehat{\Z_n},\overline{\mathbb Z_{\ge 0}})$ and
for all $\eta\in\widehat{\Z_n}$.
\ee
This finishes the proof.
\end{proof}

The proof of the following example is similar to that of the previous example and so we will omit it.

\begin{eg}
Adopt the notation of \autoref{exa: W}. Let $(B, \beta)$ be the $\Z_n$-dynamical system given by
\[B=\{f\in C(\mathbb T, \mathcal W)\colon f(k)=f(1)\in \mathcal W^\alpha\text{ for } k\in\Z_n\subseteq\T\},\]
and $\beta=( \gamma\otimes\alpha)|_B$. Then $\Cu^{\mathbb Z_n}(B, \beta)$ is $\mathbf{Cu}^{\mathbb{Z}_n}$-isomorphic to the semigroup
\beqa
\left\{(f,g)\in \mathrm{Lsc}\left(\mathbb T, \overline{\mathbb R_{\geq 0}}\right)\oplus C(\widehat{\Z_n}, \overline{\mathbb R_{\geq 0}})\colon
\begin{aligned}
&f(\omega_n z)=f(z)\text{ for } z\in \mathbb T\\
&\ \text{and } \ f(1)=\sum_{\eta\in\widehat{\Z_n}} g(\eta)
\end{aligned}\right\},
\eeqa
where for $\tau\in \widehat{\Z_n}$, the corresponding semimodule structure is given by
\[\left[\tau\cdot (f,g)\right](z,\eta)=(f(z),g(\tau+\eta))\]
for all $z\in \T$ and for all $\eta\in\widehat{\Z_n}$.
\end{eg}

\section{A characterization of freeness using the equivariant Cuntz semigroup}

In this section, we give an application of the equivariant Cuntz semigroup in the context
of free actions of locally compact spaces, which resembles Atiyah-Segal's characterization
of freeness using equivariant $K$-theory; see \cite{AtiSeg_completion}.
Indeed, in \autoref{thm:freeness}, we characterize freeness of a compact Lie group action
on a commutative $C^*$-algebra in terms of a certain canonical map to the equivariant Cuntz
semigroup. We define this map, for arbitrary \ca s, below.

\begin{df}\label{df:NatMapAtiSeg}
Let $\alpha\colon G\to\Aut(A)$
be a continuous action of a compact group $G$ on a \ca\ $A$.
We define a natural $\CCu$-map $\phi\colon \Cu(A^G)\to \Cu^G(A,\alpha)$ as follows. Given a positive
element
$a\in \K(\ell^2(\N))\otimes A^G$, regard it as an element in $(\K(\ell^2(\N))\otimes A)^G$ by
giving $\ell^2(\N)$ the trivial $G$-representation, and set $\phi([a])=[a]_G$.\end{df}

\begin{rem}
Here is an alternative description of $\phi$. Let $\iota\colon A^G\to A$ be the canonical inclusion.
Since $\iota$ is equivariant, it induces a $\CCu^G$-morphism
\[\Cu(\iota)\colon \Cu^G(A^G) \to \Cu^G(A,\alpha)\]
between the equivariant Cuntz semigroups. Now, by \autoref{prop:EgTrivialAction},
there exists a natural $\CCu^G$-isomorphism
$\Cu^G(A^G)\cong \Cu(G)\otimes \Cu(A^G)$. Then $\phi$ is the restriction of $\Cu(\iota)$ to the
second tensor factor.
\end{rem}

We need a proposition first, which is interesting in its own right. Let $\alpha\colon G\to\Aut(A)$
be a continuous action of a compact group $G$ on a \ca\ $A$, and let $a\in A^G$. We denote by
$c_a\colon G\to A$ the continuous function with constant value equal to $a$. Note that $c_a$
belongs to $L^1(G,A,\alpha)$, and the assignment $a\mapsto c_a$ defines a $\ast$-homomorphism
$c\colon A^G\to L^1(G,A,\alpha)$. (Recall that the product in $L^1(G,A,\alpha)$ is given by twisted
convolution.)

\begin{prop}\label{prop:CanMapCorner}
Let $\alpha\colon G\to\Aut(A)$
be a continuous action of a compact group $G$ on a \ca\ $A$.
Denote by $\sigma\colon \Cu^G(A,\alpha)\to \Cu(A\rtimes_\alpha G)$ the canonical $\CCu$-isomorphism
constructed in \autoref{Thm:JulgCuG}. Then there exists a commutative diagram
\begin{align*}
\xymatrix{
\Cu(A^G)\ar[d]_-{\Cu(c)}\ar[r]^-{\phi}& \Cu^G(A,\alpha)\ar[dl]^-{\sigma}\\
\Cu(A\rtimes_\alpha G).&
}
\end{align*}
\end{prop}
\begin{proof}
Abbreviate $\K(\ell^2(\N))$ (with the trivial $G$-action) to $\K$.
By the construction of the map $\sigma$, we need to show that the following diagram is commutative:
\begin{align*}
\xymatrix{
\Cu(A^G)\ar[d]_-{\Cu(c)}\ar[r]^-{\phi}& \Cu^G(A,\alpha)\ar[r]^{\chi} &\Cu(\K(\Hi_A)^G)
\ar[d]^-{\Cu(\theta)} \\
\Cu(A\rtimes_\alpha G)& \Cu((\K(L^2(G))\otimes A)^G)\ar[l]^-{\Cu(\psi)}&
\Cu((\K\otimes \K(L^2(G))\otimes A)^G)\ar[l]^-{\kappa}.
}
\end{align*}

For an irreducible representation $(\Hi_\pi,\pi)$ of $G$, set $d_\pi=\dim(\Hi_\pi)$.
Write
\[\Hi_\C= \ell^2(\N)\otimes \left(\bigoplus_{[\pi]\in \widehat{G}} \left(\bigoplus_{j=1}^{d_\pi}\Hi_\pi\right)\right).\]
Let $1_G\colon G\to \U(\C)$ denote the trivial representation, and let
\[W\colon \ell^2(\N)\cong \ell^2(\N)\otimes \Hi_{1_G} \hookrightarrow \Hi_\C\]
be the isometry corresponding to the canonical inclusion ($\Hi_{1_G}$ is just $\C$).
Write $V\colon \ell^2(\N)\otimes A\to \Hi_A$ for $W\otimes \id_A$.

Let a positive element $a\in \K\otimes A^G$ be given. Then $(\chi\circ\phi)([a])$ corresponds to the
class of $[VaV^*]$ in $\Cu(\K(\Hi_A)^G)$. Denote by $e\colon L^2(G)\to L^2(G)$ the projection onto the
constant functions.
With the presentation of $\Hi_\C$ used above, it is clear
that $\Cu(\theta)$ maps $[VaV^*]$ to the class of
\[a\otimes e\in (\K\otimes A\otimes\K(L^2(G)))^G\cong (\K\otimes\K(L^2(G))\otimes A)^G.\]

Since $\kappa$ is induced by the embedding $(\K(L^2(G))\otimes A)^G\to \K\otimes(\K(L^2(G))\otimes A)^G$
as the upper left corner, and since $\psi$ is equivariant (see \autoref{prop:FixPtCoaction}), it is now
not difficult to check that $\sigma(\phi([a]))$ agrees with $[c_a]$ in $\Cu(A\rtimes_\alpha G)$.
\end{proof}

We recall a version of the Atiyah-Segal completion theorem that is convenient for our purposes.
If a compact group $G$ acts on a compact Hausdorff space $X$, then there exists a canonical
map $K^*(X/G)\to K_G^*(X)$ obtained by regarding a vector bundle on $X/G$ as a $G$-vector
bundle on $X$ (using the trivial action).

For a compact group $G$, we denote by $I_G$ the augmentation ideal in $R(G)$. That is,
$I_G$ is the kernel of the dimension map $R(G)\to\mathbb{Z}$.

\begin{thm}\label{thm:A-S} (Atiyah-Segal).
Let $X$ be a compact Hausdorff space and let a compact Lie group $G$ act on $X$.
The the following statements are equivalent:
\begin{enumerate}
\item  The action of $G$ on $X$ is free.
\item  The natural map $K^*(X/G)\to K_G^*(X)$ is an isomorphism.
\item  The natural map $K^0(X/G)\to K_G^0(X)$ is an isomorphism.
\end{enumerate}\end{thm}
\begin{proof}
That (1) implies (2) is proved in Proposition~2.1 in~\cite{Seg_EqKThy}.
That (2) implies (3) is obvious. Let us show that (3) implies (1), so assume that the
natural map $K^0(X/G)\to K_G^0(X)$ is an isomorphism.

An inspection of the proof of the implication (4) $\Rightarrow$ (1) in Proposition~4.3
of~\cite{AtiSeg_completion} shows that, in our context, there exists $n\in\N$ such that
$I_G^n\cdot K_0^G(X)=0$.
Now, the $R(G)$-module $K_G^\ast(X)=K_G^0(X)\oplus K_G^1(X)$ is in
fact an $R(G)$-algebra, where multiplication is given by tensor product (with diagonal $G$-actions).
In this algebra, the class of the trivial $G$-bundle over $X$ is the unit, so it
belongs to $K_G^0(X)$. In particular, $I_G^n$ annihilates the unit of $K_G^\ast(X)$,
and hence it annihilates all of $K_G^\ast(X)$, that is, $I_G^n\cdot K_\ast^G(X)=0$. In other words, $K_\ast^G(X)$
is discrete in the $I_G$-adic topology. The implication (1) $\Rightarrow$ (4) in Proposition~4.3
of~\cite{AtiSeg_completion} now shows that the $G$-action is free.
\end{proof}

We mention here that the implication (1) $\Rightarrow $ (2) holds even if $G$ is not a Lie group, and
even if $X$ is merely locally compact (this is essentially due to Rieffel; see the proof of
\autoref{thm:freeness} below for a similar argument). However, the equivalence between (2)
and (3) may fail if $X$ is not compact: the trivial action on $\R$ is a counterexample. This can
happen even for free actions.

\vspace{0.3cm}

Recall that a \uca\ $A$ is said to be \emph{finite} if $u\in A$ and $u^*u=1$ imply
$uu^*=1$. A nonunital \ca\ is finite if its unitization is. Finally, a \ca\ $A$ is
\emph{stably finite} if $M_n(A)$ is finite for all $n\in\N$. Commutative \ca s and AF-algebras
are stably finite, as are the tensor products of these, and their subalgebras.

\begin{rem}\label{rem:CptEltsVA}
By Theorem~3.5 in~\cite{BroCiu}, if $A$ is a stably finite \ca, then the set of compact
elements in $\Cu(A)$ can be naturally identified with the Murray-von Neumann semigroup $V(A)$
of $A$.
\end{rem}

In the next theorem, for an abelian semigroup $V$, we denote by $\mathcal{G}(V)$
its Gro\-then\-dieck group. Also, if $\alpha\colon G\to\Aut(A)$ is an action
of a compact group on a \ca\ $A$, we denote by $V^G(A)$ the semigroup of equivariant
Murray-von Neumann equivalence classes of projections in $(A\otimes\K(\Hi_\mu))^G$,
where $\mu\colon G\to\U(\Hi)$ is a finite dimensional unitary representation. With
this notation, the equivariant $K$-theory $K_\ast^G(A)$ of $A$ is, by definition,
the group $\mathcal{G}(V^G(A))$; see Chapter~2 in~\cite{Phi_Book}.

\begin{thm}\label{thm:freeness}
Let $X$ be a locally compact, metric space and let a compact group $G$ act on $X$.
Consider the following statements:
\begin{enumerate}
\item  The action of $G$ on $X$ is free.
\item  The canonical map $\phi\colon \Cu(C_0(X/G))\to \Cu^G(C_0(X))$ is a $\CCu$-isomorphism.
\end{enumerate}
Then (1) implies (2). If $G$ is a Lie group and $X$ is compact, then the
converse is also true.
\end{thm}
\begin{proof}
Assume that the action of $G$ on $X$ is free, and denote by $\alpha\colon G\to \Aut(C_0(X))$
the induced action.
By \autoref{prop:CanMapCorner}, under the identification
$\Cu^G(C_0(X),\alpha)\cong \Cu(C_0(X)\rtimes_\alpha G)$ provided by \autoref{Thm:JulgCuG},
the canonical map $\phi$ becomes the map at the level of the (ordinary) Cuntz semigroup induced by the map
$c\colon C_0(X)^G=C_0(X/G)\to C_0(X)\rtimes_\alpha G$ defined before \autoref{prop:CanMapCorner}.
Denote by $e\in \K(L^2(G))$ the projection onto the constant
functions. By the main theorem in~\cite{Ros_corner}, we have $c(C_0(X)^G)=e(C_0(X)\rtimes_\alpha G)e$.
Since $\Cu$ is a stable functor, it is enough to show that $e$ is a full projection in $C_0(X)\rtimes_\alpha G$.

For $a,b\in C_0(X)$, denote by $f_{a,b}\colon G\to C_0(X)$ the function given by
\[f_{a,b}(g)(x)=a(x)b(g^{-1}\cdot x)\]
for $g\in G$ and $x\in X$.
It is an easy consequence of the Stone-Weierstrass theorem that the set
\[\left\{f_{a,b}\colon a,b\in C_0(X)\right\}\]
has dense linear span in $C_0(X)\rtimes G$.

Denote by $I$ the ideal in $C_0(X)\rtimes G$ generated by $e$.
Let $(a_\lambda)_{\lambda\in \Lambda}$ be an approximate
identity for $C_0(X)$. Upon averaging over $G$, we may assume that $a_\lambda$ belongs to $C_0(X)^G$
for all $\lambda\in\Lambda$. Let $a,b\in C_0(X)$. Then $f_{a,a_\lambda b}=c_{a_\lambda}f_{a, b}$ for
$\lambda\in\Lambda$, and hence
\[f_{a,b}=\lim_{\lambda\in\Lambda}f_{a,a_\lambda b}=\lim_{\lambda\in\Lambda}\left(c_{a_\lambda}f_{a, b}\right),\]
so $f_{a,b}$ belongs to $I$. We conclude that $C_0(X)\rtimes G=I$, as desired.

Assume now that $G$ is a compact Lie group, that $X$ is compact, and that
the canonical map $\Cu(C(X/G))\to \Cu^G(C(X))$ is an isomorphism in $\CCu$.
We claim that the canonical map $K^0(X/G)\to K_G^0(X)$ is an isomorphism.

Under the natural identification given by \autoref{Thm:JulgCuG}, the canonical
inclusion $c\colon C(X/G)\to C(X)\rtimes G$ induces an isomorphism $\Cu(c)$ at the level of the Cuntz semigroup
(see also \autoref{prop:CanMapCorner}).
The algebra $C(X/G)$ is clearly stably finite. On the other hand, $C(X)\rtimes_\alpha G$ is also stably
finite because it is a subalgebra of the stable finite \ca\ $C(X)\otimes \K(L^2(G))$.
By \autoref{rem:CptEltsVA}, the restriction of the isomorphism $\Cu(c)$ to the compact elements
of $C(X/G)$ yields an isomorphism
\[\psi\colon V(C(X/G))\to V(C(X)\rtimes G)\]
between the respective Murray-von Neumann semigroups of projections.
By taking the Grothendieck construction, one gets an isomorphism
\[\varphi\colon \mathcal{G}(V(C(X/G))) \to \mathcal{G}(V(C(X)\rtimes G))\]
between the respective Grothendieck groups. We want to conclude from this that $\varphi$ induces
an isomorphism between the $K_0$-groups
of these $C^*$-algebras.

Since $C(X/G)$ is unital, we have $\mathcal{G}(V(C(X/G)))=K_0(C(X/G))$. However, $C(X)\rtimes G$ is not
unital unless $G$ is finite, and it is even not clear whether it (or its stabilization) has
an approximate identity consisting of projections. (This would also imply that its $K_0$-group
is obtained as the Grothendieck group of its Murray-von Neummann semigroup.)

Instead, we appeal to Julg's theorem for equivariant $K$-theory.
Indeed, the proof given in Theorem~2.6.1 in~\cite{Phi_Book}
shows that if $A$ is a
unital $C^*$-algebra, $G$ is a compact group, and $\alpha\colon G\to\Aut(A)$ is a continuous
action, then there exists a canonical isomorphism of semigroups
\[V^G(A)\cong V(A\rtimes G).\]
Since $K^G_0(A)$ is the Grothendieck group of $V^G(A)$, it follows that the same is true for
$K_0(A\rtimes_\alpha G)$. In our context, this shows that $\varphi$ induces an isomorphism
\[\theta\colon K^0(X/G)\to K^G_0(C(X))\cong K^0_G(X).\]

Now, the implication (3) $\Rightarrow$ (1) in \autoref{thm:A-S} shows that the action
of $G$ on $X$ is free.
\end{proof}

\section{Classification of actions using \texorpdfstring{$\Cu^G$}{CuG}}
In this section, we classify a class of actions of finite abelian groups on
certain stably finite \ca s using the equivariant Cuntz semigroup; see \autoref{thm:ConjActsCu}.
We describe the general strategy first.

Let $G$ be a finite abelian group, and let $\alpha$ and
$\beta$ be actions of $G$ on \ca s $A$ and $B$. A $\CCu^G$-homomorphism $\rho\colon \Cu^G(A,\alpha)\to
\Cu^G(B,\beta)$ can be regarded, via \autoref{thm:JulgWithCuG}, as a $\CCu$-homomorphism
$\Cu(A\rtimes_\alpha G)\to \Cu(B\rtimes_\beta G)$ that is equivariant with respect to the
dual actions $\widehat{\alpha}$ and $\widehat{\beta}$. If $A\rtimes_\alpha G$ and $B\rtimes_\beta G$
belong to a class of \ca s for which the Cuntz semigroup classifies homomorphisms (up to approximate
unitary equivalence), then we can
obtain a $\ast$-homomorphism $\varphi\colon A\rtimes_\alpha G\to B\rtimes_\beta G$ such that
$\varphi\circ\widehat{\alpha}_\tau$ is approximately unitarily equivalent to $\widehat{\beta}_\tau\circ\varphi$
for all $\tau\in \widehat{G}$. When the dual actions have the Rokhlin property, results from
\cite{GarSan_RokConstrI} can be used to replace $\varphi$ with an approximately unitarily equivalent
$\ast$-homomorphism $\theta$ which is actually equivariant. If $\theta$ moreover satisfies a scale
condition (which can be phrased in terms of $\rho$), one can essentially restrict $\psi$ (using
\autoref{thm:ConjActsCu}) to an equivariant homomorphism $\psi\colon A\to B$ which satisfies
$\Cu^G(\psi)=\rho$. Finally, when $\rho$ is an isomorphism, we can choose $\theta$ and $\psi$ to
be $\ast$-isomorphisms.

\vspace{0.3cm}

The following result will allow us to go from equivariant $\ast$-homomorphisms between double crossed
products to equivariant $\ast$-homomorphisms between the original dynamical systems.
Recall that $\precsim$ stands for Cuntz subequivalence (see the beginning of Section~2.1).

\begin{thm}\label{thm:ConjActsCu}
Let $G$ be a finite group, let $A$ and $B$ be \ca s such that $B$ has stable rank one,
and let $\alpha\colon G\to\Aut(A)$ and $\beta\colon G\to \Aut(B)$ be actions.
Denote by $\lambda\colon G\to \U(\ell^2(G))$ the left regular representation,
and let $e\in\K(\ell^2(G))$ be the projection onto the constant functions on $G$.
Let $s_A$ and $s_B$ be $G$-invariant strictly positive elements in $A$ and $B$,
respectively.
\be\item
Every equivariant $\ast$-homomorphism
\[\varphi\colon (A\otimes\K(\ell^2(G)),\alpha\otimes\Ad(\lambda))\to
(B\otimes\K(\ell^2(G)),\beta\otimes\Ad(\lambda))\]
satisfying
\[\varphi(s_A\otimes e)\precsim s_B\otimes e \ \mbox{ in } \ (B\otimes\K(\ell^2(G)))^{\beta\otimes\Ad(\lambda)}\]
induces an equivariant $\ast$-homomorphism $\psi\colon (A,\alpha)\to (B,\beta)$, and conversely.
\item The actions $\alpha$ and $\beta$ are conjugate if and only if there exists an equivariant
isomorphism
\[\varphi\colon (A\otimes\K(\ell^2(G)),\alpha\otimes\Ad(\lambda))\to
(B\otimes\K(\ell^2(G)),\beta\otimes\Ad(\lambda))\]
such that $\varphi(s_A\otimes e)$ is Cuntz equivalent to $s_B\otimes e$ in $(B\otimes\K(\ell^2(G)))^{\beta\otimes\Ad(\lambda)}$.
\ee
\end{thm}
\begin{proof}
(1). If $\psi\colon A\to B$ is an equivariant $\ast$-homomorphism, then
\[\varphi=\psi\otimes\id_{\K(\ell^2(G))}\colon A\otimes\K(\ell^2(G))\to B\otimes\K(\ell^2(G))\]
is also equivariant and satisfies
\[\varphi(s_A\otimes e)\precsim s_B\otimes e \ \mbox{ in } \ (B\otimes\K(\ell^2(G)))^{\beta\otimes\Ad(\lambda)}.\]

Conversely, let $\varphi$ be an equivariant $\ast$-homomorphism as in the statement. Using that $B$ has
stable rank one together with Theorem~3 in~\cite{CowEllIva},
choose $x\in (B\otimes\K(\ell^2(G)))^{\beta\otimes\Ad(\lambda)}$ such that
$x^*x=\varphi(s_A\otimes e)$ and $xx^*\in \mathrm{Her}(s_B\otimes e)$. Let $x=v|x|$ be the polar decomposition
of $x$ in the bidual of $(B\otimes\K(\ell^2(G)))^{\beta\otimes\Ad(\lambda)}$. Then
conjugation by $v$ defines a $\ast$-isomorphism
\[\Ad(v)\colon \mathrm{Her}(\varphi(s_A\otimes e))\to \mathrm{Her}(xx^*).\]
Using that $v$ is fixed by the extension of $\beta\otimes\Ad(\lambda)$ to the bidual of
$B\otimes\K(\ell^2(G))$, we get
\[\Ad(v)\circ (\beta\otimes\Ad(\lambda))\circ\Ad(v^*)|_{\mathrm{Her}(\varphi(s_A\otimes e))}
=\Ad(vv^*)\circ (\beta\otimes\Ad(\lambda))|_{\mathrm{Her}(\varphi(s_A\otimes e))}.\]
Now, since $vv^*$ is a unit for $\mathrm{Her}(xx^*)$, the above composition equals
$\beta\otimes\Ad(\lambda)$ on $\mathrm{Her}(\varphi(s_A\otimes e))$. It follows that
$\Ad(v)$ is equivariant with respect to $\beta\otimes\Ad(\lambda)$.
Define $\psi\colon A\to B$ to be the following composition:
\[\xymatrix{A\cong \mathrm{Her}(s_A\otimes e) \ar[r]^-{\varphi} &
\mathrm{Her}(\varphi(s_A\otimes e)) \ar[rr]^-{\Ad(v)} & &
\mathrm{Her}(s_B\otimes e)\cong B .}\]
Then $\psi$ is easily seen to be equivariant, and this finishes the proof.

(2). Use Proposition~2.5 in~\cite{CiuEllSan_TypeI} to choose
$x\in (B\otimes\K(\ell^2(G)))^{\beta\otimes\Ad(\lambda)}$ with $x^*x=\varphi(s_A\otimes e)$ and
$\mathrm{Her}(xx^*)=\mathrm{Her}(s_B\otimes e)$. Keeping the notation from the previous part,
it follows that $\Ad(v)$ is an (equivariant)
isomorphism between $\mathrm{Her}(\varphi(s_A\otimes e))$ and $\mathrm{Her}(s_B\otimes e)$.
We conclude that $\psi=\Ad(v)\circ\varphi$ determines an equivariant isomorphism $A\to B$,
as desired. \end{proof}

Next, we introduce the class of actions we will focus on.

\begin{df}\label{df: loc representable}
An action $\alpha$ of a finite group $G$ on a
$C^*$-algebra $A$ is said to be \emph{locally representable} if there exist an
increasing sequence $(A_n)_{n\in\N}$ of subalgebras of $A$ such
that $A_n$ is $\alpha$-invariant for all $n\in \N$ and such that
$A=\overline{\bigcup\limits_{n\in\N}A_n}$, and unitary representations
$u^{(n)}\colon G\to \U(M(A_n))$ of $G$, for $n\in\N$, such that such that
$\alpha_g(a)=\Ad(u^{(n)}_g)(a)$ for all $g\in G$, for all $a\in A_n$,
and for all $n\in \N$.

Let $\mathbf{B}$ be a class of $C^*$-algebras and let $A$ be a $C^*$-algebra in
$\mathbf{B}$. If $\alpha$ is a locally representable action of $G$ on
$A$ such that the subalgebras $A_n$ as
above can be chosen to belong to $\mathbf{B}$, then we
say that $\alpha$ is \emph{locally representable in $\mathbf{B}$}.
\end{df}

Actions that are locally representable in the class of AF-algebras were
studied and classified by Handelman and Rossmann in \cite{HanRos_CptAF}.
The invariant they used is easily seen to be equivalent to the equivariant
$K$-theory of the actions. In \autoref{thm: equiv morphisms of approx rep},
we use the equivariant Cuntz semigroup to classify actions that are locally
representable actions in a class of stably finite algebras containing all
AI-algebras.

We will need the following easy preservation result.

\begin{lma}\label{lem: CrPrdLocRep} Let $\mathbf{B}$ be a class of $C^*$-algebras that
is closed under countable direct limits and under direct sums. Let $A$ be a
$C^*$-algebra in $\mathbf{B}$,
let $G$ be a finite group and let $\alpha$ be an action of $G$ on $A$. Assume that
$\alpha$ is locally representable in $\mathbf{B}$. Then $A\rtimes_\alpha G$ belongs
to $\mathbf{B}$.
\end{lma}
\begin{proof}
Choose an
increasing sequence $(A_n)_{n\in\N}$ of subalgebras of $A$ that belong to
$\mathbf{B}$ with $A=\varinjlim A_n$, and unitary representations
$u^{(n)}\colon G\to \U(M(A_n))$ of $G$
such that such that $\alpha_g(a)=\Ad(u^{(n)}_g)(a)$ for all $g$ in $G$, for all $a$
in $A_n$, and for all $n$ in $\N$. Using continuity of the crossed product functor
in the first step and
the fact that inner actions are cocycle equivalent to the trivial action in the
second step, we conclude that
$$A\rtimes_\alpha G \cong \varinjlim A_n\rtimes_{\Ad(u^{(n)})} G\cong
\varinjlim (A_n\otimes C^*(G))\cong \varinjlim A_n\oplus \cdots \oplus A_n.$$
Now, the assumptions on the class $\mathbf{B}$ imply that $A\rtimes_\alpha G$ belongs
to $\mathbf{B}$.\end{proof}


The following is the main result of this section. For a finite group $G$, we
denote by $e\in \K(\ell^2(G))^G$ the projection onto the constant functions.
If $\alpha\colon G\to \Aut(A)$ is an action and $s_A\in A^G$ is a strictly
positive element in $A$, the element $s_A\otimes e$ belongs to $(A\otimes\K(\ell^2(G)))^G$,
and hence it has a well-defined class $[s_A\otimes e]_G$ in $\Cu^G(A,\alpha)$.

We refer the reader to \cite{Rob_classif} for the definition of 1-dimensional
NCCW-com\-plex\-es, as well as for the classification of certain direct limits of
such \ca s. In the next theorem, $\mathbf{R}$ will denote the class of
unital \ca s that can
be written as inductive limits of one-dimensional NCCW-complexes with trivial $K_1$-groups.
(Unitality of the algebras can be dropped if one instead assumes the existence of an approximate
unit consisting of projections.) Algebras
in $\mathbf{R}$ are classified by their Cuntz semigroup, by the main result in \cite{Rob_classif}.

\begin{thm}\label{thm: equiv morphisms of approx rep}
Let $G$ be a finite abelian group, let $A$ and $B$ be \ca s in $\mathbf{R}$.
Let $\alpha\colon G\to \Aut(A)$ and $\beta\colon G\to\Aut(B)$ be locally representable in $\mathbf{R}$.
Let $s_A$ and $s_B$ be strictly positive $G$-invariant elements in $A$ and
$B$, respectively.
\be
\item
For every $\Cu(G)$-semimodule morphism $\rho\colon \Cu^G(A,\alpha)\to \Cu^G(B,\beta)$
satisfying $\rho([s_A]_G)\leq [s_B]_G$ in $\Cu^G(B,\beta)$, there
exists an equivariant $\ast$-homomorphism $\psi\colon (A,\alpha)\to (B,\beta)$ with
$\Cu^G(\psi)=\rho$. Moreover, $\psi$ is unique up to approximate unitary equivalence with
$G$-invariant unitaries.
\item
The actions $\alpha$ and $\beta$ are conjugate if and only if there exists a $\Cu(G)$-semimodule
isomorphism $\rho\colon \Cu^G(A,\alpha)\to \Cu^G(B,\beta)$
with $\rho([s_A]_G)= [s_B]_G$ in $\Cu^G(B,\beta)$.
\ee\end{thm}
\begin{proof}
(1). By \autoref{lem: CrPrdLocRep}, the crossed product $A\rtimes_\alpha G$
is again inductive limit of one-dimensional NCCW complexes with trivial $K_1$-group.
Moreover, the dual action $\widehat{\beta}$ has the
Rokhlin property by part~(ii) of Proposition~4.4 in~\cite{Naw_RpNonunital}.
Let $\rho\colon \Cu^G(A,\alpha)\to \Cu^G(B,\beta)$ be a $\Cu(G)$-semimodule homomorphism
as in the statement. Using \autoref{thm:JulgWithCuG} and
\autoref{prop: semimodule structure when G is abelian}, $\rho$ can be regarded as an equivariant morphism
\[\tilde\rho\colon (\Cu(A\rtimes_\alpha G), \Cu(\widehat{\alpha}))\to(\Cu(B\rtimes_\beta G), \Cu(\widehat{\beta})).\]
Since $\rho([s_A]_G)\leq [s_B]_G$, we use part~(i) in Theorem~3.13 of~\cite{GarSan_RokConstrI} to deduce that
the morphism $\rho$ lifts to a $\widehat{G}$-equivariant $\ast$-homomorphism
\[\theta\colon (A\rtimes_\alpha G, \widehat{\alpha})\to (B\rtimes_\beta G, \widehat{\beta})\]
that satisfies $\Cu(\theta)=\tilde\rho$. (Although Theorem~3.13 of~\cite{GarSan_RokConstrI} is
stated for the functor $\Cu^\sim$, the assumptions on $A$ imply that we can replace $\Cu^\sim$ with $\Cu$.)

Applying the crossed product functor, we obtain an equivariant $\ast$-homomorphism
\[\widehat{\theta}\colon (A\rtimes_\alpha G)\rtimes_{\widehat{\alpha}}\widehat{G}\to (B\rtimes_\beta G)\rtimes_{\widehat{\beta}}\widehat{G},\]
which, by Takai duality, is equivalent to an equivariant $\ast$-homomorphism
$\varphi\colon A\otimes\K(\ell^2(G))\to B\otimes\K(\ell^2(G))$.

Observe that since $e\in \K(\ell^2(G))$ is a rank-one projection,
the element $s_A\otimes e \in (A\otimes\K(\ell^2(G)))^G$ represents
the same class in $\Cu^G(A,\alpha)$ as $s_A$ (see \autoref{cor:StabUnirep}), and similarly for $s_B\otimes e$. It follows that
$\rho([s_A\otimes e]_G)\leq [s_B\otimes e]_G$.
Under the canonical identification of $B\rtimes_\beta G$ with
the fixed point algebra of $B\otimes\K(\ell^2(G))$,
we then get
\[\theta(s_A\otimes e)\precsim s_B\otimes e \ \mbox{ in } \ \ (B\otimes\K(\ell^2(G)))^{\beta\otimes\Ad(\lambda)}.\]
Part~(1) in~\autoref{thm:ConjActsCu} shows that there exists an equivariant $\ast$-homomorphism
$\varphi\colon (A,\alpha)\to (B,\beta)$ which induces $\psi$, and clearly $\Cu^G(\varphi)=\rho$, as desired.

The uniqueness statement follows from part (ii) of Theorem~3.13 in~\cite{GarSan_RokConstrI}.

(2). The proof of this part is similar to the proof of the first part.
Let $\rho\colon \Cu^G(A,\alpha)\to \Cu^G(B,\beta)$ be a $\Cu(G)$-semimodule homomorphism
as in the statement. Using \autoref{thm:JulgWithCuG} and
\autoref{prop: semimodule structure when G is abelian}, $\rho$ can be regarded as an equivariant morphism
\[\tilde\rho\colon (\Cu(A\rtimes_\alpha G), \Cu(\widehat{\alpha}))\to(\Cu(B\rtimes_\beta G), \Cu(\widehat{\beta})).\]
Since $\rho([s_A]_G)= [s_B]_G$, we use part~(i) in Theorem~3.14 of~\cite{GarSan_RokConstrI} to deduce that
the morphism $\rho$ lifts to a $\widehat{G}$-equivariant $\ast$-isomorphism
\[\theta\colon (A\rtimes_\alpha G, \widehat{\alpha})\to (B\rtimes_\beta G, \widehat{\beta})\]
that satisfies $\Cu(\theta)=\tilde\rho$.
We obtain an equivariant $\ast$-isomorphism
\[\widehat{\theta}\colon (A\rtimes_\alpha G)\rtimes_{\widehat{\alpha}}\widehat{G}\to (B\rtimes_\beta G)\rtimes_{\widehat{\beta}}\widehat{G},\]
which, by Takai duality, is equivalent to an equivariant $\ast$-isomorphism
$\varphi\colon A\otimes\K(\ell^2(G))\to B\otimes\K(\ell^2(G))$.

It follows that $\theta(s_A\otimes e)\sim s_B\otimes e$ in $(B\otimes\K(\ell^2(G)))^{\beta\otimes\Ad(\lambda)}$.
Part~(2) in~\autoref{thm:ConjActsCu} shows that there exists an equivariant $\ast$-isomorphism
$\varphi\colon (A,\alpha)\to (B,\beta)$ which induces $\psi$, and clearly $\Cu^G(\varphi)=\rho$, as desired.
\end{proof}

The assumptions in the above theorem can be relaxed to obtain more general conclusions:
\begin{enumerate}
\item For the conclusion in (1), one only needs to assume that $B$ is separable, that $B\rtimes_\beta G$ has stable rank one and
that $\widehat{\beta}$ has the Rokhlin property. The proof is in fact identical in this case.
\item If the conditions on $\rho([s_A\otimes e]_G)$ and
$[s_B\otimes e]_G$ are omitted, then one can produce a $\beta$-cocycle $\omega\colon G\to \U(M(B))$ and an
equivariant $\ast$-homomorphism $\varphi\colon (A,\alpha)\to (B,\beta^\omega)$.
\item If the conditions
on $\rho([s_A]_G)$ and $[s_B]_G$ are omitted, one has to replace the dynamical system $(A,\alpha)$ with
$(A\otimes\K(\ell^2(\N)),\alpha\otimes\id_{\K(\ell^2(\N))})$, and similarly with $(B,\beta)$.
\end{enumerate}


\providecommand{\bysame}{\leavevmode\hbox to3em{\hrulefill}\thinspace}
\providecommand{\MR}{\relax\ifhmode\unskip\space\fi MR }
\providecommand{\MRhref}[2]{%
  \href{http://www.ams.org/mathscinet-getitem?mr=#1}{#2}
}
\providecommand{\href}[2]{#2}

\end{document}